\documentclass[12pt,a4paper,reqno,dvipsnames]{amsart}

\usepackage[utf8]{inputenc}
\usepackage[english]{babel}
\usepackage[normalem]{ulem} 

\usepackage{enumerate}
\usepackage{xypic}
\usepackage{amsmath}
\usepackage{amsthm}
\usepackage{amsfonts}
\usepackage{amssymb}
\usepackage{graphicx,color}
\usepackage{geometry}
\usepackage{wrapfig}
\usepackage{setspace}
\usepackage{multicol}
\usepackage{xcolor}
\usepackage{soul} 

\newtheorem{theorem}{Theorem}
\newtheorem{lemma}[theorem]{Lemma}
\newtheorem{proposition}[theorem]{Proposition}
\newtheorem{definition}[theorem]{Definition}
\newtheorem{remark}[theorem]{Remark}
\newtheorem{corollary}[theorem]{Corollary}

\newcommand{\HOM}{\operatorname{Hom}}
\newcommand{\ap}{\bullet}



\title[Morita equivalence and globalization]{Morita equivalence and globalization for partial Hopf actions on nonunital algebras}
\date{\today}

\author[M. \ M. \ Alves]{Marcelo Muniz Alves}
\address{Departamento de Matem\'atica, Centro Polit\'ecnico, Universidade Federal do Paran\'a, Curitiba - PR, Brasil}
\email{marcelomsa@ufpr.br}

\author[T. L. Ferazza]{Tiago Luiz Ferrazza}
\address{Universidade Estadual do Paran\'a  - Campus de Paranagu\'a, Rua Comendador Correa Junior, 117, Paranaguá - PR, Brasil}
\email{tiagolferrazza@gmail.com}

\thanks{This study was financed in part by the Coordena\c{c}\~ao de Aperfeiçoamento de Pessoal de N\'ivel Superior – Brasil (CAPES) – Finance Code 001. The first author was partially supported by Conselho Nacional de Desenvolvimento Cient\'ifico e Tecnol\'ogico - CNPq (project 309469/2019-8). }

\begin{document}

	\maketitle

	\begin{abstract}
		In this work we investigate partial actions of a Hopf algebra $H$ on nonunital algebras and the associated partial smash products. We show that our partial actions correspond to nonunital algebras in the category of partial representations of $H$. The central problem of existence of a globalization for a partial action is studied in detail, and we provide sufficient conditions for the existence (and uniqueness) of a minimal globalization  for associative algebras in general. Extending previous results by Abadie, Dokuchaev, Exel and Simon, we define Morita equivalence for partial Hopf actions, and we show that if two symmetrical partial actions are Morita equivalent then their standard globalizations are also Morita equivalent. Particularizing to the case of a partial action on an algebra with local units, we obtain several strong results on equivalences of categories of modules of partial smash products of algebras and partial smash products of $\Bbbk$-categories. 
		
	\end{abstract}

	\section{Introduction}
	
	A partial action of a Hopf algebra $H$ on a unital algebra $A$ is a weakened version of the well-known concept of $H$-module algebra.
	A main motivation for introducing this subject was the previous development of a theory of partial group actions, which had then recently culminated in the Galois theory for partial group actions obtained by Dokuchaev, Ferrero and Paques in  \cite{ferrero}.  
	
	Partial group actions on algebras were originally introduced by Exel in the area of $C^*$-algebras \cite{Circle.actions}, and later were being investigated from a more purely algebraic point of view. 
	Some important features are the construction of a partial skew group algebra $A \ast G$ associated to a partial action of $G$ on $A$, which is a $G$-graded algebra, which is well-suited for the computation of cohomology invariants \cite{Circle.actions,AAR17};
	the existence, with mild hypotheses, of a globalization or enveloping action of $G$ on an algebra $B$ which contains $A$ as an ideal, with partial action induced by the $G$-action on $B$ \cite{dok}; a Morita context between the skew group algebra $A* G$ and the skew group algebra $B * G$ of its globalization \cite{dok}; 
	a Galois theory for partial actions \cite{ferrero}. 
	
	Partial actions of Hopf algebras on unital algebras
	first appear in \cite{caenepeel}. In this paper, Caenepeel and Jansen show that to a (unital) partial $H$-module algebra $A$ there corresponds a nonunital algebra structure on the tensor product $A \otimes H$, and the idempotent $1_A \# 1_H$ generates a unital ideal called the partial smash product $\underline{A \# H}$. It was shown in \cite{alves1} that there is always an $H$-module algebra $B$ which globalizes $A$, and that there is a Morita context between $\underline{A \# H}$ and $B \# H$.
	There is a Hopf-Galois theory as well for partial actions and coactions  \cite{caenepeel}, and partial $H$-module algebras may be characterized as algebras in an appropriate category, the category of partial representations \cite{alves4}; (see the survey \cite{Dokuchaev.survey}).

	From this point on the theory was extended in the direction of generalizing the objects that act, the ones that are acted upon, or both. In the first case, there are partial actions of weak Hopf algebras on unital algebras \cite{partial.weak.hopf}; in the second case, there are the partial actions of Hopf algebras on   $\Bbbk$-categories \cite{alves2}; and there are partial actions of multiplier Hopf algebras on nondegenerate nonunital algebras, 
	which fall in the third case \cite{ABFFM21,globalization-fonseca-fontes-martini}.

	In the present work we investigate partial actions of Hopf algebras on nonunital algebras in the broadest sense possible, including the degenerate ones, but we also study in depth the case of partial actions on algebras with local units. 
	This approach  was motivated by partial Hopf actions on $\Bbbk$-categories, which were in turn inspired by group actions and Hopf actions on categories \cite{cibils-solotar-galois,anca1}.
	
	We relate partial $H$-module nonunital algebras and nonunital algebras in the category of partial representations of $H$ (Theorem \ref{teo1}). Constructing the partial smash product $\underline{A \# H}$, we relate its left modules with the $(A,H)$-modules, extending a result from \cite{rafael}. We introduce globalizations and we establish 
	sufficient conditions for the existence and uniqueness of minimal globalizations (Thm. \ref{theorem.minimal.globalization} and Thm. \ref{thm.lifts}). It is also proved that 
	there is a strict Morita context between $\underline{A\#H}$ and $B\#H$, where $B$ is a globalization of the partial action of $H$ on $A$.
	
	In \cite{alves2} it was shown that every partial action on a $\Bbbk$-category $\mathcal{C}$ induces a partial action on the ``matrix algebra'' $a(\mathcal{C})$, which is an algebra with local units. Conversely, given an algebra $A$ with a system of local units $S$, one may construct  a  $\Bbbk$-category $\mathcal{C}^S(A)$, and we study partial actions on $A$ that induce partial actions on this category. We prove that the category of the left unital modules of $A$ is equivalent to the category of the left modules of the category $\mathcal{C}^S(A)$. Also, if $\mathcal{C}$ is a $\Bbbk$-category, we proved that the category of the left $\mathcal{C}$-modules and the category of the left unital $a(\mathcal{C})$-modules are equivalent.

	We develop a theory of Morita equivalence of partial Hopf actions, extending the concept and results of Morita equivalence of partial group actions presented in \cite{abadie}. Following the first results of Abadie, Dokuchaev, Exel and Sim\'on, in \cite{abadie}, we define the concept of Morita equivalence of partial $H$-actions, and with this we prove that every symmetric partial action on an idempotent algebra is Morita equivalent to a partial action where the algebra has trivial right (or left) annihilator (Thm. \ref{thm.morita.equivalent.partial.action}).
	In \cite{abadie}, the authors constructed a canonical globalization for a regular partial group action and proved that whenever two regular partial $G$-actions are Morita equivalent, the global actions of its canonical globalizations are also Morita equivalent; we have obtained a similar result for partial $H$-actions (Thm. \ref{thm.morita.equivalence.globalization}).

	Finally, it also holds that if $H$ is a Hopf algebra and $A$ is a partial $H$-module algebra with a system of local units $S$, then the categories of modules over the algebras $\underline{A\# H}$ and $\underline{a(\mathcal{C}^S(A))\#H}$, and over the $\Bbbk$-category $\underline{\mathcal{C}^S(A)\#H}$, are all equivalent (Corollaries \ref{corollary.morita.smash.1} and \ref{corollary.morita.smash.2}).
	
	Throughout this work, all the linear structures will be considered over a field $\Bbbk$; for instance, algebra means $\Bbbk$-algebra. Unless otherwise stated, ``module'' stands for ``left module'' and  ``partial action'' stands for ``left partial action''. 
	
	\section{Partial Hopf Actions}
	
	The intent of the next three subsections is to relate the suggested definition of partial Hopf actions on associative algebras with previous results found in the literature.
	
	\subsection{Partial Hopf Actions and the Smash Product}
	
	In \cite{caenepeel}, Caenepeel and Jansen introduced the concept of a Hopf algebra $H$ acting partially on a unital algebra $A$ as a linear map from $H\otimes A$ to $A$ which satisfies some necessary and sufficient conditions for the respective smash product $A\# H$ to be a unital, associative algebra.
	
	\begin{definition}  [\cite{caenepeel}] 
		\label{def.unital.partial.action} Let $H$ be a Hopf algebra and $A$ an algebra with unit $1_A$. A linear map $\cdot:H\otimes A \longrightarrow A$, $h\otimes a \mapsto h \cdot a$ is called a partial action of $H$ on $A$ if, for all $a,b \in A$, $h,g \in H$,
		\begin{enumerate}[\normalfont(1)]
			\item $1_H\cdot a = a$;
			\item $h\cdot(ab)=\sum(h_{(1)}\cdot a)(h_{(2)}\cdot b)$;
			\item $h\cdot(g\cdot a)=\sum(h_{(1)}\cdot 1_A)(h_{(2)}g\cdot a)$,
		\end{enumerate} 
		where $\Delta(h)=\sum h_{(1)}\otimes h_{(2)}$. In this case we say that $A$ is a partial $H$-module algebra with unity. If the equality
		\begin{enumerate}[\normalfont(1)]\addtocounter{enumi}{3}
			\item $h\cdot(g\cdot a)=\sum(h_{(1)}g\cdot a)(h_{(2)}\cdot 1_A)$
		\end{enumerate}
		also holds, then we say that this partial action is \textit{symmetrical} \emph{\cite{alves2} }.
	\end{definition}
	
	Inspired by \cite{alves2} and, mainly, by \cite{caenepeel}, we introduce the following definition of a partial action on associative algebras in general.
	
	\begin{definition}\label{partial.action}
		Let $A$ be an associative algebra. A linear map $\cdot:H\otimes A \longrightarrow A$, $h\otimes a \longmapsto h \cdot a$ will be called a \textit{partial action} of $H$ on $A$ if, for all $a,b \in A$, $h,k \in H$,  
		\begin{enumerate}[\normalfont(1)]
			\item $1_H\cdot a=a$;
			\item $h\cdot (a(k\cdot b))=\sum(h_{(1)}\cdot a)(h_{(2)}k\cdot b)$.
		\end{enumerate}
		In this case, $A$ will be called a \textit{partial $H$-module algebra}. We will say that the partial action is \textit{symmetrical} if, additionally, 
		\begin{enumerate}[\normalfont(1)]\addtocounter{enumi}{2}
			\item $h\cdot ((k\cdot b)a)=\sum(h_{(1)}k\cdot b)(h_{(2)}\cdot a)$
		\end{enumerate}
		for every $a,b\in A$, $h,k\in H$.
	\end{definition}
	
	In \cite{ABFFM21} the authors considered partial actions of multiplier Hopf algebras on nondegenerate algebras; here we are considering partial actions of Hopf algebras on all kind of associative algebras, introducing additional nondegeneracy properties only when needed. At the end of this section we will show that when $H$ is a Hopf algebra and $A$ is a nondegenerate algebra, then both definitions of partial action coincide.
	
	Recall that the right annihilator of an algebra $A$ is the ideal 
	$$r(A)=\{a\in A\,|\, ba=0,\,\forall b\in A\}.$$ 
	
	Analogously, the left annihilator of $A$ is the ideal 
	$$l(A)=\{b\in A\,|\, ba=0,\,\forall a\in A\}.$$ 
	
	The following lemmas follow closely the equivalence presented in \cite{caenepeel} mentioned in the beginning of this section. For this, we will assume that $H$ is a Hopf algebra, $A$ is an associative algebra with $r(A)=0$, $\cdot: H\otimes A\to A$ is a linear map given by $\cdot(h\otimes a)=h\cdot a$ and $A\#H$ is its associated smash product, which is the vector space $A \otimes H$ endowed with the product 
	\[
	(a \# h)(b \# k) = a (h_{(1)} \cdot b) \# h_{(2)} k.
	\]
	
	As it is customary, to avoid confusion with the tensor algebra $A \otimes H$, we will use the notation  $a \# h$ for the vector $a \otimes h$ as an element of the algebra $A \# H$. 
	
	\begin{lemma} $A\#H$ is an associative algebra if and only if $$h\cdot (a(k\cdot b))=\sum(h_{(1)}\cdot a)(h_{(2)}k\cdot b)$$ for every $h, k\in H$, $a, b\in A$.
	\end{lemma}
	\begin{proof}
		In fact, $A\#H$ is an associative algebra if and only if for every $a,b,c \in A$, $h,k,l\in H$,
		\begin{eqnarray*}
			((c\# l)(a\#h))(b\#k)\!\!\!\!\!\!\!\!\!\!\!\!\!\!\!\!\!\!\!\!\!\!\!\!\!&=&\!\!\!\!\!\!\!\!\!\!\!\!\!\!\!\!\!\!\!\!\!\!\!\!\!(c\#l)((a\#h)(b\# k))\\
			&\Updownarrow &\\
			\sum c(l_{(1)}\cdot a)(l_{(2)}h_{(1)}\cdot b)\# l_{(3)}h_{(2)}k\!\!\!\!\!\!\!\!\!\!\!\!\!\!\!\!\!\!\!\!\!\!\!\!\!&=&\!\!\!\!\!\!\!\!\!\!\!\!\!\!\!\!\!\!\!\!\!\!\!\!\! \sum c(l_{(1)}\cdot(a(h_{(1)}\cdot b)))\# l_{(2)}h_{(2)}k\\
			&\,\,\,\,\,\,\,\,\,\,\,\,\,\,\,\,\,\,\,\,\Downarrow^{I\otimes\varepsilon,\,k=1_H}&\\
			\sum c(l_{(1)}\cdot a)(l_{(2)}h\cdot b)\!\!\!\!\!\!\!\!\!\!\!\!\!\!\!\!\!\!\!\!\!\!\!\!\!&=&\!\!\!\!\!\!\!\!\!\!\!\!\!\!\!\!\!\!\!\!\!\!\!\!\! c(l\cdot(a(h\cdot b)))\\
			&\Updownarrow&\\
			c[\sum (l_{(1)}\cdot a)(l_{(2)}h\cdot b)\!\!\!\!\!\!\!\!\!\!\!\!\!\!\!\!\!\!\!\!\!\!\!\!\!&-&\!\!\!\!\!\!\!\!\!\!\!\!\!\!\!\!\!\!\!\!\!\!\!\!\!(l\cdot(a(h\cdot b)))]= 0.
		\end{eqnarray*}
		Since $r(A)=0$, we have the required equality. Conversely, if $h\cdot (a(k\cdot b))=\sum(h_{(1)}\cdot a)(h_{(2)}k\cdot b)$, clearly $A\#H$ is associative.
	\end{proof}
	
	The next two lemmas are proved similarly, using a``support element'' $c$ as before.
	
	\begin{lemma} $A\# H$ is an $A$-bimodule with structure given by $$b(a\# h)b'=\sum ba(h_{(1)}\cdot b')\# h_{(2)}$$ if and only if $$h\cdot ab=\sum(h_{(1)}\cdot a)(h_{(2)}\cdot b),$$ for every $h\in H$, $a,b\in A$.
	\end{lemma}	
	
	\begin{lemma} The linear map $\iota: A\to A\#H$, $a\mapsto a\# 1_H$, is a right $A$-linear morphism if and only if $1_H\cdot a=a$.
	\end{lemma}
	
	In the previous three lemmas, even if $r(A)\neq 0$, the axioms of a partial $H$-module algebra guarantee that the respective smash product will be an associative algebra, an $A$-bimodule and that the inclusion $\iota :A\to A\# H$ is a right $A$-linear morphism. We only need $r(A)=0$ for the converse.

	\begin{proposition} Let $H$ be a cocommutative Hopf algebra. If $A$ and $B$ are both (symmetrical) partial $H$-module algebras, then $A\otimes B$ is a (symmetrical) partial $H$-module algebra via
		\begin{center}
			$h\cdot (a\otimes b)=\sum h_{(1)}\cdot a\otimes h_{(2)}\cdot b$.
		\end{center}
	\end{proposition}
	
	In \cite{alves2}, this result was obtained (for unital algebras) as a consequence of a similar result for partial Hopf actions on categories.

	\subsection{Partial $H$-module Algebras and Algebras in $_H\mathcal{M}^{par}$}
	
	In \cite{alves4}, Alves, Batista e Vercruysse proved that, when the antipode of $H$ is bijective, there is a bijective correspondence between partial $H$-module algebras with symmetrical partial actions and unital algebras in the category of the partial $H$-modules. We will show that an analogous correspondence still holds for nonunital algebras $A$ that satisfy at least one of the following properties:
	\begin{enumerate}[\normalfont(1)]
		\item $A^2=A$;
		\item $l(A)=0$;
		\item $r(A)=0$.
	\end{enumerate}
	
	In order to do so, we consider the category of $(A,H)$-modules associated to a partial $H$-module $A$ introduced in \cite{rafael}. We will show that, under the same conditions above, the category of $(A,H)$-modules is equivalent to the category of $\underline{A \# H}$-modules, and then we use the fact that the partial $H$-modules are the same as modules over an appropriate partial smash product. 
	
	We begin by recalling the definition of a partial representation of a Hopf algebra from \cite{partial.coreps}, which is a shorter version of the one originally introduced in \cite{alves4}.
	
	\begin{proposition} \cite[Lemma 2.11]{partial.coreps} \label{proposition.equivalent.definitions.partial.reps}
		Let $H$ be a Hopf $\Bbbk$-algebra, $B$ be a unital $\Bbbk$-algebra and $\pi : H \to B$ be a linear map. The following are equivalent: 
		\begin{enumerate}[\normalfont(1)]
			\item $\pi: H \to B$ satisfies
			\begin{enumerate}
				\item[\normalfont(PR1)] $\pi (1_H)  =  1_B$; 
				\item[\normalfont(PR2)] $\pi (h) \pi (k_{(1)}) \pi (S(k_{(2)}))  =   \pi (hk_{(1)}) \pi (S(k_{(2)})) $, for every $h,k\in H$;
				\item[\normalfont(PR3)]$\pi (h_{(1)}) \pi (S(h_{(2)})) \pi (k)  =   \pi (h_{(1)}) \pi (S(h_{(2)})k)$, for every $h,k\in H$.
			\end{enumerate}
			\item $\pi: H \to B$ satisfies
			\begin{enumerate}
				\item[\normalfont(PR1)] $\pi (1_H)  =  1_B$; 
				\item[\normalfont(PR4)] $\pi (h) \pi (S(k_{(1)})) \pi (k_{(2)}) = \pi (hS(k_{(1)})) \pi (k_{(2)})$, for every $h,k\in H$;
				\item[\normalfont(PR5)] $\pi (S(h_{(1)}))\pi (h_{(2)}) \pi (k) = \pi (S(h_{(1)}))\pi (h_{(2)} k)$, for every $h,k\in H$.
			\end{enumerate}
		\end{enumerate}
	\end{proposition}
	
	\begin{definition} 
		Let $H$ be a Hopf $\Bbbk$-algebra, and let $B$ be a unital $\Bbbk$-algebra. A \emph{partial representation} of $H$ in $B$ is a linear map $\pi: H \rightarrow B$ which satisfies the equivalent conditions of Proposition \ref{proposition.equivalent.definitions.partial.reps}.
	\end{definition}

	\begin{remark}[\cite{alves4}] If $H$ is cocommutative, then the items in the definition of a partial representation coalesce into {\normalfont(PR1)}, {\normalfont(PR2)} and {\normalfont(PR5)}.
	\end{remark}
	
	We will show in the following that, under mild conditions, a partial $H$-module algebra $A$ carries a partial representation $\pi : H \to End(A)$ (this is well-known when $A$ is unital).  
	
	First note that if $A$ is any partial $H$-module algebra, then for all $x,y\in A$, $h\in H$, we have that
	\begin{equation}\label{eq1}
		\sum h_{(1)}\cdot S(h_{(2)})\cdot xy= \sum  (h_{(1)}\cdot S(h_{(2)})\cdot x)y,
	\end{equation}
	because
	\begin{eqnarray*}
		\sum h_{(1)}\cdot S(h_{(2)})\cdot xy&=& \sum h_{(1)}\cdot [(S(h_{(3)})\cdot x)(S(h_{(2)})\cdot y)]\\
		&=&\sum  (h_{(1)}\cdot S(h_{(4)})\cdot x)(h_{(2)}S(h_{(3)})\cdot y)\\
		&=&\sum  (h_{(1)}\cdot S(h_{(2)})\cdot x)y.
	\end{eqnarray*}
	
	Equality (\ref{eq1}) means that the linear map $a\mapsto \sum h_{(1)}\cdot S(h_{(2)})\cdot a$ is a right $A$-module map from $A$ to $A$ for every $h\in H$. Analogously, if the partial action is symmetrical, the equality
	\begin{equation}\label{eq2} 
		\sum S(h_{(1)})\cdot h_{(2)}\cdot xy=x(\sum S(h_{(1)})\cdot h_{(2)}\cdot y)
	\end{equation}
	means that the linear map $a\mapsto \sum S(h_{(1)})\cdot h_{(2)}\cdot a$ is a left $A$-module map from $A$ to $A$ for every $h\in H$.
	
	\begin{lemma} Let $A$ be a partial $H$-module algebra with symmetrical partial action and $H$ a Hopf algebra. If $A$ is either idempotent, $l(A)=0$ or $r(A)=0$ and the antipode of $H$ is bijective, then the linear map $\pi: H\to End(A)$, defined by $\pi(h)(a)=h\cdot a$, is a partial representation.
	\end{lemma}
	\begin{proof}
		(PR1) is straightforward: given $a \in A$,  $\pi(1_H)(a)=1_H\cdot a=a$, therefore $\pi(1_H)=id_A$.Let us proceed by proving (PR2) and (PR3) first for an idempotent algebra $A$. 
		
		Since for any element $a\in A$ there exist $b_1,\ldots,b_n,c_1,\ldots,c_n \in A$ such that $a=\sum_{i=1}^nb_ic_i$ and $\pi$ is a linear map, we only need to check the axioms of partial representation for elements of the form $xy\in A$. Let $x,y \in A$ and $h,k\in H$,  and
		\begin{eqnarray*}
			\sum \pi(h)\pi(k_{(1)})\pi(S(k_{(2)}))(xy)&=& \sum h\cdot k_{(1)}\cdot S(k_{(2)})\cdot (xy)\\
			&\stackrel{ \eqref{eq1} }{=}& \sum h\cdot ((k_{(1)}\cdot S(k_{(2)})\cdot x)y)\\
			&=& \sum (h_{(1)}k_{(1)}\cdot S(k_{(2)})\cdot x)(h_{(2)}\cdot y)\\
			&=& \sum (h_{(1)}k_{(1)}\cdot S(k_{(4)})\cdot x)(h_{(2)}k_{(2)}S(k_{(3)})\cdot y)\\
			&=&\sum hk_{(1)}\cdot [(S(k_{(3)})\cdot x)(S(k_{(2)})\cdot y)]\\
			&=& \sum hk_{(1)}\cdot S(k_{(2)})\cdot (xy)\\
			&=& \sum \pi(hk_{(1)})\pi(S(k_{(2)}))(xy).
		\end{eqnarray*}	
		This proves that item (PR2) holds.  For item (PR3), we have that
		\begin{eqnarray*}
			\sum \pi(k_{(1)})\pi(S(k_{(2)}))\pi(h)(xy)&=& \sum k_{(1)} \cdot S(k_{(2)})\cdot h\cdot (xy)\\
			&=& \sum k_{(1)} \cdot S(k_{(2)})\cdot (h_{(1)}\cdot x)(h_{(2)}\cdot y)\\
			&=& \sum k_{(1)} \cdot [(S(k_{(3)})h_{(1)}\cdot x)(S(k_{(2)})\cdot h_{(2)}\cdot y)]\\
			&=& \sum (k_{(1)}\cdot S(k_{(2)})h_{(1)}\cdot x)(h_{(2)}\cdot y)\\
			&=& \sum k_{(1)} \cdot [(S(k_{(3)})h_{(1)}\cdot x)(S(k_{(2)})h_{(2)}\cdot y)]\\
			&=&\sum k_{(1)}\cdot S(k_{(2)})h\cdot (xy)\\
			&=& \sum \pi(k_{(1)})\pi(S(k_{(2)}h))(xy).
		\end{eqnarray*}

		Now we will assume that $A$ is not necessarily idempotent, but that instead we have $l(A)=0$. Then, for every $a,x\in A$, $h,k\in H$, we have also that $\pi(1_H)=id_A$ and
		\begin{eqnarray*}
			\sum \pi(h)\pi(k_{(1)})\pi(S(k_{(2)}))(a)x&=&[h\cdot k_{(1)}\cdot S(k_{(2)})\cdot a]x\\
			&=&\sum h_{(1)}\cdot [(k_{(1)}\cdot S(k_{(2)})\cdot a)(S(h_{(2)})\cdot x)]\\
			&=& \sum (h_{(1)}k_{(1)}\cdot S(k_{(3)})\cdot a)(h_{(2)}\cdot S(h_{(3)})\cdot x)\\
			&=& \sum h_{(1)}k_{(1)}\cdot [(S(k_{(3)})\cdot a)(S(k_{(2)})\cdot S(h_{(2)})\cdot x)]\\
			&=& \sum h_{(1)}k_{(1)}\cdot S(k_{(2)})\cdot (a(S(h_{(2)})\cdot x))\\
			&=&  \sum h_{(1)}k_{(1)}\cdot [(S(k_{(3)})\cdot a)(S(k_{(2)})S(h_{(2)})\cdot x)]\\
			&=& \sum (hk_{(1)}\cdot S(k_{(2)})\cdot  a)x\\
			&=&\sum \pi(hk_{(1)})\pi(S(k_{(2)}))(a)x.
		\end{eqnarray*}
		
		Since $l(A)=0$ and this holds for all $x\in A$, we have that $$\sum \pi(h)\pi(k_{(1)})\pi(S(k_{(2)}))=\sum \pi(hk_{(1)})\pi(S(k_{(2)})).$$ 
		
		For item (PR3), we have that
		\begin{eqnarray*}
			\sum \pi(h_{(1)})\pi(S(h_{(2)}))\pi(k)(a)x&=&\sum [h_{(1)}\cdot S(h_{(2)})\cdot k\cdot a]x\\
			&\stackrel{ \eqref{eq1} }{=}&\sum h_{(1)}\cdot S(h_{(2)})\cdot [(k\cdot a)x]\\
			&=&\sum h_{(1)}\cdot (S(h_{(3)})k\cdot a)(S(h_{(2)})\cdot x)\\
			&=& \sum (h_{(1)}\cdot S(h_{(4)})k\cdot a)(h_{(2)}S(h_{(3)})\cdot x)\\
			&=& \sum (h_{(1)}\cdot S(h_{(2)})k\cdot a)x\\
			&=& \sum \pi(h_{(1)})\pi(S(h_{(2)})k)(a)x.
		\end{eqnarray*}
		Since $l(A)=0$ and this holds for all $x\in A$, we have that 
		$$\pi (h_{(1)}) \pi (S(h_{(2)})) \pi (k)  =   \pi (h_{(1)}) \pi (S(h_{(2)})k).$$	
		
		Finally, assume that $r(A) =0$ and that the antipode is bijective. By recurring to Proposition \ref{proposition.equivalent.definitions.partial.reps},
		it is enough to prove that properties (PR4) and (PR5) hold. 
		For item (PR4),
		\begin{eqnarray*}
			\sum x\pi(h)\pi(S(k_{(1)}))\pi(k_{(2)})(a)&=& \sum x(h\cdot S(k_{(1)})\cdot k_{(2)}\cdot a)\\
			&=& \sum h_{(2)}\cdot [(S^{-1}(h_{(1)})\cdot x)(S(k_{(1)})\cdot k_{(2)}\cdot a)]\\
			&=& \sum (h_{(2)}\cdot S^{-1}(h_{(1)})\cdot x)(h_{(3)}S(k_{(1)})\cdot k_{(2)}\cdot a)\\
			&=& \sum h_{(2)}S(k_{(1)})\cdot [(k_{(2)}\cdot S^{-1}(h_{(1)})\cdot x)(k_{(3)}\cdot a)]\\
			&=&  \sum h_{(2)}S(k_{(1)})\cdot [(k_{(2)}S^{-1}(h_{(1)})\cdot x)(k_{(3)}\cdot a)]\\
			&=&\sum x(hS(k_{(1)})\cdot k_{(2)}\cdot a)\\
			&=&\sum x\pi(hS(k_{(1)}))\pi(k_{(2)})(a),
		\end{eqnarray*}
		since $r(A)=0$ and this holds for all $x\in A$, we have that $$\sum \pi(h)\pi(S(k_{(1)}))\pi(k_{(2)})=\sum \pi(hS(k_{(1)}))\pi(k_{(2)}).$$ The calculations needed to verify item (PR5) are similar, although in this case the bijectivity of the antipode of $H$ is not needed. 
	\end{proof}
	
	\begin{corollary}
		Let $H$ be a cocommutative Hopf algebra and $A$ a partial $H$-module algebra with symmetrical partial action. If $r(A)=0$ or $l(A)=0$, then the linear map $\pi: H\to End(A)$, defined by $\pi(h)(a)=h\cdot a$, is a partial representation.
	\end{corollary}
	
	This corollary shows that if $G$ is a group, $H=\Bbbk G$ and $A$ is a partial $H$-module algebra with symmetrical partial action, then the mapping $\pi(g)(a)=g\cdot a$ defines a partial representation and since
	\[
	\sum \pi(h_{(1)})\pi(S(h_{(2)}))\pi(h_{(3)})= \sum \pi(h_{(1)})\pi(S(h_{(2)}) h_{(3)}) = \pi(h) \pi(1_H) = \pi(h),
	\]
	for every partial representation $\pi:H\to End(A)$, we have that 
	\begin{eqnarray} \label{eq3}
		g\cdot g^{-1}\cdot g\cdot a=g\cdot a,
	\end{eqnarray}
	for every $a\in A$, $g\in G$. This equality will be useful in the next section.
	
	\begin{definition}[\cite{alves4}] Let $H$ be a Hopf algebra. A partial $H$-module is a vector space $M$ with a partial representation $\pi: H \to End(M)$.
	\end{definition}
	Motivated by the relation between $H$-modules and representations of $H$, Alves, Batista and Vercruysse presented in \cite{alves4} an equivalent definition of partial $H$-module. 
	
	\begin{definition} \label{def.partial.H.module}
		Let $H$ be a Hopf algebra. A partial $H$-module is a vector space $M$ with a linear map $\bullet: H\otimes M \to M$  satisfying the following properties:
		\begin{enumerate}
			\item[\normalfont(PM1)] $1_H\bullet m=m$;
			\item[\normalfont(PM2)] $\sum h\bullet (k_{(1)}\bullet (S(k_{(2)})\bullet m))=hk_{(1)}\bullet (S(k_{(2)})\bullet m)$;
			\item[\normalfont(PM3)] $\sum h_{(1)}\bullet (S(h_{(2)})\bullet (k\bullet m))=h_{(1)}\bullet (S(h_{(2)})k\bullet m)$.
		\end{enumerate}    
	\end{definition}
	Of course, as a consequence of Proposition \ref{proposition.equivalent.definitions.partial.reps}, the following equalities also hold for a partial $H$-module: 
	\begin{enumerate}
		\item[\normalfont(PM4)] $\sum h\bullet (S(k_{(1)})\bullet (k_{(2)}\bullet m))=hS(k_{(1)})\bullet (k_{(2)}\bullet m)$;
		\item[\normalfont(PM5)] $\sum S(h_{(1)})\bullet (h_{(2)}\bullet (k\bullet m))=S(h_{(1)})\bullet h_{(2)}\bullet (k\bullet m)$.
	\end{enumerate}
	
	In what follows, by a partial $H$-module we will always mean a vector space endowed with a linear map $\bullet : H \otimes A \to A$ as in Definition \ref{def.partial.H.module}. If $M$ and $N$ are partial $H$-modules, a morphism of partial $H$-modules is a $\Bbbk$-linear map $f: M \to N$ such that $f(h \bullet m) = h \bullet f(m)$ for every $h \in H$ and $m \in M$ (equivalently, $f \pi_M(h)  = \pi_N (h) f$, where $\pi_M : H \to End(M)$ and $\pi_N : H \to End(N)$ are the corresponding partial representations). The category of partial $H$-modules will be denoted by  $_H\mathcal{M}^{par}$. 
	
	The defining equations of a partial representation suggest the following construction.

	\begin{definition}[\cite{alves4}]
		Let $H$ be a Hopf algebra. The algebra $H_{par}$ is the quotient algebra $T(H)/I$, where $T(H)$ is the tensor algebra and $I$ is the ideal generated by the elements:
		\begin{enumerate}[\normalfont(1)]
			\item $1_H - 1_{T(H)}$;
			\item $\sum x\otimes y_{(1)}\otimes S(y_{(2)})- xy_{(1)}\otimes S(y_{(2)})$;
			\item $\sum x\otimes S(y_{(1)})\otimes y_{(2)}- xS(y_{(1)})\otimes y_{(2)}$;
		\end{enumerate}
	\end{definition}
	
	Recall from \cite{alves4} that, when the Hopf algebra $H$ has a bijective antipode, the category of the partial $H$-modules is isomorphic to the category of the $H_{par}$-modules.
	
	In fact, if we denote the class of  $x^1\otimes\cdots\otimes x^n$ in $T(H)/I$ by $[x^1]\cdots[x^n]$, it is clear that the map 
	\[
	t: H \to H_{par}, \ \ h \mapsto [h]
	\]
	is a partial representation, and every partial representation factors through it: if $M$ is a partial $H$-module, the corresponding structure of left $H_{par}$-module is defined by $[h]\cdot m = h \bullet m$. 
	
	Given $h \in H$, let $\epsilon_{h}=[h_{(1)}][S(h_{(2)})]$.
	From \cite{alves4} we know that $H_{par}$ is generated, as a vector space, by elements of the form 
	$$ \epsilon_{k^1}\cdots  \epsilon_{k^n}[h].$$
	If $A \subset H_{par}$ is the subalgebra 
	\begin{eqnarray*}
		A= span \{\epsilon_{h^1}\cdots\epsilon_{h^n};\,\,h^i\in H\},
	\end{eqnarray*}
	then there exists a canonical partial action of $H$ on $A$, given by 
	$$h\cdot a=[h_{(1)}] \ a \ [S(h_{(2)})],$$
	and in \cite{alves4} it is proved that the linear map 
	\[
	\underline{A\# H} \to  H_{par}, \ \ 	(a\# h)1_A \mapsto  a[h]
	\]
	is an isomorphism of algebras.
	
	This last isomorphism provides another equivalence of categories $_H\mathcal{M}^{par} \cong \,_{\underline{A\# H}}\mathcal{M}$ which is given explicitly by the following: if $M$ is a left $\underline{A\# H}$-module, then it is also a partial $H$-module by the rule
	\begin{eqnarray*}
		h\bullet m : =(1_A\# h)\triangleright m,
	\end{eqnarray*}
	and, as we can see in \cite{alves4}, we have that for $a=\epsilon_{h^1}\cdots\epsilon_{h^n}\in A$, $m\in M$, $k\in H$,
	\begin{eqnarray*}
		(a\# k)\triangleright m= \sum h^1_{(1)}\bullet (S(h^1_{(2)})\bullet (\cdots (h^n_{(1)}\bullet(S(h^n_{(2)})\bullet(k\bullet m))\cdots)).
	\end{eqnarray*}
	
	Also in \cite{alves4}, the authors highlight that every partial $H$-module $M$ can be viewed as an $A$-bimodule: there are algebra morphisms 
	\[
	s : A \to \underline{A\# H}, \ \ \ \ s(a) = a \# 1_H, 
	\]
	and 
	\[
	t: A^{op} \to \underline{A\# H}, \ \ \ \ t(\varepsilon_h) = (1_A \# h_{(2)})(1_A \# S^{-1}(h_{(1)})),  
	\]
	which satisfy $s(a) t(b) = t(b) s(a)$ for all $\, b \in A$, and so 
	\[
	am : = s(a)\triangleright m, \ \ \ \ ma : =t(a)\triangleright m
	\]
	define an $A$-bimodule structure on $M$.
	We note that $s$ and $t$ are part of the Hopf algebroid structure on $H_{par}$ (see [\cite{alves4}]). 
	
	\begin{definition}[\cite{rafael}] Let $A$ be a unital partial $H$-module algebra. A vector space $M$ is a (left) partial $(A,H)$-module if $M$ is a left $A$-module together with a linear map $H\otimes M\to M$, $h\otimes m\mapsto hm$ such that:
		\begin{enumerate}[\normalfont(1)]
			\item $1_Hm=m$;
			\item $h(a(km))=\sum (h_{(1)}\cdot a)((h_{(2)}k)m)$,
		\end{enumerate}
		for every $h,k\in H$, $a\in A$, $m\in M$.	
	\end{definition}
	
	Since for every $h\in H$, we have that
	\begin{eqnarray*}
		h\cdot 1_A=\sum [h_{(1)}][S(h_{(2)})]=\epsilon_{h}
	\end{eqnarray*} 
	and, as we can see in \cite{rafael}, every $\underline{A\# H}$-module is a partial $(A,H)$-module, then we have that
	\begin{eqnarray}\label{eq4}
		h\bullet x&=&h\bullet (1_Ax)=\sum \epsilon_{h_{(1)}}(h_{(2)}\bullet x), \end{eqnarray}
	and \begin{eqnarray}\label{eq5}
		h\bullet x&=&h\bullet (x1_A)=\sum (h_{(1)}\bullet x)\epsilon_{h_{(2)}}.
	\end{eqnarray}
	
	Now we will present a definition of partial $(A,H)$-modules for when $A$ does not have a unit.
	
	\begin{definition} \label{def.partial(A,H)-module} Let $A$ be an associative partial $H$-module algebra. A vector space $M$ is a (left) partial $(A,H)$-module if $M$ is an unital $A$-module together with a linear map $H\otimes M\to M$, $h\otimes m\mapsto hm$ such that:
		\begin{enumerate}[\normalfont(1)]
			\item $1_Hm=m$;
			\item $h(a(km))=\sum (h_{(1)}\cdot a)((h_{(2)}k)m)$,
		\end{enumerate}
		for every $h,k\in H$, $a\in A$, $m\in M$.	
	\end{definition}
	
	Recall that if $A$ is a partial $H$-module algebra with unit, then the smash product $A\#H$ is the algebra defined by: $A\#H=A\otimes H$ as vector space, and the product is given by 
	\begin{eqnarray*}
		(a\otimes h)(b\otimes k)=\sum a(h_{(1)}\cdot b)\otimes h_{(2)}k.
	\end{eqnarray*}
	
	This structure is defined in \cite{caenepeel}, where Caenepeel and Jansen noticed that $A\# H$ may not have unit, but the subalgebra $\underline{A\#H}=(A\# H)(1_A\# 1_H)$, that is called the partial smash product, has unit $1_A\# 1_H$. Also in \cite{caenepeel}, the authors proved that $A\#H$ is an $A$-bimodule, and here we highlight that $\underline{A\#H}$ is, in fact, the unital part of $A\#H$ as an $A$-bimodule. Note that $\underline{A\#H}$ is generated by the elements $a\#h=\sum a(h_{(1)}\cdot 1_A)\otimes h_{(2)}$ and $\sum a(h_{(1)}\cdot 1_A)\# h_{(2)}=a\#h$.
	
	For nonunital partial $H$-module algebras, we will construct the smash product in the same way: $A\#H=A\otimes H$ as vector space, and the product is given by 
	\begin{eqnarray*}
		(a\otimes h)(b\otimes k)=\sum a(h_{(1)}\cdot b)\otimes h_{(2)}k.
	\end{eqnarray*}
	
	Now, consider the vector space $\underline{A\# H}=(A\# H)(A\# 1_H)$, which corresponds to the partial smash product when $A$ has unit, and because of that, we will also call it partial smash product. In fact, $\underline{A\# H}$ is the unital sub $A$-bimodule of $A\#H$, and it is generated by the elements $\sum a(h_{(1)}\cdot b)\# h_{(2)}$.
	
	Also, note that whenever $A$ is at least an idempotent partial $H$-module algebra, then ${a\otimes 1_H\in \underline{A\# H}}$, for every $a\in A$. It is a consequence of the fact that $xy\otimes 1_H=(x\otimes  1_H)(y\otimes 1_H)$. Hence we have the following:
	\begin{proposition} \label{prop.smash.product.idempotent}
		Let $A$ be an idempotent partial $H$-module algebra. Then  $\underline{A\# H}$
		is an idempotent algebra.
	\end{proposition}
	In \cite{rafael} it was proved that, when $A$ is a unital algebra, $\underline{A\# H}$-modules are partial $(A,H)$-modules and vice-versa. But even for a nonunital partial $H$-module algebra $B$, essentially the same argument shows that whenever $M$ is a partial $(B,H)$-module, it is then a $\underline{B\# H}$-module with action
	\begin{eqnarray*}
		\sum a(h_{(1)}\cdot b)\# h_{(2)}\triangleright m=a(h\cdot (bm));
	\end{eqnarray*}
	moreover, if $M$ is a unital $B$-module (i.e., $B M = M$), then $M$ will also be a unital $\underline{B\# H}$-module. 
	For the converse we have two problems: one of them is that even if $M$ is a unital $\underline{B\# H}$-module, it might not be a unital $B$-module with respect to the induced action. The other problem is that, even when $M$ is a unital $B$-module, it is not clear how should one define a linear map $H\otimes M\to M$ such that, with the structure of $B$-module induced by the action of $\underline{B\# H}$, $M$ becomes a partial $(B,H)$-module. Nevertheless, this can be achieved  when $M$ is a unital $B$-module  and $B$ is a left $s$-unital algebra. 
	
	Recall that a ring $B$ is left $s$-unital if for each $b\in B$ there exists $x \in B$ such that $xb = b$ (see for instance (\cite{global-s-unital,garciasimon,tominaga})). It is clear that in this case $B$ is idempotent and also that $r(B) = 0$.  
	It can be shown that if $B$ is $s$-unital, $M$ is a unital $B$-module and $m_1, \ldots, m_n \in M$, then there exists $x \in B$ such that $xm_j = m_j$ for $j=1,2, \ldots, n$ \cite[Thm.1]{tominaga}. This property holds in particular for $M = B$. 
	
	Assume that $M$ is a unital $\underline{B\# H}$-module and that $B$ is a left $s$-unital algebra, then $M$ is also a unital $B$-module with the induced action. In fact, if $M$ is a unital $\underline{B\# H}$-module, then there exist $a_i, b_i \in B$ and $h^i\in H$ such that $$m=\sum_{i,h} a_i(h^i_{(1)}\cdot b_i)\# h^i_{(2)}\triangleright m_i,$$
	but since $B$ is a left $s$-unital algebra, there exists $x\in B$ such that $xa_i=a_i$ for every $a_i$, wich means that 
	$$m=x\#1_H\triangleright\sum_{i,h} a_i(h^i_{(1)}\cdot b_i)\# h^i_{(2)}\triangleright m_i.$$
	
	Now, given $h \in H$ and $m \in M$, we define $hm=\sum h_{(1)}\cdot x\# h_{(2)}\triangleright m$, where $x\in B$ is such that $xm=m$. Note that if also $m=ym$, there exists $z\in B$ such that $zy=y$ and $zx=x$,
	then $zm=m$ and
	\begin{eqnarray*}
		\sum h_{(1)}\cdot x\#h_{(2)}\triangleright m&=&\sum h_{(1)}\cdot zx\#h_{(2)}\triangleright m\\
		&=& \sum h_{(1)}\cdot z\#h_{(2)}\triangleright x\# 1_H \triangleright m\\
		&=& \sum h_{(1)}\cdot z\#h_{(2)}\triangleright xm\\
		&=& \sum h_{(1)}\cdot z\#h_{(2)}\triangleright ym\\
		&=& \sum h_{(1)}\cdot z\#h_{(2)}\triangleright y\# 1_H \triangleright m\\
		&=&\sum h_{(1)}\cdot zy\#h_{(2)}\triangleright m\\
		&=&\sum h_{(1)}\cdot y\#h_{(2)}\triangleright m,
	\end{eqnarray*}
	hence $h\otimes m\mapsto h\cdot m$ is well-defined.
	
	Moreover, since $M$ is also a unital $B$-module given by $a\otimes m\mapsto am=a\#1_H\triangleright m$, for every $a\in B, m\in M$, it follows that
	\[
	1_Hm= x\# 1_H\triangleright m = xm = m.
	\]
	Also, let $a\in B$, $m\in M$, $h,k\in H$ and consider $x,y\in B$ such that $xa=a$, $x(a(km))=a(km)$ and $ym=m$; then we have that
	\begin{eqnarray*}
		h(a(km))&=&\sum [h_{(1)}\cdot x\# h_{(2)}]\triangleright [a\# 1_H]\triangleright [k_{(1)}\cdot y\# k_{(2)}]\triangleright m\\
		&=&\sum h_{(1)}\cdot xa\# h_{(2)}\triangleright [k_{(1)}\cdot y\# k_{(2)}]\triangleright m\\
		&=&\sum [(h_{(1)}\cdot a)(h_{(2)}k_{(1)}\cdot y)\# h_{(3)}k_{(2)}]\triangleright m\\
		&=&\sum [h_{(1)}\cdot a\# 1_H]\triangleright [h_{(2)}k_{(1)}\cdot y\# h_{(3)}k_{(2)}]\triangleright m\\
		&=&\sum (h_{(1)}\cdot a)(h_{(2)}km).
	\end{eqnarray*} 
	Therefore, $M$ is a partial $(B,H)$-module.
	
	\begin{proposition} \label{prop.unital.module}
		Let $H$ be a Hopf algebra and $B$ be an left $s$-unital partial $H$-module algebra.
		\begin{enumerate}[\normalfont(1)]
			\item If $M$ is a unital $\underline{B\# H}$-module, then it is a partial $(B,H)$-module;
			\item If $M$ is a partial $(B,H)$-module, then it is a unital $\underline{B\# H}$-module.
		\end{enumerate}
	\end{proposition}

	Recall that in \cite{garciasimon}, García and Simon defined the torsion of a left $A$-module $M$ as the $A$-submodule $t_A(M)=\{m\in M\ | \ am=0,\,\forall a\in A\}$. With this concept we can prove a similar result to Proposition \ref{prop.unital.module} for a more general class of nonunital algebras. 
	
	Let $B$ be an idempotent algebra and $H$ a Hopf algebra with bijective antipode. We will show that a unital left $\underline{B\# H}$-module $M$ with trivial torsion $B$-submodule ($t_B(M)=\{m\in M| am=\sum a\#1_H\triangleright m=0,\,\forall a\in B\}=0$) is a partial $(B,H)$-module. 
	
	In fact, for every $a,b,c\in B$, $m\in M$, we have that
	\begin{eqnarray*}
		\sum a(h_{(2)}\cdot S^{-1}(h_{(1)})\cdot b)\#h_{(3)}\triangleright cm&=&\sum a(h_{(2)}\cdot S^{-1}(h_{(1)})\cdot b)\#h_{(3)}\triangleright c\#1_H\triangleright m\\
		&=&\sum a(h_{(2)}\cdot ((S^{-1}(h_{(1)})\cdot b)c))\#h_{(3)}\triangleright m\\
		&=&\sum ab(h_{(1)}\cdot c)\#h_{(2)}\triangleright m\\
		&=&ab\#1_H\triangleright \sum (h_{(1)}\cdot c)\#h_{(2)}\triangleright m.
	\end{eqnarray*}
	Then, since every element of $B$ is a finite sum of products in $B$ and $t_B(M)=0$, whenever $m=\sum_i a_im_i=\sum_j b_jn_j\in M$, we have that $\sum_i (h_{(1)}\cdot a_i)\#h_{(2)}\triangleright m_i=\sum_j (h_{(1)}\cdot b_j)\#h_{(2)}\triangleright n_j$. Hence, the linear map $H\otimes M\to M$ given by $h(am)=\sum h_{(1)}\cdot a\#h_{(2)}\triangleright m$ is well-defined, and straightforward computations show that $M$ is a partial $(B,H)$-module. To assure that $M$ will be a unital $B$-module induced by the action of $\underline{B\#H}$, note that since $B$ is idempotent, for every $a,b\in B$, $h\in H$, there exist $a_i, c_i \in B$ such that $a=\sum_i a_ic_i$ and, therefore, $$\sum_h a(h_{(1)}\cdot b)\# h_{(2)}=\sum_{i,h} (a_i\#1_H)(c_i(h_{(1)}\cdot b)\#h_{(2)}).$$
	This proves the next result. 
	\begin{proposition}
		Let $H$ be a Hopf algebra with bijective antipode and $B$ be any partial $H$-module algebra.
		\begin{enumerate}[\normalfont(1)]
			\item If $M$ is a unital $\underline{B\# H}$-module such that $$t_B(M)=\{m\in M;\,\,am= a\#1_H\triangleright m=0,\,\forall a\in B \}=0,$$ then it is a partial $(B,H)$-module;
			\item If $M$ is a partial $(B,H)$-module, then it is a unital $\underline{B\# H}$-module.
		\end{enumerate}
	\end{proposition}
	
	\begin{lemma} Let $B$ be a (not necessarily unital) algebra in $_H\mathcal{M}^{par}$, where $H$ is a Hopf algebra with bijective antipode. Then $H$ acts partially in $B$ with symmetrical partial action.
	\end{lemma}
	\begin{proof}
		In fact, for every $x,y \in B$, $h,k\in H$, since $H_{par}\simeq \underline{A\#H}$, we have that
		\begin{eqnarray*}
			1_H\bullet x=1_A\# 1_H\triangleright x=1_{\underline{A\# H}}\triangleright x=x,
		\end{eqnarray*}
		\begin{eqnarray*}
			h\bullet (x(k\bullet y))&=& 1_A\# h\triangleright (x(1_A\# k\triangleright y))\\
			&=& \sum (1_A\# h_{(1)}\triangleright x)((1_A\# h_{(2)})(1_A\# k)\triangleright y)\\
			&=&\sum (h_{(1)}\bullet x)(((h_{(2)}\cdot 1_A)\#h_{(3)}k)\triangleright y)\\
			&=&\sum (h_{(1)}\bullet x)((\epsilon_{h_{(2)}}\# 1_H)(1_A\#h_{(3)}k)\triangleright y)\\
			&=&\sum (h_{(1)}\bullet x)\epsilon_{h_{(2)}}((1_A\#h_{(3)}k)\triangleright y)\\
			&=&\sum (h_{(1)}\bullet x)\epsilon_{h_{(2)}}((h_{(3)}k)\bullet y)\\
			&\stackrel{ \eqref{eq5} }{=}&\sum (h_{(1)}\bullet x)((h_{(2)}k)\bullet y),
		\end{eqnarray*}
		and for the symmetrical case,
		\begin{eqnarray*}
			h\bullet((k\bullet x)y)&=&\sum ((1_A\# h_{(1)})(1_A\# k)\triangleright x)(1_A\# h_{(2)}\triangleright y)\\
			&=& \sum ((h_{(1)}\cdot 1_A\# h_{(2)}k)\triangleright x)(h_{(3)}\bullet y)\\
			&=&\sum ((h_{(2)}S^{-1}(h_{(1)})\cdot 1_A)(h_{(3)}\cdot 1_A)\# h_{(3)}k\triangleright x)(h_{(4)}\bullet y)\\
			&=& \sum (h_{(2)}\cdot S^{-1}(h_{(1)})\cdot 1_A\# h_{(3)}k\triangleright x)(h_{(4)}\bullet y)\\
			&=& \sum ((1_A \# h_{(2)})(S^{-1}(h_{(1)})\cdot 1_A\# k)\triangleright x)(h_{(3)}\bullet y)\\
			&=& \sum ((1_A \# h_{(4)})(S^{-1}(h_{(3)})\cdot 1_A\# S^{-1}(h_{(2)})h_{(1)}k)\triangleright x)(h_{(5)}\bullet y)\\
			&=&\sum ([(1_A\# h_{(3)})(1_A\# S^{-1}(h_{(2)}))](1_A\# h_{(1)}k)\triangleright x)(h_{(4)}\bullet y)\\
			&=& \sum ([(1_A\# h_{(3)})(1_A\# S^{-1}(h_{(2)}))]\triangleright (1_A\# h_{(1)}k)\triangleright x)(h_{(4)}\bullet y)\\
			&=& \sum (1_A\# h_{(1)}k)\triangleright x)\epsilon_{h_{(2)}}(h_{(3)}\bullet y)\\
			&\stackrel{ \eqref{eq4} }{=}& \sum (h_{(1)})k\bullet x)(h_{(2)}\bullet y).
		\end{eqnarray*}
	\end{proof}
	
	\begin{lemma}Let $H$ be a Hopf algebra with bijective antipode and $B$ a partial $H$-module algebra with symmetrical partial action. If either $B$ is idempotent, $r(B)=0$ or $l(B)=0$, then $B$ is an algebra in $_H\mathcal{M}^{par}$.
	\end{lemma}
	\begin{proof}
		We proved that a partial $H$-module algebra with these properties is a partial $H$-module and, since $H$ acts partially in $B$, the multiplication of $B$ is clearly a morphism of partial $H$-modules. Then, as the base algebra of the monoidal structure of $_H\mathcal{M}^{par}$ is $A$, i.e., the tensor of this category is over $A$, we only need to show that $B$ is, in fact, an $A$-bimodule and the multiplication of $B$ is balanced. First, we consider $B$ as an $A$-bimodule with the structure presented in \cite{alves4} given by
		\begin{eqnarray*}
			\epsilon_{h}x&=& \sum h_{(1)}\bullet S(h_{(2)})\bullet x\\
			x\epsilon_h&=& \sum h_{(2)}\bullet S^{-1}(h_{(1)})\bullet x.
		\end{eqnarray*}
		We already saw that the linear map $x\mapsto \sum h_{(1)}\cdot S(h_{(2)})\cdot x$ is a morphism of right $B$-modules and that the linear map $x\mapsto \sum h_{(2)}\cdot S^{-1}(h_{(1)})\cdot x$ is a morphism of left $B$-modules. Hence, the multiplication of $B$ is a morphism of $A$-bimodules. To show that the multiplication of $B$ is also balanced, we have that
		\begin{eqnarray*}
			(x\epsilon_{h})y&=&(\sum h_{(2)}\bullet S^{-1}(h_{(1)})\bullet x)y\\
			&=&\sum (h_{(2)}\bullet S^{-1}(h_{(1)})\bullet x)(h_{(3)}S(h_{(4)})\bullet y)\\
			&=& \sum  (h_{(2)}S^{-1}(h_{(1)})\bullet x)(h_{(3)}\bullet S(h_{(4)})\bullet y)\\
			&=& \sum x (h_{(1)}\bullet S(h_{(2)})\bullet y)\\
			&=& x(\epsilon_{h}y).
		\end{eqnarray*}
	\end{proof}
	These results lead us to the following theorem.
	
	\begin{theorem}\label{teo1}
		Let $H$ be a Hopf algebra with bijective antipode. Then, when we consider only algebras that either are idempotent or have trivial right or left annihilator, there exist a bijective correspondence between algebras in $_H\mathcal{M}^{par}$ and partial $H$-module algebras with symmetrical partial action.
	\end{theorem}
	
	\subsection{Partial Hopf Actions and Partial Group Actions}
	
	The concept of a partial action of a Hopf algebra has its origins in the theory of partial group actions, which predates the former by some 14 years (see for instance the survey \cite{Dokuchaev.survey}).
	We recall the definition of a partial action of a group $G$. 
	
	\begin{definition}[\cite{dok}] A partial action $\alpha$ of a group $G$ on an algebra $A$ consists of a family of two-sided ideals $D_g$ in $A$, $g\in G$, and algebra isomorphisms $\alpha_g: D_{g^{-1}} \to D_g$, such that:
		\begin{enumerate}[\normalfont(1)]
			\item $D_1 = A$ and $\alpha_1: A \to A$ is the identity;
			\item $\alpha_g(D_{g^{-1}}\cap D_h)\subseteq D_g\cap D_{gh}$;
			\item $\alpha_g(\alpha_h(x))=\alpha_{gh}(x)$, for any $x\in D_{h^{-1}}\cap D_{(gh)^{-1}}$.
		\end{enumerate} 
	\end{definition}
	
	Note that $A$ might  be a nonunital algebra. It is well-known that if $A$ is a unital algebra then there is a bijective correspondence between partial actions of the group algebra $\Bbbk G$ and partial actions of the group $G$ on $A$ such that the ideals $D_g$ are generated by central idempotents, i.e., $D_g=Ae_g$ for every $g\in G$, $e_g$ a central idempotent (see for instance \cite{alves1} and \cite{caenepeel}). We will prove that a similar correspondence still holds when we consider algebras without identity.
	
	The correspondence proved in \cite{caenepeel} implies that given a symmetrical partial $\Bbbk G$-action on a unital algebra $A$, the elements $g \cdot 1_A$, for $g \in G$, are central idempotents of $A$, and the ideals of the induced partial $G$-action are of the form $D_g=A(g\cdot 1_A)$. But for every $a\in A$, we have that $$a(g\cdot 1_A)=(gg^{-1} \cdot a)(g\cdot 1_A)=g\cdot g^{-1}\cdot a,$$ which means that if we define the algebra maps 
	\begin{eqnarray*}
		\psi_g: A&\to & A\\
		a&\mapsto &g\cdot g^{-1}\cdot a,
	\end{eqnarray*}
	then $D_g=\psi_g(A)$. This remark is the key for extending the correspondence between symmetrical partial $\Bbbk G$-algebras and a class of partial $G$-actions for nonunital algebras.
	
	\begin{lemma} Consider a symmetric partial $\Bbbk G$-action on an associative algebra $A$, let $g \in G$ and let $\psi_g : A \to A$ be as defined above. 
		\begin{enumerate}[\normalfont(1)]
			\item $\psi_g$ is an $A$-bimodule map.
			\item $\psi_g(A)$ is an ideal of $A$.
			\item If $A=A^2$, $r(A)=0$ or $l(A)=0$ then $\psi_g^2=\psi_g$. 
		\end{enumerate}
	\end{lemma}
	\begin{proof}
		In fact, for every $g\in G$, $a,b,x\in A$, we have that
		\[		\psi_g(axb)=g\cdot g^{-1}\cdot (axb)
		=(gg^{-1}\cdot a)(g\cdot g^{-1}\cdot (xb))
		=a(g\cdot g^{-1}\cdot x)(gg^{-1}\cdot b)
		=a\psi_g(x)b.
		\]
		It is then obvious that $\psi_g(A)$ is an ideal. 
		Since $\psi_g$ is an $A$-bimodule map and also an algebra map it also follows that, for every $a,b\in A$,  
		\[
		\psi_g^2(ab)=\psi_g(a\psi_g(b))
		=\psi_g(a)\psi_g(b)
		=\psi_g(ab)
		\]
		and
		\[
		a\psi_g^2(b)=\psi_g^2(ab)
		=\psi_g(ab)
		=a\psi_g(b),
		\]
		which means that $a(\psi_g^2(b)-\psi_g(b))=0$. Analogously, we can show that $(\psi_g^2(b)-\psi_g(b))a$, for every $a,b\in A$. Therefore $\psi_g^2=\psi_g$ whenever $A=A^2$, $r(A)=0$ or $l(A)=0$.
	\end{proof}

	The next lemma follows directly from the previous section.
	
	\begin{lemma}\label{l9}Let $A$ be a partial $\Bbbk G$-module algebra with symmetrical partial action and $\psi_g$ as before. If $A$ is idempotent, $r(A)=0$ or $l(A)=0$, then:
		\begin{enumerate}[\normalfont(1)]
			\item $g\cdot h\cdot h^{-1}\cdot a\in \psi_{gh}(A)$;
			\item $g\cdot h\cdot a=gh\cdot h^{-1}\cdot h\cdot a$;
			\item $\psi_g\psi_h=\psi_h\psi_g$;
			\item $g\cdot g^{-1} \cdot h\cdot a=h\cdot h^{-1}g\cdot g^{-1}h\cdot a$,
		\end{enumerate} 
		for every $a\in A$, $g,h\in G$.
	\end{lemma}
	
	\begin{proposition}\label{prop1} Let $A$ be an associative algebra such that $A^2=A$, $r(A)=0$ or $l(A)=0$. There is a bijective correspondence between symmetrical partial $\Bbbk G$-actions on $A$ and partial $G$-actions $\alpha=\{\{\alpha_g\}_{g\in G},\{D_g\}_{g\in G} \}$ on $A$ endowed with algebra epimorphisms $p_g: A\to D_g$, such that:
		\begin{enumerate}[\normalfont(1)]
			\item $p_1: A\to A$ is the identity of $A$
			\item $p_g^2=p_g$;
			\item $p_gp_h=p_hp_g$;
			\item $p_g\alpha_h=\alpha_hp_{h^{-1}g}|_{D_{h^{-1}}}$. 
		\end{enumerate}
	\end{proposition}
	\begin{proof}
		Let $A$ be a partial $\Bbbk G$-module algebra with symmetrical partial action. For each $g\in G$ consider $D_g=\psi_g(A)$, $p_g=\psi_g$ and define $\alpha_g: D_{g^{-1}}\to D_g$ by $\alpha_g(\psi_{g^{-1}}(a))=g\cdot \psi_{g^{-1}}(a)$. We will prove that $\alpha$ is a partial group action.
		First, note that $\psi_1(A)=A$ and $\alpha_1(\psi_1(a))=a$. Now, 
		\[		\psi_g(a)=\psi_g^2(a)
		=g\cdot g^{-1}\cdot g\cdot g^{-1}\cdot a
		= \alpha_g(\psi_{g^{-1}}(g^{-1}\cdot a)), 
		\]
		i.e., $\alpha_g$ is surjective. For the injectivity of $\alpha_g$, assume that $\alpha_g(\psi_{g^{-1}}(a))=\alpha_g(\psi_{g^{-1}}(b))$, then
		\begin{eqnarray*}
			g^{-1}\cdot g\cdot a&\stackrel{ \eqref{eq3} }{=}&g^{-1}\cdot g\cdot g^{-1}\cdot g\cdot a\\
			&=&g^{-1}\cdot\alpha_g(\psi_{g^{-1}}(a))\\
			&=& g^{-1}\cdot\alpha_g(\psi_{g^{-1}}(b))\\
			&=& g^{-1}\cdot g\cdot g^{-1}\cdot g\cdot b\\
			&=& g^{-1}\cdot g\cdot b,
		\end{eqnarray*}
		i.e., $\psi_{g^{-1}}(a)=\psi_{g^{-1}}(b)$. Since all $\alpha_g$ are clearly algebra morphisms, we have that they are algebra isomorphisms.
		
		Now, let us prove that $\alpha_g(D_{g^{-1}}\cap D_h)=D_g\cap D_{gh}$. In fact, for any $x\in D_{g^{-1}}\cap D_h$, there exist $a,b\in A$ such that $x=\psi_{g^{-1}}(a)=\psi_h(b)$, then on one  hand
		\[
		\alpha_g(x)  = \alpha_g(\psi_{g^{-1}}(a))  = g\cdot g^{-1}\cdot g\cdot a 
		= \psi_g(g\cdot a)\in D_g,
		\]
		and on the other hand,
		\[ 
		\alpha_g(x)=\alpha_g(\psi_h(b))
		=g\cdot h\cdot h^{-1}\cdot b \in \psi_{gh}(A)=D_{gh}.
		\] 
		Hence $\alpha_g(D_{g^{-1}}\cap D_h)\subseteq D_g\cap D_{gh}$. 	
		
		For the last axiom of the definition of partial group action, let $x\in D_{h^{-1}}\cap D_{(gh)^{-1}}$, then there exists $a\in A$ such that $x=\psi_{h^{-1}}(a)$ and
		\[
		\alpha_g(\alpha_h(x)) = \alpha_g(\alpha_h(\psi_{h^{-1}}(a)))\\
		= g\cdot h\cdot h^{-1}\cdot h\cdot a\\
		=  gh\cdot h^{-1}\cdot h\cdot a\\
		= \alpha_{gh}(x).
		\]
		Item (4) is a direct consequence of item (4) of Lemma \ref{l9}.
		
		For the converse of the proposition, suppose that $\alpha$ is a partial group action of $G$ on $A$ and that there exist algebra morphisms of $A$ onto the ideals $p_g: A\to D_g$ that satisfy items (1), (2), (3) and (4) of the proposition. Define a map $\Bbbk G \otimes A \to A$ by  $g\cdot a=\alpha_g(p_{g^{-1}}(a))$, then $$1_G\cdot a=\alpha_1(p_1(a))=a$$ and
		\begin{eqnarray*}
			g\cdot(a(h\cdot b))&=&\alpha_g(p_{g^{-1}}(a\alpha_h(p_{h^{-1}}(b))))\\
			&=&\alpha_g(p_{g^{-1}}(a))\alpha_g(p_{g^{-1}}(\alpha_h(p_{h^{-1}}(b)))\\
			&=&\alpha_g(p_{g^{-1}}(a))\alpha_g(\alpha_h(p_{h^{-1}g^{-1}}p_{h^{-1}}(b)))\\
			&=&\alpha_g(p_{g^{-1}}(a))\alpha_{(gh)}(p_{h^{-1}g^{-1}}p_{h^{-1}}(b))\\
			&=&\alpha_g(p_{g^{-1}}(a))\alpha_{(gh)}(p_{h^{-1}}p_{h^{-1}g^{-1}}(b))\\
			&=&\alpha_g(p_{g^{-1}}(a))p_g(\alpha_{(gh)}(p_{h^{-1}g^{-1}}(b)))\\
			&=&p_g(\alpha_g(p_{g^{-1}}(a))\alpha_{(gh)}(p_{h^{-1}g^{-1}}(b)))\\
			&=&\alpha_g(p_{g^{-1}}(a))\alpha_{(gh)}(p_{h^{-1}g^{-1}}(b))=(g\cdot a)(gh\cdot b).
		\end{eqnarray*}
		Analogously, $g\cdot((h\cdot a)b)=(gh\cdot a)(g\cdot b)$. Hence $A$ is a partial $\Bbbk G$-module algebra with symmetrical partial action.
	\end{proof}
	
	\begin{definition}\label{def2} Let $\alpha$ be a partial action of the group $G$ on the algebra $A$. A family of algebra morphisms $\{p_g: A\to D_g\}_{g\in G}$ that satisfies all the properties mentioned in the previous proposition will be called
		a family of $\alpha$-projections. 
	\end{definition}
	Note that $\alpha$-projections correspond, in the case of nonunital algebras, to the central idempotents that characterize, in the case of unital algebras, the partial G-actions that come from partial $\Bbbk G$-actions. In fact, if $A$ is unital and $D_g=A1_g$, we only need to consider the linear maps $p_g: A\to D_g$ defined by $p_g(a)=a1_g$. Clearly each $p_g$ is an algebra map and any two such projections commute. For the last property, we have that
	\begin{eqnarray*}
		\alpha_kp_{k^{-1}g}p_{k^{-1}}(a1_{k^{-1}})&=&\alpha_k(a1_{k^{-1}}1_{k^{-1}g})\\
		&=&\alpha_k(a1_{k^{-1}})\alpha_k(1_{k^{-1}}1_{k^{-1}g})\\
		&=&\alpha_k(a1_{k^{-1}})1_k1_g\\
		&=& \alpha_k(a1_{k^{-1}})1_g\\
		&=&p_g\alpha_k(a1_{k^{-1}}).
	\end{eqnarray*}
	
	Actually, the partial $G$-action induced by a partial $\Bbbk G$-action is a \textit{regular} partial action, as we will prove next.

	\begin{definition}[\cite{abadie}] A regular partial $G$-action is a partial action such that for every $g_1,\cdots ,g_n$ $\in G$, we have that
		\begin{eqnarray*}
			D_{g_1}\cap\cdots\cap D_{g_n}=D_{g_1}\cdot D_{g_2} \cdots D_{g_n}.
		\end{eqnarray*}
	\end{definition}
	If we consider only idempotent algebras, then the bijection of the previous proposition is between symmetrical partial $\Bbbk G$-actions and regular partial $G$-actions. In fact, for every $g_1,\cdots g_k\in G$, since every $p_{g_i}$ is a projection, we have that if $x\in D_{g_1}\cap\cdots\cap D_{g_k}$ then
	\begin{eqnarray*}
		x=p_{g_1}(x)=\cdots=p_{g_k}(x)
	\end{eqnarray*}
	and since $x=\sum_{i=1}^{n}a_{1i}\cdots a_{ki}$, for some $a_{ji}\in A$, $1\leq j\leq k$, $1\leq i\leq n$,
	\begin{eqnarray*}
		x&=&p_{g_1}\circ\cdots\circ p_{g_k}(x)\\
		&=&\sum _{i=1}^n p_{g_1}\circ\cdots\circ p_{g_k}(a_{1i}\cdots a_{ki})\\
		&=&\sum _{i=1}^n (p_{g_1}\circ\cdots\circ p_{g_k}(a_{1i}))\cdots (p_{g_1}\circ\cdots\circ p_{g_k}(a_{ki})) \in D_{g_1}\cdot D_{g_2} \cdots D_{g_k},
	\end{eqnarray*}
	because the projections commute. Hence $D_{g_1}\cap\cdots\cap D_{g_n}\subseteq D_{g_1}\cdot D_{g_2} \cdots D_{g_n}$, and the other inclusion is a consequence of every $D_g$ being an ideal.

	\subsection{Connections with partial actions of Multiplier Hopf Algebras}
	
	Fonseca et al. introduced partial actions of a multiplier Hopf algebra $H$ on a nondegenerate algebra A, i.e., $r(A)=l(A)=0$. In this section we will show that when $H$ has $1_H$ (i.e, $H$ is a Hopf algebra), both definitions of partial actions coincide. Let us recall the definition of partial action considered in \cite{ABFFM21} (see also \cite{globalization-fonseca-fontes-martini}).
	
	\begin{definition}\label{partial.multiplier}\cite{ABFFM21} Let $H$ be a regular multiplier Hopf algebra, $A$ be a nondegenerate algebra and $M(A)$ be its multiplier algebra. A triple $(A,\cdot,\epsilon)$ is a partial $H$-module algebra if $\cdot$ is a linear map,
		\begin{eqnarray*}
			\cdot : H\otimes A&\to& A\\
			h\otimes a&\mapsto& h\cdot a
		\end{eqnarray*}
		and $\epsilon$ is a linear map $\epsilon: H\to M(A)$ satisfying the following conditions, for all $h,k\in H$ and $a,b\in A$:
		\begin{enumerate}[\normalfont(1)]
			\item $h\cdot(a(k\cdot b))=\sum (h_{(1)}\cdot a)(h_{(2)}k\cdot b)$;
			\item $\epsilon(h)(k\cdot a)=\sum h_{(1)}\cdot (S(h_{(2)})k\cdot a)$ and $\epsilon(H)A\subseteq H\cdot A$;
			\item given $h_1,...,h_n \in H$ and $a_1,...,a_m\in A$, there exists $k\in H$ such that $h_ik=kh_i=h_i$ and $h_i\cdot a_j=h_i\cdot k\cdot a_j$, for all $1\leq i\leq n$ and $1\leq j\leq m$;
			\item $H\cdot a=0$ if and only if $a=0$, that is, $\cdot$ is a nondegenerate action.
		\end{enumerate}
		Under these conditions, the map $\cdot$ is called a partial action of H on A, and we say that it is symmetric if the following additional conditions also hold:
		\begin{enumerate}[\normalfont(1)]\addtocounter{enumi}{4}
			\item $h\cdot((k\cdot a)b)=\sum (h_{(1)}k\cdot a)(h_{(2)}\cdot b)$;
			\item $(k\cdot a)\epsilon(h)=\sum h_{(2)}\cdot (S^{-1}(h_{(1)})k\cdot a)$;
			\item $A\epsilon(H)\subseteq H\cdot A$,
		\end{enumerate}
		for all $a,b\in A$ and $h,k\in H$.
	\end{definition}
	It is proved in \cite{ABFFM21} that when $H$ is a (unital) Hopf algebra then items (3) and (4) are equivalent to the condition that $1_H\cdot a=a$ for all $a\in A$; and that if $A$ is a unital algebra, then $\epsilon(h)=h\cdot 1_A$. This shows that when there exists $1_H$ and $A$ is a nondegenerate algebra, we have a partial action as in Definition \ref{partial.action}. Conversely, if $H$ is a Hopf algebra with bijective antipode and $A$ is a nondegenerate symmetric partial $H$-module algebra, we can define $\epsilon: H\to M(A)$ by $\epsilon(h)=(e_l(h),e_r(h))$ where
	$$e_l(h)(a)=\sum h_{(1)}\cdot (S(h_{(2)})\cdot a)$$
	and
	$$e_r(h)(a)=\sum h_{(2)}\cdot(S^{-1}(h_{(1)})\cdot a),$$
	then we have a symmetric partial action as in Definition \ref{partial.multiplier}. In fact, since for all $h, k\in H$ and $a,b\in A$
	\begin{eqnarray*}
		(\sum h_{(2)}\cdot S^{-1}(h_{(2)})\cdot a)b&=&(\sum h_{(2)}\cdot S^{-1}(h_{(2)})\cdot a)(h_{(2)}S(h_{(4)})\cdot b)\\
		&=&(\sum h_{(2)}\cdot [(S^{-1}(h_{(2)})\cdot a)(S(h_{(4)})\cdot b)]\\
		&=&(\sum (h_{(2)}S^{-1}(h_{(2)})\cdot a)(h_{(3)}\cdot S(h_{(4)})\cdot b)\\
		&=&a(\sum h_{(1)}\cdot S(h_{(2)})\cdot b),
	\end{eqnarray*}
	which shows that $[\epsilon_r(h)(a)]b=a[\epsilon_l(h)(b)]$, i.e., $\epsilon$ is well-defined, $\epsilon(H)A\subseteq H\cdot A$ and $A\epsilon(H)\subseteq H\cdot A$. Moreover,
	\begin{eqnarray*}
		[e(h)(k\cdot a)]b&=&[\sum h_{(1)}\cdot S(h_{(2)})\cdot k\cdot a]b\\
		&=&\sum h_{(1)}\cdot S(h_{(2)})\cdot [(k\cdot a)b]\\
		&=&\sum h_{(1)}\cdot [(S(h_{(3)})k\cdot a)(S(h_{(2)})\cdot b)]\\
		&=&\sum (h_{(1)}\cdot S(h_{(4)})k\cdot a)(h_{(2)}S(h_{(3)})\cdot b)\\
		&=&[\sum h_{(1)}\cdot S(h_{(2)})k\cdot a]b,
	\end{eqnarray*}
	proving that item (2) of Definition \ref{partial.multiplier} is satisfied. Analogously, we have that $(k\cdot a)\epsilon(h)=\sum h_{(2)}\cdot S^{-1}(h_{(2)})k\cdot a$.

	\section{Globalization of a Partial Hopf Action \label{section.globalization}}
	
	If $B$ is an $H$-module algebra with action $\triangleright : H \otimes B \to B$ and 
	$e$ is a central idempotent of $B$, then $A = eB$ is a unital ideal of $B$ and the formula
	\begin{equation} \label{equation.formula.restriction.unital.case}
		h \cdot a = e (h \triangleright a)    
	\end{equation}
	defines a partial $H$-action on $A$  which is called an induced, or restricted, partial action in \cite{alves1}.
	Conversely, in \cite{alves1} and \cite{alves3} it is shown that every partial Hopf action on a unital algebra admits a globalization, which in particular means that every partial action is the restriction of a global action; moreover, there exists a minimal globalization that is unique up to isomorphism. The original partial action is equivalent to a restricted partial action in this globalization. 
	
	In this section we will show that every symmetrical partial Hopf action is a restriction of a global action as in the case of partial actions on algebras with unit.
	
	Going back to the situation where $B$ is an $H$-module algebra, $e$ is a central idempotent of $B$ and $A = eB$, the partial action on $A$ defined by formula \eqref{equation.formula.restriction.unital.case} comes from the canonical projection of $B$ onto $A$ given by multiplication by $e$.  In the following we will show that if $A$ has trivial right or left annihilator then there exists a projection
	(Proposition \ref{proposition.quasi-g.is.g.if.r(A).or.l(A).trivial}) of the globalization onto $A$ which plays the role of multiplication by a central idempotent in the unital case. In Section \ref{section.local.units} we show that such a projection may be obtained for an ideal which has local units (see Proposition \ref{proposition.induced.action.local.units}).
	
	Let us recall the definition of globalization for partial actions on unital algebras. 
	\begin{definition}[\cite{alves1},\cite{alves3}] \label{definition.globalization} Let $\cdot: H \otimes A \longrightarrow A$ be a symmetrical partial action, where $A$ is a unital algebra. A pair $(B,\theta)$ is called an enveloping action, or globalization, for the partial action $\cdot$ if:
		\begin{enumerate}[\normalfont(1)]
			\item $B$ is an $H$-module algebra (possibly nonunital), with action $\triangleright$;
			\item $\theta: A \longrightarrow B$ is a monomorphism of algebras;
			\item $\theta(h\cdot a)=\theta(1_A)(h\triangleright \theta(a))=(h\triangleright \theta(a))\theta(1_A)$, $\forall a\in A, \,h\in H$;
			\item $B=H\triangleright\theta(A)$.
		\end{enumerate} 
	\end{definition}
	Items (3) and (4) imply that $\theta (1_A)$ is a central idempotent of $B$ and that $\theta(A)$ is the ideal generated by $\theta(1_A)$. Hence
	\[
	h \rightharpoonup x : = \theta (1_A) (h \triangleright x), \ h \in H, x \in \theta(A), 
	\]
	determines a partial action on $\theta(A)$. It follows from items (2) and (3)  that $\theta$ is a monomorphism of partial actions.
	
	Now, let $A$ be any associative algebra that is a partial $H$-module algebra with symmetrical partial action $\cdot: H\otimes A\to A$. Clearly, item (3) of the previous definition of globalization poses a problem in  the context of nonunital algebras,  and so we introduce a version of globalization for partial actions on associative algebras in general.
	
	\begin{definition}\label{def3} A \textit{globalization} for the partial action $\cdot$ is a pair $(B,\theta)$ such that 
		\begin{enumerate}[\normalfont(1)]
			\item $B$ is an $H$-module algebra (possibly nonunital), with action $\triangleright$;
			\item $\theta: A \longrightarrow B$ is a monomorphism of algebras;
			\item For every $a,b\in A$ and $h\in H$, we have $\theta((h\cdot a)b)=(h\triangleright \theta(a))\theta(b)$ and $\theta(b(h\cdot a))=\theta(b)(h\triangleright \theta(a))$ 
			\item $B=H\triangleright\theta(A)$.
			\item $h\bullet\theta(a)=\theta(h\cdot a)$ defines a partial action of $H$ on $\theta(A)$.
		\end{enumerate}
	\end{definition}
	Note that if $(B,\theta)$ is a globalization for $\cdot$, from items (3) and (4), we still have that $\theta(A)$ is an ideal of $B$, as can be seen in \cite{alves1}, just as for the globalization of a partial $H$-action on a unital algebra. 
	
	We remark that if $A$ is a partial $H$-module algebra with unit and $(B,\theta)$ is a globalization, then $(B,\theta)$ is actually a globalization in the sense of Definition \ref{definition.globalization} and vice-versa. 
	
	Note that if $h\bullet \theta(a):=\theta(h\cdot a)$ defines a partial action of $H$ on $\theta(A)$, we have that $\theta$ turns out to be a morphism of partial $H$-module algebras. Moreover,  $\theta(b)(h\bullet \theta(a))=\theta(b)\theta(h\cdot a)=\theta(b)(h\triangleright \theta(a))$ and this suggests that $h\bullet \theta(a)$ and $h\triangleright \theta(a)$ could be associated in some way. 
	When $A$ is a unital algebra, we have that $\theta(h\cdot a)=\theta(1_A(h\cdot a))=\theta(1_A)(h\triangleright\theta(a))=:h\rightharpoonup \theta(a)$, so item (5) of the definition of globalization was not necessary in this case; moreover, the existence of $\theta(1_A)$ provides a projection of $B=H\triangleright\theta(A)$ onto $\theta(A)$ in such a manner that we could say that the partial action $\cdot:H\otimes A\to A$ was induced by the action $\triangleright: H\otimes B\to B$. These considerations motivate the next result.

	\begin{proposition} \label{proposition.projection}
		Let $B$ be an $H$-module algebra with action $\triangleright$ and $A$ an ideal of $B$. If there exists a projection $\pi: H\triangleright A\to H\triangleright A$, with image $A$, such that
		\[\pi(h\triangleright (a\pi(k\triangleright b)))=\sum \pi(h_{(1)}\triangleright a))\pi(h_{(2)}k\triangleright b),\]
		for every $h,k\in H$ and $a,b\in A$, then the linear map $h\ap a=\pi(h\triangleright a)$ defines a partial action of $H$ on $A$.
	\end{proposition}
	
	\begin{proof}
		In fact, we have that $1_H\bullet a=\pi(1_H\triangleright a)=\pi(a)=a$ and 
		\[
		h \ap(a(k \ap b))= \pi(h\triangleright (a\pi(k\triangleright b)))= \sum \pi(h_{(1)}\triangleright a))\pi(h_{(2)}k\triangleright b)
		=\sum (h_{(1)} \ap a)(h_{(2)}k \ap b),
		\]
		for every $h,k\in H$ and $a,b\in A$.
	\end{proof}
	
	Now we will show that, whenever $r(A)=0$ or $l(A)=0$, then the partial action is induced by the action of the globalization.

	\begin{proposition} \label{proposition.quasi-g.is.g.if.r(A).or.l(A).trivial}
		Let $A$ be a partial $H$-module algebra with symmetrical partial action given by $\cdot:H\otimes A\to A$ and $(B,\theta)$ a globalization where $B$ is an $H$-module algebra with action $\triangleright$. If $r(A)=0$ or $l(A)=0$, then $\pi(h\triangleright\theta(a))=\theta(h\cdot a)$ defines a projection as in Proposition \ref{proposition.projection}.
	\end{proposition}
	\begin{proof}
		We will assume that $r(A)=0$, since the case where $l(A)=0$ follows from an analogous argument. Assume that $\sum_ih_i\triangleright\theta(a_i)=0$, then for every $x\in A$ we have that
		\begin{eqnarray*}
			\theta(x)(\sum_ih_i\triangleright\theta(a_i))&=&\sum_i\theta(x(h_i\cdot a_i))\\
			&=&\theta(x)(\sum_i\theta(h_i\cdot a_i))\\
			&=&0.
		\end{eqnarray*}
		Since $\theta$ is a monomorphism and $r(A)=0$, we have that $\sum_i\theta(h_i\cdot a_i)=0$, i.e., $\pi(h\triangleright\theta(a))=\theta(h\cdot a)$ is well-defined. The fact that $\pi$ satisfies the property of Proposition \ref{proposition.projection} follows immediately from the hypothesis that $h\bullet\theta(a)=\theta(h\cdot a)$ defines a partial action of $H$ on $\theta(A)$.
	\end{proof}

	Proposition \ref{proposition.quasi-g.is.g.if.r(A).or.l(A).trivial} shows that the existence of a globalization implies that the partial action $\cdot$ is induced by the global action $\triangleright$.
	
	\begin{remark}
		In \cite{globalization-fonseca-fontes-martini}, a projection that satisfies the conditions of Proposition \ref{proposition.projection} is called an $H$-projection.
	\end{remark}

	\begin{definition} Let $\cdot: H \otimes A \longrightarrow A$ be a symmetrical partial action. A globalization $(B,\theta)$ will be called \textit{minimal} if $\sum_{i}kh_i\cdot a_i=0$, for all $k\in H$, implies $\sum_{i}h_i\triangleright \theta(a_i)=0$.
	\end{definition}
	
	In \cite{alves1} the authors defined minimal globalizations in the following way:
	
	\begin{definition}
		Let $H$ be a Hopf algebra and $A$ a unital partial $H$-module algebra. A globalization $(B,\theta)$ is minimal if $\theta(1_A)M=0$ implies that $M=0$, for every $H$-submodule $M$ of $B$.
	\end{definition}
	
	Since $\theta(A)$ is an ideal of $B$, we can consider the projection $\Pi_0: B\to A$ given by $\Pi_0(x)=\theta(1_A)x$; then the previous definition is equivalent to saying that the kernel of $\Pi_0$ does not contain any nonzero $H$-submodule of $B$. 
	
	When $r(A) =0$ there is an analogous projection.

	\begin{definition}
		Let $A$ be a partial $H$-module algebra such that $r(A)=0$ and let $(B,\theta)$ be a globalization; then we define $\Pi: B\to \theta(A)$ as the linear projection $\Pi(h\triangleright\theta(a))=\theta(h\cdot a)$.
	\end{definition}
	The linear map $\Pi: B\to \theta(A)$, given by $\Pi(h\triangleright\theta(a))=\theta(h\cdot a)$ (that is equivalent to $\Pi_0$ when $A$ is unital), is well-defined because if $\sum_i h_i\triangleright\theta(a_i)=\sum_j k_j\triangleright\theta(b_j)$, then for every $c\in A$ we have that
	\[
	\theta(c\sum_ih_i\cdot a_i) = \theta(c)\sum_i h_i\triangleright\theta(a_i) = \theta(c)\sum_j k_j\triangleright\theta(b_j)
	=\theta(c\sum_j k_j\cdot b_j),
	\]
	and since $r(A)=0$ and $\theta$ is a monomorphism, it follows that $\sum_i h_i\cdot a_i=\sum_j k_j\cdot b_j$. Note also that $\Pi$ is a projection by its definition. 
	
	\begin{proposition}
		Let $H$ be a Hopf algebra and $A$ a (not necessarily unital) partial $H$-module algebra. If $r(A)=0$, then a globalization $(B,\theta)$ is minimal if and only if the kernel of $\Pi$ does not contain any nonzero $H$-submodule of $B$.
	\end{proposition}  
	\begin{proof}
		Suppose that the globalization $(B,\theta)$ is minimal and let $M$ be an $H$-submodule of $B$ such that $\Pi(M)=0$. Then for every $\sum_i h_i\triangleright\theta(a_i)\in M$, we have that $\sum_i kh_i\triangleright\theta(a_i)\in M$ for every $k\in H$, which means that $\theta(\sum_i kh_i\cdot a_i)=0$. Since $\theta$ is injective and the globalization is minimal, we conclude that $\sum_i h_i\triangleright\theta(a_i)=0$, hence $M=0$. Let us now assume that the kernel of $\Pi$ does not contain any nonzero $H$-submodule of $B$ and suppose that $\sum_i kh_i\cdot a_i=0$ for every $k\in H$. Then the $H$-submodule $M$ of $B$ generated by $\sum_i h_i\triangleright\theta(a_i)$ is contained in the kernel of $\Pi$, but since $\ker\Pi$ does not contain any nonzero $H$-submodule of $B$, we conclude that $M=0$, hence $\sum_i h_i\triangleright\theta(a_i)=0$.
	\end{proof}
	
	\begin{lemma}\label{r(B).r(A).minimal} Let $\cdot: H\otimes A \to A$ be a partial action and $(B,\theta)$ a globalization.
		\begin{enumerate}[\normalfont(1)]
			\item If $r(B)=0$, then $r(A)=0$;
			\item If $r(B)=0$, then $(B,\theta)$ is a minimal globalization;
			\item If $r(A)=0$, $H$ has bijective antipode and $(B,\theta)$ is a minimal globalization, then $r(B)=0$.
		\end{enumerate}
	\end{lemma}
	\begin{proof}
		First, note that if $a\in r(A)$, then $(h\triangleright\theta(b))\theta(a)=\theta((h\cdot b)a)=0$, for every $h\in H$, $b\in B$, i.e., $\theta(a)\in r(B)=0$, and since $\theta$ is a monomorphism, we have that $a=0$.
		
		Now, note that if $\sum_i kh_i\cdot a_i=0$ for every $k\in H$, then
		\begin{eqnarray*}
			(k\triangleright\theta(b))(\sum_i h_i\triangleright\theta(a_i))&=&\sum(k_{(1)}\triangleright\theta(b))(\sum_i k_{(2)}S(k_{(3)})h_i\triangleright\theta(a_i))\\
			&=&\sum k_{(1)}\triangleright[\theta(b)(\sum_i S(k_{(3)})h_i\triangleright\theta(a_i))]\\
			&=&\sum k_{(1)}\triangleright\theta(b(\sum_iS(k_{(3)})h_i\cdot a_i))\\
			&=&0
		\end{eqnarray*}
		Therefore $\sum_i h_i\triangleright\theta(a_i)\in r(B)=0$, hence $(B,\theta)$ is a minimal globalization.
		
		Finally, suppose that $r(A)=0$ and that $(B,\theta)$ is a minimal globalization. If $\sum_i h_i\triangleright\theta(a_i)\in r(B)$, we have that
		\begin{eqnarray*}
			\theta(b(\sum_i kh_i\cdot a_i))&=&\theta(b)(\sum_i kh_i\triangleright\theta(a_i))\\
			&=&\sum (k_{(2)}S^{-1}(k_{(1)})\triangleright\theta(b))(\sum_i k_{(3)}h_i\triangleright\theta(a_i))\\
			&=&\sum k_{(2)}\triangleright[(S^{-1}(k_{(1)})\triangleright\theta(b))(\sum_ih_i\triangleright\theta(a_i))]\\
			&=&0
		\end{eqnarray*}
		for every $k\in H$ and $b\in A$. Since $\theta$ is a monomorphism, we have that $\sum_i kh_i\cdot a_i\in r(A)=0$, and since $(B,\theta)$ is minimal, we have that $\sum_i h_i\triangleright\theta(a_i)=0$.
	\end{proof}
	The previous lemma shows that there exist a (maybe) easier way to verify if a globalization is minimal. Now we will show that there always exists a minimal globalization for any symmetrical partial action over any associative algebra.
	
	We recall that the convolution algebra $\HOM(H,A)$ is a left $H$-module algebra with action $\triangleright$ given by $(k\triangleright f)(h)=f(hk)$. Any partial action $\cdot: H \otimes A \longrightarrow A$ induces a morphism of algebras $\varphi: A \longrightarrow \HOM(H,A)$ given by $\varphi(a)(h)=h\cdot a$, and 
	$$H\triangleright \varphi(A) = \left\{ \sum_i h_i \triangleright \varphi(a_i) ; h_i \in H, a_i \in A \right\}$$ is the smallest $H$-submodule algebra of $\HOM(H,A)$ that contains $\varphi(A)$.
	
	\begin{lemma} Let $\cdot: H \otimes A \longrightarrow A$ be a partial action. Then 
		\begin{enumerate}[\normalfont(1)]
			\item $B=H\triangleright \varphi(A)$ is an $H$-module subalgebra of $\HOM (H,A)$;
			\item If the partial action is symmetrical then $\varphi(A)$ is an ideal of $B$;
			\item If $r(A)=0$ and $\varphi(A)$ is an ideal of $B$, then the partial action is symmetrical.
		\end{enumerate}
	\end{lemma}
	\begin{proof}
		Let $k,h \in H$, $a,b \in A$. Then the equality 
		\begin{equation} \label{equation.varphi.induced}
			\varphi(a)(h\triangleright\varphi(b)) = \varphi(a(h\cdot b)).
		\end{equation}
		follows by the same reasoning of \cite{alves1} for unital algebras.	Similarly, if the partial action is symmetrical then 
		\begin{equation} \label{equation.varphi.induced.symmetrical}
			(h\triangleright\varphi(a))\varphi(b) = \varphi((h\cdot a)b).
		\end{equation}
		Hence 
		\[
		(k \triangleright \varphi(a))(h\triangleright\varphi(b)) = k_1  \triangleright ( \varphi(a))(S(k_2) h\triangleright\varphi(b))
		= k_1  \triangleright ( \varphi(a (S(k_2) h \cdot b)),
		\]
		which shows that $B$ is a subalgebra of $\HOM(H,A)$; clearly it is an $H$-submodule as well. 
		
		Equation \eqref{equation.varphi.induced} also shows that $\varphi(A)$ is a right ideal of $B$ and, if the partial action is symmetrical, then \eqref{equation.varphi.induced.symmetrical} implies that  $\varphi(A)$ is a left ideal of $B$.
		
		Now assume that $\varphi(A)$ is a left ideal of $B$. Then, for $k\in H$, $a\in A$, for every $b\in A$, we have that $(k\triangleright\varphi(a))\varphi(b)\in\varphi(A)$. Hence, for every $c\in A$,
		\begin{eqnarray*}
			\varphi(c)\varphi((k\cdot a)b)(h)&=&\sum(h_{(1)}\cdot c)(h_{(2)}\cdot(k\cdot a))(h_{(3)}\cdot b)\\
			&=&\sum(h_{(1)}\cdot(c(k\cdot a)))(h_{(2)}\cdot b)\\
			&=&\sum(h_{(1)}\cdot c)(h_{(2)}k\cdot a)(h_{(3)}\cdot b)\\
			&=&\varphi(c)(k\triangleright \varphi(a))\varphi(b)(h).
		\end{eqnarray*}
		Now, since $r(A)=0$ and $\varphi$ is a monomorphism of algebras, we have that $r(\varphi(A))=0$. Then, as the previous calculation holds for every $c\in A$, we have that $$\varphi((k\cdot a)b)(h)=(k\triangleright \varphi(a))\varphi(b)(h),$$ i.e., the partial action is symmetrical.
	\end{proof}
	\begin{theorem} \label{theorem.minimal.globalization} Let $\cdot: H \otimes A \longrightarrow A$ be a symmetrical partial action. Then, the pair $(B,\varphi)$, where $B=H\triangleright \varphi(A)$, is a minimal globalization for $\cdot$.
	\end{theorem} 
	\begin{proof}
		The fact that $(B,\varphi)$ satisfies item (1)-(4) of Definition \ref{def3} follows from similar calculations for the unital case. Here we only need to prove that there exists a projection as in Proposition \ref{proposition.projection}. In this case, the projection $\pi: H\triangleright \varphi(A)\to H\triangleright\varphi(A)$ given by $\pi(h\triangleright\varphi(a))=\varphi(h\cdot a)$, for all $h\in H$ and $a\in A$, is well-defined even if $r(A)\neq 0$ and $l(A)\neq 0$. To prove this, note that if $\sum h_i\triangleright \varphi(a_i)=\sum k_j\triangleright\varphi(b_j)$, then applying these morphisms in $1_H$ we obtain $\sum h_i\cdot a_i=\sum k_j\cdot b_j$, hence $\varphi(\sum h_i\cdot a_i)=\varphi(\sum k_j\cdot b_j)$.
	\end{proof}
	\begin{definition} The pair $(B,\varphi)$ as before will be called the \textit{standard} globalization of the symmetrical partial action $\cdot: H \otimes A \longrightarrow A$.
	\end{definition}
	
	We already know that when we consider only unital algebras, there exists only one minimal globalization, up to isomorphism \cite{alves1}. The next theorem shows that this also holds for a broader class of associative algebras.

	\begin{theorem} \label{thm.lifts}Let $H$ be a Hopf algebra, let $A$ and $A'$ be partial $H$-module algebras with symmetrical partial actions, and let $(B,\theta)$, $(B',\theta')$ be globalizations for $A$ and $A'$ respectively. Assume that $r(A')=0$ or $l(A')=0$. If  $(B,\theta)$ is a minimal globalization and $\alpha : A' \to A$ is a morphism of partial actions, then there exists a unique morphism of $H$-module algebras $\Phi : B' \to B$ such that 
		$\Phi (\theta'(a)) = \theta(\alpha(a))$ for each $a \in A'$.
	\end{theorem}
	
	\begin{proof}  Let the global $H$-actions on $B'$ and $B$ be denoted by ``$\blacktriangleright$'' and ``$\triangleright$'' respectively. 
		If $\Phi : B' \to B$ satisfies the conditions above then clearly it is unique, since $B' = H \blacktriangleright \theta'(A')$ and hence
		\[
		\Phi \left(\sum_{i=1}^nh_i\blacktriangleright \theta'(a_i)\right) =
		\sum_{i=1}^nh_i\triangleright  \theta (\alpha(a_i));
		\]
		then we will define a linear map $\Phi : B' \to B$ by the previous expression. 
		
		First we will prove that $\Phi$ is well-defined. It is enough to show that if $x=\sum_{i=1}^nh_i\blacktriangleright \theta'(a_i)=0$ then  
		${\sum_{i=1}^nh_i\triangleright \theta(\alpha(a_i))=0}$.
		For every $c\in A'$ and $k\in H$,
		\[	
		0=\theta'(c)(k\blacktriangleright x) =\theta'(c)(\sum_{i=1}^nkh_i\blacktriangleright \theta'(a_i)) = \theta'(c\sum_{i=1}^nkh_i\cdot a_i).
		\]
		Since $\theta': A' \to B'$ is injective, we have that $\sum_{i=1}^nc(kh_i\cdot a_i)=0$ for every $c\in A'$, $k\in H$. 
		Assuming that $r(A')=0$ it follows that $\sum_{i=1}^nkh_i\cdot a_i=0$ for every $k\in H$, hence also 
		\[
		\sum_{i=1}^nkh_i\cdot \alpha(a_i)= \alpha \left( \sum_{i=1}^nkh_i\cdot a_i \right) = 0
		\] in $A$; since $(B,\theta)$ is a minimal globalization, we have that $\sum_{i=1}^nh_i\triangleright \theta(\alpha(a_i))=0$, thus showing that $\Phi: B' \to B$ is well-defined. The case where $l(A')=0$ follows by a similar calculation. 
		
		$\Phi$ is a morphism of $H$-modules by construction, and the proof that $\Phi$ is a morphism of (nonunital) algebras is the same as in \cite{alves1}, which is based essentially on the equality
		\begin{eqnarray*}
			(h\triangleright\theta(a))(k\triangleright\theta(b))=\sum h_{(1)}\triangleright\theta(a(S(h_{(2)})k\cdot b)),
		\end{eqnarray*}
		that is a consequence of axiom $3)$ of the definition of globalization. 
	\end{proof}
	
	\begin{corollary}\label{r(A).l(A).trivial.unique.glob.minimal}
		Let $H$ be a Hopf algebra and $A$ a partial $H$-module algebra with symmetrical partial action. If $r(A)=0$ or $l(A)=0$, there exist only one minimal globalization of the partial action on $A$ up to isomorphism.
	\end{corollary}
	\begin{proof}
		Let $(B,\theta)$, $(B', \theta')$ be two minimal globalizations of the partial $H$-module algebra $A$. By the previous result the identity map of $A$ lifts to morphisms of $H$-module algebras $\Phi : B' \to B$ and $\Psi : B \to B'$ and, by uniqueness, it follows that $\Psi \circ \Phi = I_{B'}$ and $\Phi \circ \Psi = I_B$.
	\end{proof}
	
	\begin{theorem}\label{funtorial} Let $H$ be a Hopf algebra, let $A$ and $A'$ be partial $H$-module algebras with symmetrical partial actions, and let $(B,\varphi)$, $(B',\varphi')$ be the standard globalizations for $A$ and $A'$, respectively. If $\alpha : A' \to A$ is a morphism of partial actions, then there exists a unique morphism of $H$-module algebras $\Phi : B' \to B$ such that 
		$\Phi (\varphi'(a)) = \varphi(\alpha(a))$ for each $a \in A'$. 
	\end{theorem}
	\begin{proof} If there exists such a morphism, it should be as in Theorem \ref{thm.lifts}. But in this case, $\Phi$ is well-defined even if $r(A')$ and $l(A')$ are nontrivial. In fact, if $\sum_ih_i\triangleright\varphi'(a_i)=0$, then $(\sum_ih_i\triangleright\varphi'(a_i))(k)=0$ for every $k\in H$, i.e., $\sum_ikh_i\cdot a_i=0$. Since standard globalizations are minimal and $\alpha(\sum_ikh_i\cdot a_i)=\sum_ikh_i\cdot \alpha(a_i)=0$, we have that $\Phi(\sum_ih_i\triangleright\varphi'(a_i))=\sum_ih_i\triangleright\varphi(\alpha(a_i))=0$.
	\end{proof}
	
	\begin{remark}
		It follows from Theorem \ref{funtorial} that the standard globalization defines a functor from partial $H$-module algebras to the category of $H$-module algebras. Moreover, it follows from Lemma \ref{r(B).r(A).minimal} that if $H$ has bijective antipode, then the standard globalization defines a functor from partial $H$-module algebras with trivial right annihilator to the category of $H$-module algebras with trivial right annihilator.
	\end{remark}
	
	Note that if $r(A)\neq0$ and $l(A)\neq 0$, there always exists at least one minimal globalization, the standard globalization. By similar calculations as in Theorems \ref{thm.lifts} and \ref{funtorial}, we know that there always exists at least an \emph{epimorphism} from the standard globalization to any other minimal globalization.

	\section{Partial Actions on Algebras with local units \label{section.local.units}}
	\subsection{Categorizable partial actions}
	Now that we have established a good definition for partial actions on associative algebras in general, we will restrict our study to algebras with local units and associate the obtained results to the known results about partial actions on categories.
	We begin by presenting an equivalent definition of partial action in the context of algebras with local units.
	
	\begin{definition}
		A $\Bbbk$-category is a category $\mathcal{C}$ such that every set of morphisms ${}_y\mathcal{C}_x= \HOM_{\mathcal{C}}(x,y)$ is a vector space over $\Bbbk$ and the composition maps $\circ :{}_z\mathcal{C}_y \times {}_y\mathcal{C}_x \to {}_z\mathcal{C}_x $ are $\Bbbk$-bilinear for every $x,y,z \in \mathcal{C}_{0}$.
	\end{definition}
	
	\begin{definition}
		Let $A$ be a $\Bbbk$-algebra. 
		The set $S=\{e_\lambda\}_{\lambda\in\Lambda}$ is called a \textit{system of local units} of $A$ if $e_\lambda^2=e_\lambda$, for every $\lambda\in\Lambda$, and for every finite subset $F$ of $A$, there exists $e_\alpha\in S$ such that $e_\alpha a=ae_\alpha=a$ for every $a\in F$. In this case we say that $A$ is an \textit{algebra with local units}.
	\end{definition}

	\begin{remark}
		Let $A$ be an algebra and $S$ be a system of local units. 
		We recall that there is a partial order on $S$ defined by 
		\[
		e_\alpha \leq e_\beta \iff e_\alpha e_\beta =  e_\beta e_\alpha = e_\alpha.
		\]
		By the definition of system of local units, 
		given any finite set $F$ of local units, there exists a local unit $e_\alpha$ such that $e_\lambda e_\alpha = e_\lambda e_\beta = e_\lambda$ for every $e_\lambda \in F$, i.e., $ e_\lambda \leq e_\alpha$ for all $e_\beta \in F$.
	\end{remark}
	
	\begin{proposition}Let $H$ be a Hopf algebra, $A$ an algebra with local units with system of local units $S=\{e_{\lambda}\}_{\lambda \in \Lambda}$, and $\cdot : H\otimes A \longrightarrow A$ a linear map. Then this map defines a partial $H$-action if and only if, for all $a,b \in A$, $h,k \in H$, we have:
		\begin{enumerate}[\normalfont(1)]
			\item $1_H\cdot a=a$;
			\item $h\cdot(ab)=\sum(h_{(1)}\cdot a)(h_{(2)}\cdot b)$;
			\item $h\cdot (k\cdot a)=\sum(h_{(1)}\cdot e_{\alpha})(h_{(2)}k\cdot a)$, for every $e_{\alpha} \in S$ such that $e_{\alpha}(k\cdot a)=k\cdot a$.
		\end{enumerate}
		Additionally, this partial partial action is symmetrical if and only if $$h\cdot (k\cdot a)=\sum(h_{(1)}k\cdot a)(h_{(2)}\cdot e_{\beta})$$ for every $e_{\beta} \in S$ such that $(k\cdot a)e_{\beta}=k\cdot a$.
	\end{proposition}
	It is clear that if the linear map $\cdot:H\otimes A\to A$ is a partial action then (1), (2) and (3) hold. Conversely, assuming these properties,
	item (2) of Definition 2 follows from the fact that if $A$ is a partial $H$-module algebra with local units, for every $a,b\in A$, $h,k\in H$, there exists $e_\alpha \in S$ such that $be_\alpha=b$ and $e_\alpha(k\cdot a)=k\cdot a$, then:
	\begin{eqnarray*}
		h\cdot (b(k\cdot a))&=&\sum (h_{(1)}\cdot b)(h_{(2)}\cdot(k\cdot a))\\
		&=&\sum (h_{(1)}\cdot b)(h_{(2)}\cdot e_\alpha)(h_{(3)}k\cdot a)\\
		&=&\sum (h_{(1)}\cdot be_\alpha)(h_{(2)}k\cdot a)\\
		&=&\sum (h_{(1)}\cdot b)(h_{(2)}k\cdot a).
	\end{eqnarray*}
	This calculation illustrates that if $\cdot : H\otimes A \longrightarrow A$ is a partial action for some s.l.u. (system of local units) $S$ of $A$, then it will be a partial action for any other s.l.u. of $A$.
	
	Now we will see a particular class of partial actions on algebras with local units, the \textit{categorizable} ones. This terminology is motivated by the fact that given a categorizable partial action, then we can induce a partial action on an associated category.
	
	We begin by recalling the construction of the algebra with local units associated to a $\Bbbk$-category and the definition of partial $H$-module category.

	\begin{definition} Let $\mathcal{C}$ be a $\Bbbk$-category. We will denote by $a(\mathcal{C})$ the algebra consisting of elements of the form $(\,_yf_x)_{x,y\in\mathcal{C}_0}$, with finite $_yf_x\neq 0$, where the multiplication is given by matrix multiplication and the composition of $\mathcal{C}$.
	\end{definition}

	Note that $(\,_yf_x)_{x,y\in\mathcal{C}_0}$ is the matrix notation of the elements of the algebra $a(\mathcal{C})$ which, as a vector space, is the direct sum $\bigoplus_{x,y\in\mathcal{C}_0}\,_yC_x$.
	It is well-known that $a(C)$ is an algebra with local units: in fact, one has the canonical system $S_0$ of local units whose elements are finite sums of distinct elements $e_{xx}$, which are given by $_x1_x$ at the position $(x,x)$ and zero at every other position. In this manner, for every finite subset $P \subset \mathcal{C}_0$, $P \neq \varnothing$, we may associate the local unit $e_P = \sum_{x \in P} e_{x,x}$.
	
	\begin{definition}\cite{cibils-solotar-galois}. An action of a Hopf algebra $H$ on a $\Bbbk$-category $\mathcal{C}$ is a family of linear maps 
		$\triangleright=\{\triangleright_{(x,y)} : H \otimes \,_y\mathcal{C}_x \longrightarrow \,_y\mathcal{C}_x\}_{x,y \in \mathcal{C}_0}$, such that, for every $x,y,z \in \mathcal{C}_0$, $\,_yf_x\in\,_y\mathcal{C}_x$ and $\,_zg_y\in\,_z\mathcal{C}_y$, we have that:
		\begin{enumerate}[\normalfont(1)]
			\item $1_H\triangleright_{(x,y)} \,_yf_x=\,_yf_x$;
			\item $h\triangleright_{(x,z)}(\,_zg_y\circ\,_yf_x)=\sum(h_{(1)}\triangleright_{(y,z)} \,_zg_y)\circ(h_{(2)}\triangleright_{(x,y)} \,_yf_x)$;
			\item $h\triangleright_{(x,y)}(k\triangleright_{(x,y)} \,_yf_x)=
			(h k)\triangleright_{(x,y)} \,_yf_x
			$;
			\item $ h \triangleright_{(x,x)} \,_x1_x = \varepsilon(h)\, _x1_x$.
		\end{enumerate}
		We also say that $\mathcal{C}$ is an $H$-module category.
	\end{definition}    
	
	Combining the concept of $H$-module category from \cite{cibils-solotar-galois} with the definition of partial $H$-module algebra from \cite{caenepeel}, we obtain the following definition from \cite{alves2}. 
	\begin{definition}[\cite{alves2}] A partial action of $H$ on a $\Bbbk$-category $\mathcal{C}$ is a family of linear maps $\triangleright=\{\triangleright_{(x,y)} : H \otimes \,_y\mathcal{C}_x \longrightarrow \,_y\mathcal{C}_x\}_{x,y \in \mathcal{C}_0}$, such that, for every $x,y,z \in \mathcal{C}_0$, $\,_yf_x\in\,_y\mathcal{C}_x$ and $\,_zg_y\in\,_z\mathcal{C}_y$, we have that:
		\begin{enumerate}[\normalfont(1)]
			\item $1_H\triangleright_{(x,y)} \,_yf_x=\,_yf_x$;
			\item $h\triangleright_{(x,z)}(\,_zg_y\circ\,_yf_x)=\sum(h_{(1)}\triangleright_{(y,z)} \,_zg_y)\circ(h_{(2)}\triangleright_{(x,y)} \,_yf_x)$;
			\item $h\triangleright_{(x,y)}(k\triangleright_{(x,y)} \,_yf_x)=\sum(h_{(1)}\triangleright_{(y,y)} \,_y1_y)\circ(h_{(2)}k\triangleright_{(x,y)} \,_yf_x)$;
			\item If additionally $h\triangleright_{(x,y)}(k\triangleright_{(x,y)} \,_yf_x)=\sum(h_{(1)}k\triangleright_{(x,y)} \,_yf_x)\circ(h_{(2)}\triangleright_{(x,x)} \,_x1_x)$, $\triangleright$ is called a symmetrical partial action.
		\end{enumerate}
		In this case, we say that $\mathcal{C}$ is a (symmetrical) partial $H$-module category.
	\end{definition}
	Here, different from \cite{alves2}, we present the symmetrical condition as an additional property.
	
	\begin{definition} If a linear map $\cdot : H\otimes A \longrightarrow A$ is a (symmetrical) partial action and there exists some s.l.u. $S=\{e_{\lambda}\}_{\lambda \in \Lambda}$ such that $H\cdot e_\alpha \subseteq e_\alpha A e_\alpha$, for every $e_\alpha \in S$, we will say that this partial action is an \textit{$S$-categorizable (symmetrical) partial action}.
	\end{definition}

	\begin{definition}
		Let $A$ be an algebra and $S$ be a system of local units for $A$. The category $\mathcal{C}^{S}(A)$ has $\Lambda$ as its set of objects. For any two objects $\alpha,\beta \in \Lambda$ we associate the vector space of morphisms  $\mathcal{C}^{S}(A)(\alpha,\beta)=\,_\beta\mathcal{C}^{S}(A)_\alpha=e_\beta Ae_\alpha$, and composition is just the restriction of multiplication of $A$.  
	\end{definition}
	
	Note that $\mathcal{C}^{S}(A)$ is a $\Bbbk$-category. The next result explains the name ``$S$-categorizable partial action''. 
	
	\begin{proposition}
		Let $\mathcal{C}$ be a $\Bbbk$-category and let $A$ be an algebra with system of local units $S$. 
		\begin{enumerate}[\normalfont(1)]
			\item Every (symmetrical) partial $H$-module category $\mathcal{C}$ induces canonically a (symmetrical) categorizable partial action of $H$ on the algebra $a(\mathcal{C})$ for the canonical system of local units of $a(\mathcal{C})$.
			\item Every (symmetrical) $S_0$-categorizable partial action of $H$ on $a(\mathcal{C})$ induces a (symmetrical) partial $H$-action on the category $\mathcal{C}^{S_0}(A)$.
		\end{enumerate}
	\end{proposition}
	
	\begin{proof}

		Consider a (symmetrical) partial $H$-action on $\mathcal{C}$ given by the family of linear maps $\triangleright=\{\triangleright_{(x,y)}:H\otimes \,_y\mathcal{C}_y\to \,_y\mathcal{C}_y\}$. The corresponding partial action on $a(\mathcal{C})$ is the linear map $\cdot: H\otimes a(\mathcal{C})\to a(\mathcal{C})$ given by $h\cdot (\,_yf_x)_{x,y}= (h\triangleright_{(x,y)}\,_yf_x)_{x,y}$.
		
		In fact, clearly $1\cdot (\,_yf_x)_{x,y}=(\,_yf_x)_{x,y}$ and
		\begin{eqnarray*}
			h\cdot(_yf_x)_{x,y}(_yg_x)_{x,y}=\sum (h_{(1)}\cdot(_yf_x)_{x,y})(h_{(2)}\cdot(_yg_x)_{x,y}).
		\end{eqnarray*}
		For the third axiom of partial action, note that
		\begin{eqnarray*}
			h\cdot k\cdot(_yf_x)_{x,y}&=&(h\triangleright_{(x,y)} k\triangleright_{(x,y)}\,_yf_x)_{x,y}\\
			&=&(h\triangleright_{(y,y)} (\,_y1_y(k\triangleright_{(x,y)}\,_yf_x)))_{x,y}\\
			&=&\sum((h_{(1)}\triangleright_{(y,y)} \,_y1_y)(h_{(2)}k\triangleright_{(x,y)}\,_yf_x))_{x,y}\\
			&=&\sum(h_{(1)}\cdot \sum_{y;\,_yf_x\neq 0} E_{yy})(h_{(2)}k\cdot (_yf_x)_{x,y}),
		\end{eqnarray*}
		i.e., there exists at least one subset $P\in\mathcal{C}_0$ such that $e_P(k\cdot  (_yf_x)_{x,y})=k\cdot (_yf_x)_{x,y}$ and $h\cdot k\cdot(_yf_x)_{x,y}=(h_{(1)}\cdot e_P)(h_{(2)}k\cdot (_yf_x)_{x,y})$. Analogously, if the original partial action is symmetrical, there exists at least one subset $Q\in\mathcal{C}_0$ such that $(k\cdot  (_yf_x)_{x,y})e_Q=k\cdot (_yf_x)_{x,y}$ and $h\cdot k\cdot(_yf_x)_{x,y}=\sum(h_{(1)}k\cdot (_yf_x)_{x,y})(h_{(2)}\cdot e_Q)$. 
		
		Now, note that whenever $e_U,e_V\in S_0$ are such that $e_U \leq e_V$
		, i.e. $U\subseteq V$, we have that $e_V-e_U = e_{V \setminus U}\in S$. Moreover, given a Hopf algebra $H$ and an algebra $A$ with s.l.u. $S=\{e_{\lambda}\}_{\lambda \in \Lambda}$, for every linear map $\cdot:H\otimes A\to A$, where $H\cdot e_\lambda\subseteq e_\lambda Ae_\lambda$, and $e_\beta - e_\alpha \in S$ whenever $e_\alpha\leq e_\beta$, we have that $\sum(h_{(1)}\cdot e_{\lambda})(h_{(2)}k\cdot a)=\sum(h_{(1)}\cdot e_{\gamma})(h_{(2)}k\cdot a)$ for every $e_{\lambda},e_\gamma \in S$ such that $e_{\lambda}a=a=e_\gamma a$. In fact, let $e_\alpha, e_\beta$ be local units such that $e_\alpha a=a=e_\beta a$, then there exists a local unit $e_\gamma$ such that $e_\alpha, e_\beta\leq e_\gamma$. Hence
		\begin{eqnarray*}
			\sum(h_{(1)}\cdot e_\gamma)(h_{(2)}k\cdot a)&=&\sum(h_{(1)}\cdot (e_\gamma-e_\alpha+e_\alpha))(h_{(2)}k\cdot a)\\
			&=&\sum(h_{(1)}\cdot (e_\gamma-e_\alpha))(h_{(2)}k\cdot a)+(h_{(1)}\cdot e_\alpha)(h_{(2)}k\cdot a)\\
			&=&\sum(h_{(1)}\cdot (e_\gamma\!-\!e_\alpha))(e_\gamma\!-\!e_\alpha)(e_\alpha)(h_{(2)}k\cdot a)\!+\!(h_{(1)}\cdot e_\alpha)(h_{(2)}k\cdot a)\\
			&=&\sum(h_{(1)}\cdot e_\alpha)(h_{(2)}k\cdot a).
		\end{eqnarray*}
		Analogously, $\sum(h_{(1)}\cdot e_\gamma)(h_{(2)}k\cdot a)=\sum(h_{(1)}\cdot e_\beta)(h_{(2)}k\cdot a)$.
		
		This is useful when one wants to verify whether some linear map $\cdot : H\otimes A \longrightarrow A$ is a (symmetrical) partial action or not, because we will only need to calculate the third axiom of partial actions for one local unit that satisfies the required property.
		
		Consequently, one can conclude that the linear map $\cdot: H\otimes a(\mathcal{C})\to a(\mathcal{C})$ mentioned before is, in fact, a (symmetrical) partial action.
		
		Conversely, given a (symmetrical) partial action $\cdot : H\otimes A \longrightarrow A$, if there exist some s.l.u. $S=\{e_{\lambda}\}_{\lambda \in \Lambda}$ such that $h\cdot e_\alpha \in e_\alpha A e_\alpha$, for all $\alpha\in \Lambda$, $h\in H$, we have that
		\begin{eqnarray*} 
			e_\alpha(h\cdot e_\alpha ae_\beta)e_\beta&=&\sum e_\alpha(h_{(1)}\cdot e_\alpha)(h_{(2)}\cdot e_\alpha ae_\beta)(h_{(3)}\cdot e_\beta)e_\beta\\
			&=&\sum (h_{(1)}\cdot e_\alpha)(h_{(2)}\cdot e_\alpha ae_\beta)(h_{(3)}\cdot e_\beta)\\
			&=&h\cdot e_\alpha ae_\beta,
		\end{eqnarray*}
		for all $e_\alpha, e_\beta\in S$, i.e., $H\cdot e_\alpha Ae_\beta \subseteq e_\alpha Ae_\beta$. 
		Hence, we can induce a (symmetrical) partial action of $H$ on the category $\mathcal{C}^{S_0}(A)$ by 
		\[
		\triangleright_{(\beta,\alpha)} : H \otimes \mathcal{C}^{S_0}(A)(\alpha,\beta) \to \mathcal{C}^{S_0}(A)(\alpha,\beta), \ \ 
		h \triangleright_{(\beta,\alpha)} e_\beta a e_\alpha = h \cdot  e_\beta a e_\alpha.
		\]
	\end{proof}

	\subsection{Globalization of partial actions on algebras with local units}
	
	We already know that every symmetrical partial action on any algebra has a globalization, so it is also true for algebras with local units.  Here we will present an equivalent definition of globalization, regarding only algebras with local units, which will be useful for relating globalizations of categorizable partial actions and globalizations of the induced partial action on the category associated to the considered system of local units.
	
	But first, we will follow the idea of Proposition \ref{proposition.projection} to provide a sufficient condition to determine if we can induce a partial action on an ideal with local units of an $H$-module algebra.
	
	\begin{definition}
		Let $A$ be an algebra with local units and $S$ be an s.l.u. of $A$. A subfamily $T\subseteq S$ which is also an s.l.u. of $A$, will be called a subsystem of local units of $S$.
	\end{definition}
	
	The goal of the next two results is to present an equivalent definition of globalization (see Definition \ref{def3}) in the particular case where the algebra has local units.

	\begin{proposition} \label{proposition.induced.action.local.units}
		Let $B$ be an $H$-module algebra with action $\triangleright$ and $A$ be an ideal of $B$. If $A$ has an s.l.u. $S=\{e_\lambda\}_{\lambda\in \Lambda}$ and for every $h\in H$, $a\in A$ there exist subsystems of local units $L(h,a)$ and $R(h,a)$ of $S$ such that $e(h\triangleright a)=(h\triangleright a)f$, for all $e\in L(h,a)$, $f\in R(h,a)$, then the linear map $h\rightharpoonup a=e(h\triangleright a)$, where $e\in L(h,a)$, determines a partial action of $H$ on $A$.	
	\end{proposition}
	\begin{proof}
		First note that if $e,e' \in L(h,a)$ and $ee'=e'e$, then $e'e, e e' \in L(h,a)$, because given $f \in R(h,a)$ we have $$ee' (h \triangleright a) = e (h \triangleright a) f = (h \triangleright a) f^2 = (h \triangleright a)f.$$ 
		Note also that if $e\in L(h,a)$ then we may enlarge $L(h,a)$ by adding each local unit $x \in S$ such that $x\geq e$. In fact, for every local unit $x\geq e$, since $x(h\triangleright a)\in A$, $A$ is an ideal of $B$ and $R(h,a)$ is a subsystem of local units, there exists $f\in R(h,a)$ such that $x(h\triangleright a)=x(h\triangleright a)f=xe(h\triangleright a)=e(h\triangleright a)$. Hence we may assume that if $e \in L(h,a)$ and $x \geq e$ then $x \in L(h,a)$. It follows that for every $h_1,\ldots,h_n\in H$ and $a_1,\ldots,a_n\in A$ we can construct a family $\{L_i=L(h_i,a_i)\}_{i=1}^{n}$ whose intersection $L=\cap_i L_i$ is also a subsystem of local units; in fact, we may consider $L_i=L$ for every $i=1,\ldots,n$. Correspondingly, we may take $R(h_i,a_i)=R$, for some subsystem of local units $R$, for every $i=1,\ldots,n$,  obtaining the pair $L,R$ of subsystems of local units such that $e (h_i \triangleright a_i) = (h_i \triangleright a_i) f$ for every $e \in L, f \in R$, and every $i=1,\ldots,n$. 
		
		Following Proposition \ref{proposition.projection}, consider the linear map $\pi: H\triangleright A\to A$ given by $\pi(h\triangleright a)=e(h\triangleright a)$, where $e\in L(h,a)$. This map is well-defined: if $e, e'\in L(h,a)$ then, taking an arbitrary $f \in R(h,a)$,
		\[
		e' (h \triangleright a) = (h \triangleright a) f = e (h \triangleright a). 
		\]
		
		By construction, $\pi$ is a projection of $H \triangleright A$ onto $A$. Moreover, given $a, b \in B$, $h,k \in H$, we may choose $e_1 \in L(h,(a\pi(k\triangleright b)))$ and $e_2\in L(k,b)$ such that $ae_2=a$, thus obtaining
		\[ 	  
		\pi(h\triangleright (a\pi(k\triangleright b))) 
		=e_1(h\triangleright (ae_2(k\triangleright b)))
		=\sum e_1(h_{(1)}\triangleright a)(h_{(2)}k\triangleright b).\]
		Now, it follows from the fact that $A$ is an ideal of $B$ and the previous discussions on the intersections of $L(h_i,b_i)$ that we may choose a local unit $e_3$ such that 
		\[ \sum e_1(h_{(1)}\triangleright a)(h_{(2)}k\triangleright b)  =   
		\sum e_1(h_{(1)}\triangleright a)e_3(h_{(2)}k\triangleright b) 
		=   \sum \pi(h_{(1)}\triangleright a)\pi(h_{(2)}k\triangleright b), 
		\]
		and it follows from Proposition \ref{proposition.projection} that $h \rightharpoonup a = \pi (h \triangleright a)$ defines a partial action on $A$. 
	\end{proof}
	
	\begin{proposition} \label{definition.globalization.local.units} Let $\cdot: H \otimes A \longrightarrow A$ be a symmetrical partial action where $A$ is an algebra with s.l.u. $S=\{e_\lambda\}_{\lambda\in\Lambda}$. Then the pair $(B,\theta)$ is a globalization for $\cdot$ if and only if:
		\begin{enumerate}[\normalfont(1)]
			\item $B$ is an $H$-module algebra (possibly nonunital), with action $\triangleright$;
			\item $\theta: A \longrightarrow B$ is an algebra monomorphism;
			\item $\theta(A)$ is an ideal of $B$;
			\item $\theta(e_\alpha)(h\triangleright \theta(a))=(h\triangleright \theta(a))\theta(e_\beta)$, for every $e_\alpha,e_\beta \in S$, $a\in A$, such that $e_\alpha(h\cdot a)=h\cdot a=(h\cdot a)e_\beta$, and $\theta(h\cdot a)=\theta(e_\alpha)(h\triangleright \theta(a))$, for every $e_\alpha \in S$ such that $e_\alpha(h\cdot a)=h\cdot a$;
			\item $B=H\triangleright\theta(A)$.
		\end{enumerate}
	\end{proposition}
	
	Note that when $A$ has local units and $(B,\theta)$ is a globalization, we have that 
	\begin{eqnarray*}
		\theta(h\cdot a)&=&\theta(e)(h\triangleright \theta(a))\\
		&=&(h\triangleright \theta(a))\theta(f),
	\end{eqnarray*}
	for every $e,f \in A$ such that $e(h\cdot a)=h\cdot a=(h\cdot a)f$. In this case the subsystems considered are $L(h,a)=\{e_\alpha \in S; \,\, e_\alpha(h\cdot a)=h\cdot a \}$ and $R(h,a)=\{e_\alpha \in S;\,\, (h\cdot a)e_\alpha=h\cdot a \}$, and then item (4) means that $\theta$ is a morphism of partial actions.

	\subsection{Globalization of partial gradings}
	
	Let $G$ be a group. Partial $G$-gradings have been investigated in \cite{alves1, alves2, alves4, dilations}. In this section we will use the results and definitions presented until now to find the minimal globalization of a good partial $G$-grading on $FMat_\mathbb{N}(\Bbbk)$, the algebra of $\mathbb{N} \times \mathbb{N}$ matrices with finitely many nonzero entries. 
	
	Following \cite{ion}, a \textit{good grading} of the matrix algebra $M_n(\Bbbk)$ is a $G$-grading where each elementary matrix $E_{i,j}$ is homogeneous. In this same paper all good gradings are classified \cite[Proposition 2.1, Corollary 2.2]{ion}: every such grading is determined by the $(n-1)$-tuple of elements of $G$ $(deg(E_{1,2}, \ldots, deg(e_{n-1,n})))$. Alternatively, every such grading is associated to a sequence $(g_2, \ldots, g_n) \in G^{n-1}$ by defining $g_1 = 1$ and $deg(E_{i,j}) = g_i g_j^{-1}$. It is shown that if $G$ is torsionfree then every $G$-grading of $M_n(\Bbbk)$ is a good grading (Theorem 1.4, Corollary 1.5). Moreover, if $M_n(\Bbbk)$ has a $G$-grading where some matrix $E_{i,j}$ is homogeneous then there is an isomorphism of $G$-graded algebras with $M_n(\Bbbk)$ endowed with a good $G$-grading (Corollary 1.6).

	It is well-known that $G$-gradings of an algebra $A$ correspond to right $\Bbbk G$-coactions on $A$ and, when $G$ is a finite group, these in turn are the same as left $(\Bbbk G)^*$-actions on $A$, which is a particular case of the right $H^*$-coaction / left $H$-action correspondence for a finite-dimensional Hopf algebra $H$. The same still holds for partial $H^*$-actions and $H$-coactions and, based on this result, we will work with partial gradings by a finite group $G$, i.e., partial (right)  coactions of $\Bbbk G$, as partial (left)   $(\Bbbk G)^*$-actions.

	Good partial $G$-gradings were introduced for the matrix algebra $M_n(\Bbbk)$ in \cite{alves4}, and here we extend this idea naturally to the algebra of finite matrices $FMat_\mathbb{N}(\Bbbk)$.

	\begin{definition}
		Let $G$ be a finite group. A good partial $G$-grading of $FMat_{\mathbb{N}}(\Bbbk)$ is a partial action of $(\Bbbk G)^*$ on $FMat_\mathbb{N}(\Bbbk)$ such that each elementary matrix $E_{ij}$ is a simultaneous eigenvector for all the elements of the canonical base $\{p_g ; g \in G\}$ of $(\Bbbk G)^*$. 
	\end{definition}

	In what follows, when we consider the matrix algebra $FMat_{\mathbb{N}}(\Lambda)$, where $\Lambda$ is an associative algebra, we will write $(\alpha)E_{ij}$ to denote the matrix which has the element $\alpha \in \Lambda$ at the $(i,j)$-entry of the matrix, and has zero in all other entries; in particular, $E_{ij} = (1_{\Lambda})E_{ij}$.
	
	Let $G$ be a finite abelian group such that $char(\Bbbk)\nmid |G|$, $H=(\Bbbk G)^*$ and $A=FMat_{\mathbb{N}}(\Bbbk)$. Suppose that the map $\cdot: H\otimes A\to A$ is a good partial $G$-grading on $A$. Then, by \cite{alves2} and \cite{alves4}, there exists a subgroup $L$ of $G$ and a family $\{t_{ij}\}_{i,j\in \mathbb{N}}\subset G$ such that 
	\begin{equation} \label{equation.tij}
		t_{ik}t_{kj}L=t_{ij}L    
	\end{equation}
	and 
	\begin{equation}\label{equation.partial.grading}
		p_g\cdot E_{ij}=\delta_{gL,t_{ij}L}\frac{1}{|L|}E_{ij}. 
	\end{equation}
	Let us describe a globalization for this partial grading. 
	
	Equality \eqref{equation.tij} implies that the subspace $B=\bigoplus_{i,j\in \mathbb{N}}B_{ij}$ of $FMat_{\mathbb{N}}(\Bbbk G)$,  where $B_{ij}=\bigoplus_{g\in t_{ij}L} \Bbbk g$, is a subalgebra of $FMat_{\mathbb{N}}(\Bbbk G)$. Let $\theta : A \to B$ be the map
	\[
	E_{ij}\mapsto \left(\delta_{(x,y),(i,j)}\frac{1}{|L|}\sum_{g\in t_{ij}L}g\right)_{x,y\in \mathbb{N}}=(\frac{1}{|L|}\sum_{g\in t_{ij}L}g)E_{ij}.
	\]
	The natural action of $H$ in $\Bbbk G$ given by $p_g\rightharpoondown h=\delta_{g,h}h$ induces a structure of $H$-module algebra on $FMat_{\mathbb{N}}(\Bbbk G)$, and $B$ is an $H$-submodule algebra of $FMat_{\mathbb{N}}(\Bbbk G)$. The $H$-action on $B$ is given explicitly by 
	$$p_g\triangleright (\sum_{h\in t_{ij}L}\alpha_hh)E_{ij}=(\delta_{gL,t_{ij}L}\alpha_gg)E_{ij}.$$ 
	
	We claim that $(B,\theta)$ is a minimal globalization for the good partial $G$-grading on $A$. In fact, it is not hard to show that $B=H\triangleright \theta(A)$ and $\theta$ is a monomorphism of algebras. To prove that $\theta(A)$ is an ideal of $B$, note that for every $g\in G$, $i,j,k\in\mathbb{N}$, we have
	\begin{eqnarray*}
		\theta(E_{ik})(p_g\triangleright \theta(E_{kj}))&=&\delta_{gL,t_{kj}L}\frac{1}{|L|}\theta(E_{ij}),\\
		(p_g\triangleright \theta(E_{ik}))\theta(E_{kj})&=&\delta_{gL,t_{ik}L}\frac{1}{|L|}\theta(E_{ij}).
	\end{eqnarray*}
	Consequently,
	\[
	\theta(p_g\cdot E_{ij})=\delta_{gL,t_{ij}L}\frac{1}{|L|}\theta(E_{ij})
	=\theta(E_{ii})(p_g\triangleright \theta(E_{ij}))
	=(p_g\triangleright \theta(E_{ij}))\theta(E_{jj}).
	\]
	Therefore, given $M \in FMat_{\mathbb{N}}(\Bbbk)$, it follows that $\theta(p_g\cdot M)=\theta(E)(p_g\triangleright \theta(M))=(p_g\triangleright \theta(M))\theta(F)$, for every pair of finite sums $E$, $F$ of distinct matrices$E_{ii}'s$ such that $EM=M=MF$. 
	
	Finally, suppose that $\sum_ikh_i\cdot a_i=0$ for every $k\in H$, with $h_i=\sum_{g\in G}\alpha^i_gp_g$ and $a_i=\sum_{j,l\in\mathbb{N}}a^i_{jl}E_{jl}$, where almost every $a^i_{jl}=0$. Hence, for every $g\in G$, choosing $k=p_g$, we have that
	\[
	0=\sum_i \alpha^i_gp_g\cdot a_i=\sum_{i,j,l}\alpha^i_gp_g\cdot a^i_{jl}E_{jl}
	=\sum_{i,j,l}\alpha^i_g\delta_{gL,t_{jl}L}a^i_{jl}\frac{1}{|L|}E_{jl},
	\]
	then, for every $j,l\in\mathbb{N}$, it follows that
	\begin{eqnarray*}
		\sum_{i}\alpha^i_g\delta_{gL,t_{jl}L}a^i_{jl}=0.
	\end{eqnarray*}
	Hence,
	\begin{eqnarray*}
		\sum_i h_i\triangleright\theta(a_i)&=&\sum_{i,g}\alpha^i_gp_g\triangleright(\sum_{j,l}\sum_{h\in t_{jl}L}\frac{1}{|L|}a^i_{jl}h)E_{jl}\\
		&=&\sum_{i,j,l,g}(\alpha^i_g\frac{1}{|L|}a^i_{jl}\delta_{gL,t_{jl}L}g)E_{jl}=0,
	\end{eqnarray*}
	thus showing that $(B,\theta)$ is the minimal globalization of the partial action $\cdot:H\otimes A\to A$.

	\subsection{Globalization of $S$-categorizable partial actions}
	In this section we will consider globalizations of $S$-categorizable partial actions and their connections to globalizations of partial actions on categories.
	Before we proceed, we recall some definitions from \cite{alves2}.

	\begin{definition}\cite{alves2} (semicategories and semifunctors)
		\begin{enumerate} [\normalfont(1)]
			\item A $\Bbbk$-semicategory $\mathcal{C}$ consists of a class $\mathcal{C}_0$ (its objects) and, for each $x, y \in \mathcal{C}_0$, a $\Bbbk$-linear space ${}_y\mathcal{C}_x$, the space of morphisms from $x$ to $y$, such that $\mathcal{C}$ satisfies all the axioms of a $\Bbbk$-category except for the existence of an identity for every object.
			\item A $\Bbbk$-semifunctor from the $\Bbbk$-semicategory $\mathcal{C}$ to the $\Bbbk$-semicategory $\mathcal{D}$ is a family of $\Bbbk$-linear maps  ${}_yF_x : {}_y \mathcal{C}_x \to {}_{F(y)} \mathcal{C}_{F(x)}$, where $(x,y) \in \mathcal{C}_0 \times \mathcal{C}_0$,  such that 
			${}_zF_x ({}_zg_y \circ {}_yf_x) = {}_zF_y ({}_zg_y) \circ {}_yF_x ( {}_yf_x)$ for every $x,y,z \in \mathcal{C}_0$, $ {}_yf_x \in {}_y\mathcal{C}_x$ and ${}_zg_y \in {}_z\mathcal{C}_y$. As usual, we will simply write $F$ instead of ${}_y F_x$.
			\item Let $\mathcal{C}$ be a $\Bbbk$-semicategory. 
			\begin{enumerate}[\normalfont(a)]
				\item A $\mathcal{C}_0$-semicategory  is a semicategory $\mathcal{D}$ whose class of objects is $\mathcal{C}_0$.
				\item A $\mathcal{C}_0$-subsemicategory $\mathcal{E}$ of a $\mathcal{C}_0$-semicategory $\mathcal{D}$ is a subsemicategory of $\mathcal{D}$ such that $\mathcal{E}_0 = \mathcal{D}_0 = \mathcal{C}_0$.
				\item A $\mathcal{C}_0$-semifunctor is a functor between $\mathcal{C}_0$-semicategories that is the identity on the objects.
			\end{enumerate}	
		\end{enumerate}
	\end{definition}
	
	\begin{definition}[\cite{alves2}]
		Let $\mathcal{C}$ be a $\Bbbk$-category. An ideal of $\mathcal{C}$ is a $\mathcal{C}_0$-subsemicategory $J$ of $\mathcal{C}$ such that $_zf_x\circ \,_xl_y\circ\,_yg_w\in \,_zJ_w$ for every $_xl_y\in \,_xJ_y$, $_zf_x\in\,_z\mathcal{C}_x$, $_yg_w\in\,_y\mathcal{C}_w$.
	\end{definition}
	
	The concept of central idempotent in a $\Bbbk$-category appears in \cite{alves2} but, in fact, it extends trivially to the context of $\Bbbk$-semicategories. 
	
	\begin{definition}[\cite{alves2}] \label{central.idempotent} A central idempotent in a $\Bbbk$-semicategory $\mathcal{C}$ is an idempotent natural transformation $e$ of the identity functor $Id_{\mathcal{C}}$ to itself, i.e., it is a collection $e=\{_xe_x\}_{x\in\mathcal{C}_0}$, where each $\,_xe_x\in\,_x\mathcal{C}_x$ is an idempotent endomorphism such that 
		\begin{eqnarray*}
			_ye_y\circ \,_yf_x=\,_yf_x\circ\,_xe_x
		\end{eqnarray*}
		for every $_yf_x\in\,_y\mathcal{C}_x$, $x,y \in \mathcal{C}_0$.
		Given a central idempotent $e$, the ideal $J$ of $\mathcal{C}$ generated by $e$ is given by 
		\begin{eqnarray*}
			_yJ_x=\,_ye_y\,_y\mathcal{C}_x\,_xe_x=\,_ye_y\,_y\mathcal{C}_x=\,_y\mathcal{C}_x\,_xe_x.
		\end{eqnarray*}
	\end{definition}
	
	The definition of a central idempotent of a category was motivated by the well known fact that if $B$ is an $H$-module algebra with action $\triangleright$ and $A$ is the ideal of $B$ generated by a central idempotent $e$, i.e., $A=Be$, then the mapping $g\cdot a=(h\triangleright a)e$ determines a partial action of $H$ on $A$. 
	Moreover, if $\mathcal{C}$ is a semicategory with finite number of objects, then $B=a(\mathcal{C})$ is a nonunital algebra, and every element of the form $\sum_{x\in\mathcal{C}_0} e_{xx}$ such that $e_{xx}\circ \,_xf_y=\,_xf_y\circ e_{yy}$ for every $x,y\in \mathcal{C}_0$, i.e., 
	$e=\{e_{xx}  \}_{x \in \mathcal{C}_0}$ determines a central idempotent as in the definition above, is a central idempotent of this algebra. Definition \ref{central.idempotent} is written in such a way that it works even for categories with infinite number of objects.
	
	\begin{definition}[\cite{alves2}] Let $\mathcal{C}$ be a partial $H$-module category. A globalization of the partial action is a pair $(\mathcal{B},F)$ where
		\begin{enumerate}[\normalfont(1)]
			\item $\mathcal{B}$ is an $H$-module semicategory over $\mathcal{C}_0$, with action $\triangleright$;
			\item $F:\mathcal{C}\longrightarrow\mathcal{B}$ is a faithful $\mathcal{C}_0$-semifunctor and $F(\mathcal{C})$ is the ideal of $\mathcal{B}$ generated by the central idempotent $e=\{F(\,_x1_x)\}_{x\in\mathcal{C}_0}$;
			\item $\mathcal{B}=H\triangleright F(\mathcal{C})$;
			\item $F$ intertwines the partial action on $\mathcal{C}$ and the induced partial action on $F(\mathcal{C})$, i.e., for every $_yf_x\in\,_y\mathcal{C}_x$ we have
			\begin{eqnarray*}
				F(h\cdot\,_yf_x)=F(\,_y1_y)(h\triangleright F(\,_yf_x))=(h\triangleright F(\,_yf_x))F(\,_x1_x).
			\end{eqnarray*}
		\end{enumerate}
	\end{definition}
	
	\begin{proposition}
		Let $\cdot: H \otimes A \longrightarrow A$ be an $S$-categorizable symmetrical partial action and $(B,\theta)$ be an enveloping action. Then the partial action of $H$ on $\mathcal{C}^S(A)$ induced by $\cdot$ has a globalization $(\mathcal{B},F)$ given by:
		\begin{itemize}
			\item $\mathcal{B}_0=\mathcal{C}^S(A)_0$\\
			$\mathcal{B}(\alpha,\beta)=\,_\beta\mathcal{B}_\alpha=H\triangleright \theta(\,_\beta\mathcal{C}^S(A)_\alpha)=H\triangleright \theta(e_\beta Ae_\alpha)$;
			\item $F:\mathcal{C}^S(A)\longrightarrow \mathcal{B}$ is given by $F_0(\alpha)=\alpha$ and $F_1(e_\beta ae_\alpha)=\theta(e_\beta ae_\alpha)$.
		\end{itemize}
	\end{proposition}
	\begin{proof}
		In fact,
		\begin{enumerate}[\normalfont(1)]
			\item $\mathcal{B}$ is an $H$-module semicategory over $\mathcal{C}^S(A)_0$, with action $\blacktriangleright$ induced by $\triangleright$;
			\item $F$ is a faithful functor, because $\theta$ is monomorphism, $e=\{\theta(e_\alpha)\}_{e_\alpha \in S}$ is a central idempotent and $F(\mathcal{C}^S(A))$ is the ideal of $\mathcal{B}$ generated by $e$;
			\item $\mathcal{B}=H\triangleright F(\mathcal{C}^S(A))$;
			\item $F(h\cdot e_\beta ae_\alpha)=\theta(h\cdot e_\beta ae_\alpha)=\theta(e_\beta)(h\triangleright \theta(e_\beta ae_\alpha))=F(e_\beta)(h\triangleright F(e_\beta ae_\alpha))$ and\\
			$F(h\cdot e_\beta ae_\alpha)=\theta(h\cdot e_\beta ae_\alpha)=(h\triangleright \theta(e_\beta ae_\alpha))\theta(e_\alpha)=(h\triangleright F(e_\beta ae_\alpha))F(e_\alpha)$.
		\end{enumerate}
	\end{proof}
	
	Conversely, if the partial action $\triangleright: H \otimes \mathcal{C} \longrightarrow \mathcal{C}$ has a globalization $(\mathcal{B},F)$, then the partial action $\cdot: H \otimes a(\mathcal{C}) \longrightarrow a(\mathcal{C})$ induced by $\triangleright$ has a globalization given by the pair $(a(\mathcal{B}),\theta)$, where $\theta((\,_yf_x)_{x,y})=(F(\,_yf_x))_{x,y}$ and the action on $a(\mathcal{B})$ is induced by the action on $\mathcal{B}$.

	\section{A Morita context  between $\underline{A\# H}$ and $B\# H$}

	In \cite{alves3}, Alves and Batista proved that there exists a strict Morita context between $\underline{A\# H}$ and $B\# H$, whenever $(B,\theta)$ is a globalization for the symmetrical partial action on the unital algebra $A$ and $H$ has bijective antipode. We will show that an equivalent argument holds for nonunital algebras.
	
	We recall from \cite{garciasimon} that two idempotent rings are Morita equivalent, i.e., their categories of unital and torsionfree modules are equivalent, if and only if there exists a strict Morita context where the modules are unital.

	\begin{lemma} Let $A$ be a partial $H$-module algebra and $\theta: A\to B$ be a globalization. Then the map
		\[
		\Phi: \underline{A\# H} \to B\# H, \ \ \ \ \	\Phi(\sum a(h_{(1)}\cdot b)\# h_{(2)})=\sum \theta(a)(h_{(1)}\triangleright \theta(b))\# h_{(2)},
		\]
		is an algebra monomorphism.
	\end{lemma}
	
	\begin{proof} Consider the linear map $\Phi :A\# H \longrightarrow B\# H$ defined by 
		\begin{eqnarray*}
			\Phi(a \# h)= \theta(a) \# h.
		\end{eqnarray*}
		By the same calculations as in \cite{alves1} $\Phi$ is a well-defined  algebra morphism (due to the axioms of a globalization), and $\Phi$ is injective because $\theta$ is injective.
		
		Its restriction to the subalgebra $\underline{A\# H}$,  which we will also denote by $\Phi$, is therefore an injective algebra morphism, whose expression on a generator $\sum a(h_{(1)}\cdot b)\# h_{(2)}$ is given by
		\begin{eqnarray*}
			\Phi(\sum a(h_{(1)}\cdot b)\# h_{(2)})=\sum \theta(a)(h_{(1)}\triangleright \theta(b))\# h_{(2)}.
		\end{eqnarray*}
	\end{proof}

	\begin{theorem}\label{teo2} Let $H$ be a Hopf algebra with bijective antipode,  $A$ be an idempotent algebra which is a  partial $H$-module algebra, and let  $(B,\theta)$ be a globalization for $A$. Then, there exists a strict Morita context between $\underline{A\# H}$ and $B\# H$.
	\end{theorem}
	\begin{proof} As in \cite{alves1} and \cite{dok},  consider the vector spaces
		\begin{eqnarray*}
			M&=& \{\sum_i \theta(a_i)\# h_i;\,\, a_i\in A,\,\, h_i\in H\}\\
			N&=&\{\sum_{i,(h_i)}(h_i)_{(1)}\triangleright \theta(a_i)\# (h_i)_{(2)};\,\, a_i \in A,\,\, h_i\in H\},
		\end{eqnarray*}
		i.e., $M=\Phi(A\# H)$ and $N$ is the subspace of $B\# H$ generated by the elements $\sum h_{(1)}\triangleright \theta(a)\# h_{(2)}$. Since $M$ and $N$ are subspaces of $B\# H$ and $\underline{A\# H}$ can be seen as a subspace of $B\# H$, by the previous lemma, it is not hard to show that $M$ is an $\underline{A\# H}-B\# H$-bimodule and $N$ is a $B\# H-\underline{A\# H}$-bimodule; these bimodule structures are the same ones considered in \cite{alves1}.
		
		For the rest of the Morita context, define the maps
		\begin{eqnarray*}
			\tau&:&M\otimes_{B\# H}N\longrightarrow \underline{A\# H}\cong \Phi(\underline{A\# H})\subseteq B\# H\\
			\sigma&:&N\otimes_{\underline{A\# H}}M\longrightarrow B\# H,
		\end{eqnarray*}
		both induced by the multiplication in $B\#H$. This is possible because $M$, $N$ and $\underline{A\# H}$ are seen as subspaces of $B\# H$, as mentioned before. Since the multiplication on $B\# H$ is associative, we have that both $\tau$ and $\sigma$ are bimodule morphisms. Hence, from \cite{garciasimon}, we only need to prove that they are surjective.
		
		First, we have that $MN\subseteq \Phi(\underline{A\# H})$ because for every $a,b \in A$ and $h,k\in H$ we have
		\begin{eqnarray*}
			(\theta(a)\# h)(\sum k_{(1)}\triangleright \theta(b)\# k_{(2)})&=& \sum \theta(a)(h_{(1)}k_{(1)}\triangleright \theta(b))\# h_{(2)}k_{(2)}\\
			&=& \sum \theta(a)((hk)_{(1)}\triangleright \theta(b))\# (hk)_{(2)}\,\in \,\Phi(\underline{A\# H}).
		\end{eqnarray*} 
		Since $\sum \theta(a)(h_{(1)}\triangleright \theta(b))\# h_{(2)}=(\theta(a)\# h)(\theta(b)\# 1_H)$ and $\theta(b)\# 1_H$ lies in $N$, we have that $MN=\Phi(\underline{A\# H})$. 
		Finally, as $NM\subseteq B\# H$, $h\triangleright\theta(a)\# k$ is a generator of $B\# H$ as vector space and
		\begin{eqnarray*}
			h\triangleright\theta(a)\# k=\sum(h_{(1)}\triangleright\theta(a_1)\# h_{(2)})(\theta(a_2)\# S(h_{(3)})k)\,\in\, NM,
		\end{eqnarray*} 
		where $a=\sum_i a_{1i}a_{2i}=\sum a_1a_2$ ($A$ is idempotent), we have that $NM=B\# H$.
	\end{proof}
	
	Since $A$ is idempotent, we have that both $\underline{A\# H}$ and $B\# H$ are also idempotent, and by \cite{garciasimon} we have the following result.
	
	\begin{theorem} Under the hypotheses of the previous theorem, we have that the category of the unital and torsionfree left $B\# H$-modules ($B\# H$-mod) and the category of the unital and torsionfree left $\underline{A\# H}$-modules ($\underline{A\# H}$-mod) are equivalent.
	\end{theorem}
	
	\begin{remark}
		Under the hypothesis of Theorem \ref{teo2}, if $A$ is a unital algebra, even if $B\# H$ does not have unit, there still exist an equivalence between the categories of $\underline{A\# H}$-modules and the category of the unital and torsionfree left $B\# H$-modules.
	\end{remark}
	
	\section{Morita equivalence between $A$ and $a(\mathcal{C}^S(A))$}

	In this section we will show that every algebra $A$ with s.l.u. $S$ is Morita equivalent to the algebra $a(\mathcal{C}^S(A))$. This result will be fundamental further on for deriving an equivalence between the categories of unital $\underline{A\# H}$ modules and of $\underline{\mathcal{C}^S(A)\# H}$ modules
	(see Corollary \ref{corollary.morita.smash.2}).
	In order to do that, we will prove that the category of the unital $A$-modules is equivalent to the category of $\mathcal{C}^S(A)$-modules, and then we will show that for every $\Bbbk$-category $\mathcal{C}$, the category of the $\mathcal{C}$-modules is equivalent to the category of the unital $a(\mathcal{C})$-modules. We begin by recalling the definition of modules of a $\Bbbk$-category.
	
	\begin{definition}
		Let $\mathcal{C}$ be a $\Bbbk$-category. A $\mathcal{C}$-module $\mathcal{M}$ is a functor from $\mathcal{C}$ to the category of $\Bbbk$-vector spaces $\textbf{Vec}$. In other words, for every object $x\in \mathcal{C}_0$, a vector space $_x\mathcal{M}$ is determined, and for every linear map $f:x\to y$ in $\mathcal{C}_1$, a linear transformation $f:\,_x\mathcal{M} \to\,_y\mathcal{M}$ is determined, such that for all $x, y, z  \in \mathcal{C}_0$, ${}_xm \in _x\mathcal{M}$ and $_yg_z, _zf_x \in \mathcal{C}_1$,
		\begin{enumerate}[\normalfont(1)]
			\item $_x1_x\,_xm=\,_xm$;
			\item $_yg_z(\,_zf_x\,_xm)=(_yg_z\circ\,_zf_x)\,_xm$.
		\end{enumerate}
	\end{definition} 
	
	We know that given an algebra with local units $A$ and a fixed s.l.u. $S=\{e_\lambda\}_{\lambda\in\Lambda}$, every $A$-module $M$ can be seen as a $\mathcal{C}^S(A)$-module, namely $\mathcal{M}=\{\,_\lambda\mathcal{M}\}_{\lambda\in\Lambda}$, where $\,_\lambda\mathcal{M}=e_\lambda M$ and the actions on each $e_\lambda M$ are the restrictions of the original action of $A$ on $M$.
	
	Now, to construct an $A$-module arising from a $\mathcal{C}^S(A)$-module, note that given a $\mathcal{C}^S(A)$-module $\mathcal{M}=\{\,_\lambda\mathcal{M}\}_{\lambda\in\Lambda}$, for every $e_\lambda\leq e_\alpha$, we have that $e_\lambda\in e_\alpha Ae_\lambda\cap e_\lambda Ae_\alpha\cap e_\lambda Ae_\lambda$, hence $e_\lambda\cdot_{(\lambda,\lambda)}\,_\lambda m=e_\lambda\cdot_{(\alpha,\lambda)} (e_\lambda\cdot_{(\lambda,\alpha)}\,_\lambda m) $ for every $_\lambda\in\,_\lambda\mathcal{M}$. Since $e_\lambda\cdot_{(\lambda,\lambda)} \square =id_{_\lambda\mathcal{M}}$, we have that $I_{(\lambda,\alpha)}=e_\lambda\cdot_{(\lambda,\alpha)}\square$ is an injective linear map. This calculation is illustrated by the following diagram.
	
	$$\xymatrix{{}_\lambda M \ar@(ul,dl)[]_{e_\lambda \cdot_{(\lambda,\lambda)}} \ar@<1ex>[r]^{e_\lambda\cdot_{(\lambda,\alpha)}}&{}_\alpha M \ar@<1ex>[l]^{e_\lambda \cdot_{(\alpha,\lambda)}}} $$
	
	\begin{lemma}$\{\{\,_\lambda\mathcal{M}\}_{\lambda\in \Lambda},\{I_{(\lambda,\alpha)}:\,_\lambda\mathcal{M}\to\,_\alpha\mathcal{M}\}_{\lambda\leq\alpha\in\Lambda}\}$ forms a direct system of vector spaces.
	\end{lemma}
	\begin{proof}
		Follows directly from the fact that $e_\gamma e_\alpha=e_\alpha$ whenever $e_\alpha\leq e_\gamma$.
	\end{proof}
	
	It is well known that every direct system of vector spaces $\{\iota_{ij}:M_i\to M_j \}$ has a (unique) limit, i.e., there exists a vector space $M$ with inclusions $\iota_i: M_i\to M$ such that for every vector space $N$ with inclusions $\nu_i: M_i\to N$ there exists a unique linear transformation $\theta: M\to N$ such that $\theta\circ \iota_i=\nu_i$ for every $i$.
	
	To describe the limit of the direct system of the previous lemma, consider the subspace $T$ of $ \oplus_{\lambda} \,_\lambda\mathcal{M}$ generated by elements of the form $_\lambda m-I_{(\lambda,\alpha)}(\,_\lambda m)$ with $e_\lambda\leq e_\alpha$; then $\underrightarrow{lim}\,\,_\lambda\mathcal{M}=\oplus\,_\lambda\mathcal{M}/T$ by Proposition $2.6.8$ of \cite{weibel}. 
	Since $I_{(\lambda,\alpha)}$ is injective, we can identify $_\lambda M$ with $\overline{_\lambda M}\subset M$ by $_\lambda m\mapsto \overline{_\lambda m}$, the equivalence class of $_\lambda m$ in $M$.
	
	Now, consider $M=\underrightarrow{lim}\,\,_\lambda\mathcal{M}$ and define
	\begin{eqnarray*}
		\bullet: A\otimes M &\longrightarrow & M\\
		a\otimes \overline{\,_\lambda m} &\mapsto & \overline{e_\beta ae_\alpha\cdot_{(\alpha,\beta)}I_{(\lambda,\alpha)}\,_\lambda m},
	\end{eqnarray*}
	where $ae_\alpha=a=e_\beta a$ and $e_\lambda\leq e_\alpha$. In other words, $a\bullet \overline{\,_\lambda m}=\overline{a\cdot_{(\alpha,\beta)}e_\lambda\cdot_{(\lambda,\alpha)}\,_\lambda m}=\overline{ae_\lambda\cdot_{(\lambda,\beta)}\,_\lambda m}$
	
	\begin{proposition} $M$ is an $A$-module via $\bullet$.
	\end{proposition}
	\begin{proof}
		We will prove only that $\bullet$ is well-defined, since the remaining verifications are then straightforward. Because of the description $a\bullet \overline{\,_\lambda m}=\overline{ae_\lambda\cdot_{(\lambda,\beta)}\,_\lambda m}$, that does not depend on the choice of $e_\alpha$, we will show that $\overline{ae_\lambda\cdot_{(\lambda,\beta)}\,_\lambda m}=\overline{ae_\lambda\cdot_{(\lambda,\gamma)}\,_\lambda m}$ whenever $e_\beta\leq e_\gamma$, in fact 
		\begin{eqnarray*}
			\overline{ae_\lambda\cdot_{(\lambda,\beta)}\,_\lambda m}&=& \overline{I_{(\beta,\gamma)}(ae_\lambda\cdot_{(\lambda,\beta)}\,_\lambda m)}\\
			&=&\overline{e_\gamma \cdot_{(\beta,\gamma)}ae_\lambda\cdot_{(\lambda,\beta)}\,_\lambda m}\\
			&=& \overline{e_\gamma ae_\lambda\cdot_{(\lambda,\gamma)}\,_\lambda m}\\
			&=&\overline{ae_\lambda\cdot_{(\lambda,\gamma)}\,_\lambda m}.
		\end{eqnarray*}
		Finally, if we chose $e_\alpha$ and $e_\beta$ such that $e_\beta a=e_\alpha a=a$, there always exist $e_\gamma \geq e_\alpha,e_\beta$, and we are done.   
	\end{proof}
	
	\begin{theorem}
		Let $A$ be an algebra with s.l.u. $S$. Then the category of the unital $A$-modules is equivalent to the category of $\mathcal{C}^S(A)$-modules.
	\end{theorem}
	\begin{proof}
		First, note that if $\overline{_\lambda m}=\overline{_\alpha m}\in G(\mathcal{M})$, there exist $\gamma\in\Lambda$ such that $I^\mathcal{M}_{(\lambda,\gamma)}(_\lambda m)=I^\mathcal{M}_{(\alpha,\gamma)}(_\alpha m)$. Let $f:\mathcal{M}\to \mathcal{N}$ be a morphism in $\mathcal{C}^S(A)-Mod$, then 
		\begin{align*}
			I^\mathcal{N}_{(\lambda,\gamma)}(f_\lambda(_\lambda m)) & =e_\lambda\cdot_{(\lambda,\gamma)}f_\lambda(_\lambda m)
			=f_\lambda(e_\lambda\cdot_{(\lambda,\gamma)}\,_\lambda m)
			=f_\gamma(I^\mathcal{M}_{(\lambda,\gamma)}(_\lambda m))
			\\
			&=f_\gamma(I^\mathcal{M}_{(\alpha,\gamma)}(_\alpha m)) =I^\mathcal{N}_{(\alpha,\gamma)}(f_\alpha(_\alpha m)),
		\end{align*}
		
		i.e., $\overline{f_\lambda(_\lambda m)}=\overline{f_\alpha(_\alpha m)}$ and $G$ is well-defined.
		Now, let $A-Mod$ be the category of all unital $A$-modules. We have the following functors
		\begin{eqnarray*}
			F: A-Mod&\longrightarrow &\mathcal{C}^S(A)-Mod\\
			(M,\cdot)&\mapsto & \{e_\lambda M\}_{\lambda\in\Lambda}\\
			f:M\to N&\mapsto & \{F(f)_\lambda =f|_{e_\lambda M}\}\\
			\\
			G: \mathcal{C}^S(A)-Mod &\longrightarrow & A-Mod\\
			\mathcal{M} &\mapsto & (\underrightarrow{lim}\,\,\mathcal{M}_\lambda ,\bullet)\\
			f=\{f_\lambda:\,_\lambda\mathcal{M}\to \,_\lambda\mathcal{N}\}&\mapsto & G(f)(\overline{_\lambda m})=\overline{f_\lambda(_\lambda m)}
		\end{eqnarray*} 
		Since $\underrightarrow{lim}\,\,e_\lambda M=M$ whenever $M\in A-Mod$ and the $A$-module structure of $M$ coincides with that given by the functor $G$, we have that $GF=Id_{A-Mod}$ on objects, and is easy to see that $GF=Id_{A-Mod}$ in the morphisms too.  
		Now, consider the linear map
		\begin{eqnarray*}
			\varphi_\lambda: \,_\lambda\mathcal{M} &\longrightarrow & e_\lambda \bullet G(\mathcal{M})\\
			_\lambda m &\mapsto & \overline{_\lambda m}.
		\end{eqnarray*}
		Since $\overline{_\lambda m}=\overline{e_\lambda\cdot_{(\lambda,\lambda)}\,_\lambda m}=e_\lambda\bullet\overline{_\lambda m}$, we have that $\varphi$ is well-defined. Now, suppose that there exist $_\lambda m, \,_\lambda n\in \,_\lambda\mathcal{M}$ such that $\overline{_\lambda m}=\overline{_\lambda n}$. Then, there exist $e_\lambda\leq e_\beta$ such that $I_{(\lambda,\beta)}(_\lambda m)=I_{(\lambda,\beta)}(_\lambda n)$, since $I_{(\lambda,\beta)}$ is injective, we conclude that $_\lambda m=\,_\lambda n$. Finally, note that for every $\lambda,\beta\in\Lambda$ there exist $\gamma\in\Lambda$ such that $e_\lambda\leq e_\gamma$ and $e_\beta\leq e_\gamma$, then
		\begin{eqnarray*}
			e_\lambda\bullet\overline{_\beta m}&=&\overline{e_\lambda e_\gamma\cdot_{(\gamma,\lambda)}I_{(\beta,\gamma)}(_\beta m)}\\
			&=&\varphi_\lambda(e_\lambda e_\gamma\cdot_{(\gamma,\lambda)}I_{(\beta,\gamma)}(_\beta m)),
		\end{eqnarray*} 
		since $e_\lambda e_\gamma\cdot_{(\gamma,\lambda)}I_{(\beta,\gamma)}(_\beta m)\in\,_\lambda\mathcal{M}$, we have that $\varphi_\lambda$ is surjective and, consequently, each $\varphi_\lambda$ is an isomorphism of vector spaces. Moreover, as 
		\begin{eqnarray*}
			e_\beta ae_\lambda \bullet \varphi_\lambda(_\lambda m)&=&e_\beta ae_\lambda \bullet \overline{_\lambda m}\\
			&=&\overline{e_\beta ae_\lambda\cdot_{(\lambda,\beta)}\,_\lambda m}\\
			&=&\varphi_\beta(e_\beta ae_\lambda\cdot_{(\lambda,\beta)}\,_\lambda m),
		\end{eqnarray*}
		we have that the following diagrams commute:
		\begin{center}
			$\xymatrixcolsep{3pc}\xymatrix{
				_\lambda\mathcal{M} \ar[r]^{e_\beta ae_\lambda}\ar[d]_{\varphi_\lambda} & _\beta\mathcal{M}\ar[d]^{\varphi_\beta}\\
				e_\lambda\bullet G(\mathcal{M})\ar[r]_{e_\beta ae_\lambda} & e_\beta\bullet G(\mathcal{M})}$
		\end{center}
		i.e., since every $\varphi_\lambda$ is a linear isomorphism, $\varphi=\{\varphi_\lambda\}_{\lambda\in\Lambda}$ is a $\mathcal{C}^S(A)$-module isomorphism. Hence $FG\cong Id_{\mathcal{C}^S(A)-Mod}$ when restricted to the objects.  Finally, we have that $$FG(f)(\varphi_\lambda^\mathcal{M}(_\lambda m))=G(f)(\varphi_\lambda^\mathcal{M}(_\lambda m))=\varphi_\lambda^\mathcal{N} (f_\lambda(_\lambda m)).$$ Then $\varphi$ is actually a natural transformation, hence a natural isomorphism.
	\end{proof}
	
	\begin{theorem}\label{cac}
		Let $\mathcal{C}$ be a $\Bbbk$-category. Then the category of the $\mathcal{C}$-modules and the category of the unital $a(\mathcal{C})$-modules are equivalent.
	\end{theorem} 
	\begin{proof}
		In fact, consider the functors
		\begin{eqnarray*}
			F:\mathcal{C}-Mod &\to &a(\mathcal{C})-Mod\\
			M=\{_xM\}_{x\in\mathcal{C}_0}&\mapsto & \bigoplus_{x\in\mathcal{C}_0}\,_xM\\
			\theta: M\to N&\mapsto & F(\theta)((_xm)_{x\in\mathcal{C}_0})=(\theta_x(_xm))_{x\in\mathcal{C}_0}
		\end{eqnarray*}
		and
		\begin{eqnarray*}
			G: a(\mathcal{C})-Mod &\to & \mathcal{C}-Mod\\
			M&\mapsto & \{e_xM\}_{x\in\mathcal{C}_0}\\
			\alpha: M\to N &\mapsto & G(\alpha)_x=\alpha,
		\end{eqnarray*}
		where $e_x$ denotes the matrix with $_x1_x$ in the $(x,x)$ position and $0$ otherwise. Note that $S=\{\sum_{i=1}^ne_{x_i}\}$ is a system of local units for $a(\mathcal{C})$. Now, we have that 
		\begin{eqnarray*}
			FG(M)=\bigoplus_x\,_xG(M)=\bigoplus_xe_xM= M
		\end{eqnarray*}
		and
		\begin{eqnarray*}
			_yGF(\{_xM\})=e_y(\bigoplus_x\,_xM)=\,_yM,
		\end{eqnarray*}
		then $FG=Id_{a(\mathcal{C})-Mod}$ and $GF=Id_{\mathcal{C}-Mod}$ in the objects. For the morphisms, we have that 	\[
		(FG(\alpha))(m)=(FG(\alpha))(\sum_{i=1}^n e_{x_i}m) =\sum_{i=1}^n \alpha(e_{x_i}m)
		= \sum_{i=1}^n e_{x_i}\alpha(m)
		= \alpha(m),
		\]
		where $e=\sum_{i=1}^n e_{x_i}$ is such that $em=m$. And
		\begin{eqnarray*}
			(GF(\theta))_y(_ym)=F(\theta)(_ym)=\theta_y(_ym).
		\end{eqnarray*}
	\end{proof}
	
	Now we can enunciate the desired result.
	
	\begin{corollary} Let $A$ be an algebra with s.l.u. $S$. Then $A$ and $a(\mathcal{C}^S(A))$ are Morita equivalent.
	\end{corollary}
	By \cite{ahnmarki}, every algebra with local units is Morita equivalent to some algebra with enough idempotents, i.e., to some algebra $R$ such that $R=\bigoplus_e eR=\bigoplus_e Re$, where $\{e \}$ is a family of orthogonal idempotents in $R$. The previous corollary prove the same thing and give a better description of such algebra with enough idempotents. Also, in spite of the notation as subset of a system of local units, nothing require that the local units must be different. So, if we consider a unital algebra $A$ with an s.l.u. $S=\{e_1,\cdots,e_n\}$ where every $e_i=1_A$, this corollary provides the well known Morita equivalence between $A$ and $Mat_{n\times n}(A)$.
	
	We can push this result even further: in \cite{des} the authors prove that given two algebras with enough idempotents $A$ and $B$, there exists an infinite set of indexes $X$ such that $FMat_X(A)\cong FMat_X(B)$ as algebras. As a consequence, we obtain the following result.
	
	\begin{corollary} Let $A$ be an algebra with s.l.u. $S=\{e_\lambda\}_{\lambda\in\Lambda}$ and $B$ an algebra with s.l.u. $T=\{f_i\}_{i\in\Gamma}$. If $A$ and $B$ are Morita equivalent, then there exist an infinite set of indices $X$ whose cardinality is greater than or equal to those of $\Lambda$ and $\Gamma$ such that $FMat_X(a(\mathcal{C}^S(A)))\cong FMat_X(a(\mathcal{C}^T(B)))$.
	\end{corollary}

	\section{Morita Equivalence of Partial Actions}
	
	In this section we introduce Morita equivalence of partial Hopf actions, extending part of the theory developed in in \cite{abadie} for partial group actions. We will show that every partial action is Morita equivalent to a partial action of the Hopf algebra on an algebra with trivial right annihilator. Turning to the case of a partial $H$-module algebra with local units $A$ with $S$-globalizable partial action, we will prove that $\underline{A\# H}$ is Morita equivalent to $\underline{a(\mathcal{C}^S(A))\# H}$, where the partial action on $a(\mathcal{C}^S(A))$ is induced by the $S$-categorizable partial action on $A$.

	\begin{definition}\label{def4} Let $A$ and $B$ be associative partial $H$-module algebras with partial actions $\cdot_A$ and $\cdot_B$, respectively. We will say that $\cdot_A$ and $\cdot_B$ are \textit{Morita equivalent partial actions} if
		\begin{enumerate}[\normalfont(1)]
			\item There exists a strict Morita context $(A,B,\,_AM_B,\,_BN_A,\tau,\sigma)$, where $M$ and $N$ are unital bimodules;
			\item There exists a partial action $\triangleright: H\otimes C\to C$, where $C$ is the context algebra $C=\left(\begin{array}{cc}
				A & M\\
				N & B
			\end{array}\right)$, such that its restriction to $\left(\begin{array}{cc}
				A & 0\\
				0 & 0
			\end{array}\right)$ and $\left(\begin{array}{cc}
				0 & 0\\
				0 & B
			\end{array}\right)$ coincides with $\cdot_A$ and $\cdot_B$, respectively.
		\end{enumerate} 
	\end{definition}
	We recall that the multiplication of matrices in $C$ comes from the bimodule structures on $M$ and $N$ and from the maps $\tau: M \otimes_B N \to A$ and $\sigma: N \otimes_A M \to B$ of the Morita context. Throughout this subsection, we will denote $\tau(m,n)=(mn)$ and $\sigma(n,m)=(nm)$, for every $m\in M$, $n\in N$.
	
	Note that if $\cdot_A$ and $\cdot_B$ are Morita equivalent, then the partial action $\triangleright$ on $C$ provides a linear map $\varphi_M: H\otimes M\to M$, $\varphi_M(h\otimes m)=hm$. 
	In fact, since $M$ is a unital left $A$-module, for any $m\in M$, we have that $m=\sum_i a_im_i$, with $a_i\in A$, $m_i\in M$ and
	\begin{eqnarray*}
		h\triangleright \left(\begin{array}{cc}
			0 & m\\
			0 & 0
		\end{array}\right)&=&\sum_i h\triangleright \left(\begin{array}{cc}
			a_i & 0\\
			0 & 0
		\end{array}\right)\left(\begin{array}{cc}
			0 & m_i\\
			0 & 0
		\end{array}\right)\\
		&=&\sum_i \left(\begin{array}{cc}
			h_{(1)}\cdot_A a_i & 0\\
			0 & 0
		\end{array}\right)h_{(2)}\triangleright\left(\begin{array}{cc}
			0 & m_i\\
			0 & 0
		\end{array}\right)\in \left(\begin{array}{cc}
			A & M\\
			0 & 0
		\end{array}\right).
	\end{eqnarray*} 
	Analogously, since $M$ is a unital right $B$-module, we have that
	\begin{eqnarray*}
		h\triangleright\left(\begin{array}{cc}
			0 & m\\
			0 & 0
		\end{array}\right)\in \left(\begin{array}{cc}
			0 & M\\
			0 & B
		\end{array}\right),
	\end{eqnarray*}
	hence
	\begin{eqnarray*}
		h\triangleright\left(\begin{array}{cc}
			0 & m\\
			0 & 0
		\end{array}\right)\in \left(\begin{array}{cc}
			0 & M\\
			0 & B
		\end{array}\right)\cap \left(\begin{array}{cc}
			A & M\\
			0 & 0
		\end{array}\right)=\left(\begin{array}{cc}
			0 & M\\
			0 & 0
		\end{array}\right),
	\end{eqnarray*}
	and therefore, we define $hm$ by the equation
	\begin{eqnarray*}
		h\triangleright \left(\begin{array}{cc}
			0 & m\\
			0 & 0
		\end{array}\right)=\left(\begin{array}{cc}
			0 & hm\\
			0 & 0
		\end{array}\right)\in\left(\begin{array}{cc}
			0 & M\\
			0 & 0
		\end{array}\right).
	\end{eqnarray*}
	
	In the same way, the partial action on $C$ provides a linear map $\varphi_N:H\otimes N\to N$, $\varphi_N(h\otimes n)=hn$. 
	Thus we may write the partial action in the following form:
	\begin{eqnarray*}
		h\triangleright \left(\begin{array}{cc}
			a & m\\
			n & x
		\end{array}\right)=\left(\begin{array}{cc}
			h\cdot_Aa & hm\\
			hn & h\cdot_Bx
		\end{array}\right).
	\end{eqnarray*}
	
	\begin{lemma}
		The linear map $\varphi_M$ satisfies the following properties:
		\begin{enumerate}[\normalfont(1)]
			\item $1_Hm=m$;
			\item $h(a(km))=\sum (h_{(1)}\cdot_A a)((h_{(2)}k)m)$;
			\item $h(m(k\cdot_B b))=\sum (h_{(1)}m)(h_{(2)}k\cdot_B b)$,
		\end{enumerate}
		for every $h, k\in H$, $m\in M$.
	\end{lemma}
	
	\begin{proof}
		In fact, for item (2), we have that
		\begin{eqnarray*}
			\left(\begin{array}{cc}
				0 & h(a(km))\\
				0 & 0
			\end{array}\right)&=&h\triangleright \left(\begin{array}{cc}
				0 & a(km)\\
				0 & 0
			\end{array}\right)\\
			&=&h\triangleright \left(\left(\begin{array}{cc}
				a & 0\\
				0 & 0
			\end{array}\right)\left(k\triangleright\left(\begin{array}{cc}
				0 & m\\
				0 & 0
			\end{array}\right)\right)\right)\\
			&=&\sum \left(h_{(1)}\triangleright\left(\begin{array}{cc}
				a & 0\\
				0 & 0
			\end{array}\right)\right)\left(h_{(2)}k\triangleright\left(\begin{array}{cc}
				0 & m\\
				0 & 0
			\end{array}\right)\right)\\
			&=&\sum \left(\begin{array}{cc}
				h_{(1)}\cdot_Aa & 0\\
				0 & 0
			\end{array}\right)\left(\begin{array}{cc}
				0 & (h_{(2)}k)m\\
				0 & 0
			\end{array}\right)\\
			&=&\left(\begin{array}{cc}
				0 & \sum (h_{(1)}\cdot_Aa)((h_{(2)}k)m)\\
				0 & 0
			\end{array}\right),
		\end{eqnarray*}
		for every $a\in A$, $m\in M$, $h,k\in H$. Item 3) is proved using similar arguments.    
	\end{proof}

	In particular, items (1) and (2) show that $\varphi_M$ defines a structure of partial $(A,H)$-module on $M$ (see Definition \ref{def.partial(A,H)-module}). Analogously, $\varphi_N$ satisfies similar properties.
	
	\begin{lemma} For every $h, k\in H$, $m\in M$, $n\in N$, we have that
		\begin{enumerate}[\normalfont(1)]
			\item $h\cdot_A(m(kn))=\sum (h_{(1)}m)((h_{(2)}k)n)$;
			\item $h\cdot_B (n(km))=\sum (h_{(1)}n)((h_{(2)}k)m)$,
		\end{enumerate}
		
	\end{lemma}
	\begin{proof}
		In fact, for item (1), we have that
		\begin{eqnarray*}
			\left(\begin{array}{cc}
				h\cdot_A(m(kn)) & 0\\
				0 & 0
			\end{array}\right)&=&h\triangleright \left(\begin{array}{cc}
				(m(kn)) & 0\\
				0 & 0
			\end{array}\right)\\
			&=&h\triangleright \left(\left(\begin{array}{cc}
				0 & m\\
				0 & 0
			\end{array}\right)\left(k\triangleright\left(\begin{array}{cc}
				0 & 0\\
				n & 0
			\end{array}\right)\right)\right)\\
			&=&\sum \left(h_{(1)}\triangleright\left(\begin{array}{cc}
				0 & m\\
				0 & 0
			\end{array}\right)\right)\left(h_{(2)}k\triangleright\left(\begin{array}{cc}
				0 & 0\\
				n & 0
			\end{array}\right)\right)\\
			&=&\sum \left(\begin{array}{cc}
				0 & h_{(1)}m\\
				0 & 0
			\end{array}\right)\left(\begin{array}{cc}
				0 & 0\\
				(h_{(2)}k)n & 0
			\end{array}\right)\\
			&=&\left(\begin{array}{cc}
				\sum ((h_{(1)}m)((h_{(2)}k)n)) & 0\\
				0 & 0
			\end{array}\right),
		\end{eqnarray*}
		for every $m\in M$, $n\in N$, $h,k\in H$. Item (2) is proved using similar arguments.    
	\end{proof}
	
	In the definition of Morita equivalent partial actions, item (2) can be replaced by the existence of compatible partial $(A,H)-(B,H)$ and $(B,H)-(A,H)$-bimodule structures on $M$ and $N$, respectively.
	Just to be clear, the ``right module'' version of Definition \ref{def.partial(A,H)-module} is as follows. 
	
	\begin{definition} Let $B$ be an associative right partial $H$-module algebra. 
		A vector space $N$ is a right partial $(B,H)$-module if $N$ is a unital right $B$-module together with a linear map $N\otimes H\to N$, $n\otimes h\mapsto nh$ such that:
		\begin{enumerate}[\normalfont(1)]
			\item $n1_H=n$;
			\item $((nk)b)h =\sum (n(kh_{(1)}))(b \cdot h_{(2)})$,
		\end{enumerate}
		for every $h,k\in H$, $b\in B$, $n\in N$.
		If $N$ has also a structure of left $(A,H)$-module, where $A$ is an associative (left) partial $H$-module algebra, then we will say that $N$ is a partial $(A,H) - (B,H)$-bimodule if 
		\begin{enumerate}[\normalfont(1)]\addtocounter{enumi}{2}
			\item $a (n b) = (a n) b$ for all $a \in A$, $b \in B$;
			\item $h (n k) = (h n) k$ for all $h,k \in H$.
		\end{enumerate}
	\end{definition}
	
	\begin{proposition} \label{Proposition.Morita.equivalence.partial.bimodules}Let $A$ and $B$ be associative partial $H$-module algebras with partial actions $\cdot_A$ and $\cdot_B$, respectively, and assume that item {\normalfont(1)} of Definition \ref{def4} holds. The following are equivalent:
		\begin{enumerate}[\normalfont(1)]
			\item There exists a partial action $\triangleright: H\otimes C\to C$, where $C$ is the context algebra $C=\left(\begin{array}{cc}
				A & M\\
				N & B
			\end{array}\right)$, such that its restriction to $\left(\begin{array}{cc}
				A & 0\\
				0 & 0
			\end{array}\right)$ and $\left(\begin{array}{cc}
				0 & 0\\
				0 & B
			\end{array}\right)$ coincides with $\cdot_A$ and $\cdot_B$, respectively.
			\item $M$ has a partial $(A,H)-(B,H)$-bimodule structure and $N$ has a partial $(B,H)-(A,H)$-bimodule structure such that:
			\begin{itemize}
				\item $h\cdot_A (m(kn))=\sum (h_{(1)}m)((h_{(2)}k)n)$;
				\item $h\cdot_B (n(km))=\sum (h_{(1)}n)((h_{(2)}k)m)$,
			\end{itemize}
			for every $h\in H$, $m\in M$, $n\in N$.
		\end{enumerate} 
	\end{proposition}
	
	\begin{proof}
		The proof of (1)$\Rightarrow$(2) is in the previous calculations following Definition \ref{def4}. For (2)$\Rightarrow$(1), we define a partial action of $H$ on $C$ by 
		\begin{eqnarray*}
			h\triangleright \left(\begin{array}{cc}
				a & m\\
				n & x
			\end{array}\right)=\left(\begin{array}{cc}
				h\cdot_Aa & hm\\
				hn & h\cdot_Bx
			\end{array}\right).
		\end{eqnarray*}
	\end{proof}
	
	Note that if the hypothesis are satisfied, then we have that $h\cdot_A(mn)=\sum (h_{(1)}m)(h_{(2)}n)$ and
	$h\cdot_B(nm)=\sum (h_{(1)}n)(h_{(2)}m)$.
	
	\begin{proposition} Morita equivalence of partial actions is an equivalence relation.
	\end{proposition}
	\begin{proof}
		Let $A$, $A'$, $A''$ be associative partial $H$-module algebras with partial actions $\cdot$, $\cdot_1$, $\cdot_2$, respectively. Suppose that $\cdot$ is Morita equivalent to $\cdot_1$ and $\cdot_1$ is Morita equivalent to $\cdot_2$ with strict Morita contexts $(A,A',M,M')$ and $(A',A'',L',L'')$, respectively. Now, consider the linear maps
		\begin{eqnarray*}
			\tau: (M\otimes_{A'} L')\otimes_{A''} (L''\otimes_{A'}M')&\to & A\\
			m\otimes l'\otimes l''\otimes m'&\mapsto & (m((l'l'')m')),
		\end{eqnarray*}
		and
		\begin{eqnarray*}
			\sigma: (L''\otimes_{A'}M')\otimes_{A}(M\otimes_{A'} L') &\to & A''\\
			l''\otimes m' \otimes m\otimes l' &\mapsto & (l''((m'm)l')).
		\end{eqnarray*}
		Since these maps are induced by morphisms of strict Morita contexts, when we consider the natural $(A,A'')$-bimodule structure of $M\otimes_{A'} L'$ and the natural $(A'',A)$-bimodule structure of $L''\otimes_{A'}M'$, we have the strict Morita context $(A,A'',M\otimes_{A'} L',L''\otimes_{A'}M',\tau,\sigma)$. Finally, defining
		\begin{eqnarray*}
			h\triangleright \left( \begin{array}{cc}
				a & m\otimes l'\\
				l''\otimes m' & a''
			\end{array}\right)=\sum \left( \begin{array}{cc}
				h\cdot a & h_{(1)}m\otimes h_{(2)}l'\\
				h_{(1)}l''\otimes h_{(2)}m' & h\cdot_2 a''
			\end{array}\right),
		\end{eqnarray*}
		we have that the partial actions $\cdot$ and $\cdot_2$ are Morita equivalent. Since clearly Morita equivalence of partial actions is reflexive and symmetric, it is an equivalence relation.
	\end{proof}
	We recall from \cite{garciasimon} that two idempotent rings are Morita equivalent, i.e., its categories of unital and torsionfree modules are equivalent, if and only if there exists a strict Morita context where the modules are unital. We also recall that if $A$ is an idempotent partial $H$-module algebra then $\underline{A\# H}$ is also an idempotent algebra (Proposition \ref{prop.smash.product.idempotent}).
	\begin{theorem}
		Let $A$ and $B$ be idempotent partial $H$-module algebras with partial actions $\cdot_A$ and $\cdot_B$, respectively. If $\cdot_A$ is Morita equivalent to $\cdot_B$, then $\underline{A\# H}$ is Morita equivalent to $\underline{B\# H}$.
	\end{theorem}
	\begin{proof}
		Suppose that the Morita equivalence of $A$ and $B$ is given by the strict Morita context $(A,B,\,_AM_B,\,_BN_A)$. Note that $M\otimes H$ is an $A\#H-B\#H$-bimodule with structure given by:
		\begin{eqnarray*}
			(m\otimes h)(x\otimes k)&=&\sum m(h_{(1)}\cdot_B x)\otimes h_{(2)}k\\
			(a\otimes h)(m\otimes k)&=&\sum a(h_{(1)}m)\otimes h_{(2)}k.
		\end{eqnarray*}
		Analogously $N\otimes H$ is a $B\#H-A\#H$-bimodule. Now, since $\underline{A\# H}$ and $\underline{B\# H}$ are subalgebras of $A\# H$ and $B\# H$, respectively, we have that $M\otimes H$ is also an $\underline{A\# H}-\underline{B\# H}$-bimodule and $N\otimes H$ is also a $\underline{B\# H}-\underline{A\# H}$-bimodule. Since 
		\begin{eqnarray*}
			(M\otimes H)(B\# 1_H)(\underline{B\# H})&=&(M\otimes H)(B\# 1_H)(B\# H)(B\# 1_H)\\
			&\subseteq & (M\otimes H)(B^2\# H)(B\# 1_H)\\
			&\subseteq & (M\otimes H)(B\# 1_H),
		\end{eqnarray*}
		because $(M\otimes H)(B\# H)\subseteq M\otimes H$, we have that $(M\otimes H)(B\# 1_H)$ is an $\underline{A\# H}-\underline{B\# H}$-sub-bimodule of $M\otimes H$, and this sub-bimodule structure is given by
		\begin{eqnarray*}
			\sum a(h_{(1)}\cdot_A b)\# h_{(2)} \blacktriangleright \sum m(k_{(1)}\cdot_B z)\otimes k_{(2)}&=& \sum a(h_{(1)}(bm))(h_{(2)}k_{(1)}\cdot_B z)\otimes h_{(3)}k_{(2)},\\
			\sum m(k_{(1)}\cdot_B z)\otimes k_{(2)} \blacktriangleleft \sum x(h_{(1)}\cdot_B y)\# h_{(2)}&=& \sum m(k_{(1)}\cdot_Bzx)(k_{(2)}h_{(1)}\cdot_B y)\otimes k_{(3)}h_{(2)},
		\end{eqnarray*}
		for every $h, k\in H$, $a, b\in A$, $x, y\in B$, $m\in M$.
		
		Analogously, $(N\otimes H)(A\# H)$ is a $\underline{B\# H}-\underline{A\# H}$-sub-bimodule of $N\otimes H$ with structure given by	
		\begin{eqnarray*}
			\sum x(h_{(1)}\cdot_B y)\# h_{(2)} \blacktriangleright \sum n(k_{(1)}\cdot_A c)\otimes k_{(2)}&=& \sum x(h_{(1)}(yn))(h_{(2)}k_{(1)}\cdot_Ac)\otimes h_{(3)}k_{(2)},\\
			\sum n(k_{(1)}\cdot_Ac)\otimes k_{(2)} \blacktriangleleft \sum a(h_{(1)}\cdot_A b)\# h_{(2)}&=& \sum n(k_{(1)}\cdot_Aca)(k_{(2)}h_{(1)}\cdot_A b)\otimes k_{(3)}h_{(2)},
		\end{eqnarray*}
		for every $h, k\in H$, $a, b, c\in A$, $x, y, z\in B$, $n\in N$.
		
		Since $M$ and $N$ are unital bimodules, then $(M\otimes H)(B\#1_H)$ and $(N\otimes H)(A\#1_H)$ are also unital bimodules. Finally, consider the linear maps
		\begin{eqnarray*}
			\tau: (M\otimes H)B\otimes_{\underline{B\# H}} (N\otimes H)A&\to & \underline{A\# H}\\
			(\sum m(h_{(1)}\cdot_B x)\otimes h_{(2)})\otimes (\sum n(k_{(1)}\cdot_Aa )\otimes k_{(2)})&\mapsto & \sum (m(h_{(1)}(xn)))(h_{(2)}k_{(1)}\cdot_A a)\# h_{(3)}k_{(2)}, 
		\end{eqnarray*}
		and
		\begin{eqnarray*}
			\sigma: (N\otimes H)\otimes_{\underline{A\# H}}(M\otimes H)&\to & \underline{B\# H}\\
			(\sum n(k_{(1)}\cdot_Aa )\otimes k_{(2)})\otimes (\sum m(h_{(1)}\cdot_B x)&\mapsto & \sum (n(h_{(1)}(am)))(h_{(2)}k_{(1)}\cdot_Bx)\# h_{(3)}k_{(2)},
		\end{eqnarray*}
		that are well-defined because $M$ and $N$ are unital, and are also surjective because the Morita context $(A,B,M,N)$ is strict. Clearly $\tau$ and $\sigma$ are morphisms of bimodules, and they are balanced by construction. Hence $\underline{A\# H}$ is Morita equivalent to $\underline{B\# H}$.
	\end{proof}
	
	Consider $A$ a partial $H$-module algebra with s.l.u. $S=\{e_\lambda\}_{\lambda\in\Lambda}$ with symmetrical $S$-categorizable partial action. To prove that $\underline{A\# H}$ and $\underline{a(\mathcal{C}^S(A))\# H}$ are Morita equivalent, first we will use the strict Morita context $(A,a(\mathcal{C}^S(A)),\oplus_{\lambda}Ae_\lambda,\oplus_{\lambda}e_\lambda A,\tau,\sigma)$, where the elements of $\oplus_{\lambda}Ae_\lambda$ are seen as row matrices, the elements of $\oplus_{\lambda}e_\lambda A$ are seen as column matrices, the $A$-module structures are the usual and the $a(\mathcal{C}^S(A))$-module structures and the morphisms $\tau$ and $\sigma$ are given by matrix multiplication. 
	\begin{corollary} \label{corollary.morita.smash.1}
		Let $A$ be a partial $H$-module algebra with s.l.u. $S=\{e_\lambda\}_{\lambda\in\Lambda}$ with symmetrical $S$-categorizable partial action and consider the induced symmetrical partial action on $a(\mathcal{C}^S(A))$. Then $\underline{A\# H}$ and $\underline{a(\mathcal{C}^S(A))\# H}$ are Morita equivalent.
	\end{corollary}
	\begin{proof}
		We only need to consider the linear map
		\begin{eqnarray*}
			h\triangleright \left(\begin{array}{cc}
				a & (a^\lambda)_\lambda\\
				(a_\lambda)_\lambda & (a(\lambda,\alpha) )_{\lambda,\alpha}
			\end{array} \right)=\left(\begin{array}{cc}
				h\cdot a & (h\cdot a^\lambda)_\lambda\\
				(h\cdot a_\lambda)_\lambda & (h\cdot a(\lambda,\alpha) )_{\lambda,\alpha}
			\end{array} \right),
		\end{eqnarray*}
		where $a^\lambda\in  Ae_\lambda$, $a_\lambda\in e_\lambda A$ and $a(\lambda,\alpha)\in e_\alpha Ae_\lambda$. Since the partial action on $A$ is symmetrical $S$-categorical, this map is well-defined and it is straightforward that it determines a partial action on $C=\left(\begin{array}{cc}
			A & \oplus_{\lambda}Ae_\lambda\\
			\oplus_{\lambda}e_\lambda A & a(\mathcal{C}^S(A))
		\end{array} \right)$.
	\end{proof}

	In \cite{alves2}, the authors proved that given a partial $H$-module $\Bbbk$-category $\mathcal{C}$, we have that $\underline{a(\mathcal{C})\# H}$ and $a(\underline{\mathcal{C}\# H})$ are isomorphic, where the partial action on $a(\mathcal{C})$ is induced by the partial action on $\mathcal{C}$. Considering this fact,  Theorem \ref{cac} and the previous results of this section, we have the following result.
	
	\begin{corollary} \label{corollary.morita.smash.2}
		Let $A$ be a partial $H$-module algebra with s.l.u. $S$ with symmetrical $S$-categorizable partial action. Then the category of the unital $\underline{A\# H}$ modules is equivalent to the category of the $\underline{\mathcal{C}^S(A)\# H}$ modules.
	\end{corollary}

	In \cite{abadie}, it was proved that every partial group action is Morita equivalent to a globalizable partial action. Since we know that every symmetrical partial Hopf action has a globalization, here we will show that every symmetrical partial action is Morita equivalent to a partial action the algebra has trivial right (or left) annihilator. 
	
	Let $A$ be an associative algebra and consider $B=A/r(A)$ and $C=A/l(A)$. In \cite{garciasimon}, Garc\'ia and S\'imon highlighted the fact that $A$, $B$ and $C$ are Morita equivalent algebras. Explicitly, denoting the equivalence class of $a\in A$ in $B$ and in $C$ by $[a]_r$ and $[a]_l$, respectively, we have that the mappings $a[b]_r=ab$ and $[a]_rb=[ab]_r$ provide a structure of right $B$-module on $A$ and a structure of right $A$-module on $B$. 
	
	Hence, considering the natural structures of left $A$-module of $A$ and of left $B$-module of $B$, we have that $(A,B,\,_AA_B,\,_BB_A)$ determines a strict Morita context where the bimodules morphisms of the  context are given by the same mappings of the module structures.
	
	Analogously, $(C,A,\,_AC_C,\,_CA_A)$ determines a strict Morita context between $C$ and $A$.
	
	For the partial actions we will assume that $A$ is a partial $H$-module algebra with symmetrical partial action. Note that if $H$ has bijective antipode, then for every $x\in r(A)$, $a\in A$ and $h\in H$, we have that
	\begin{eqnarray*}
		a(h\cdot x)&=&\sum (h_{(2)}S^{-1}(h_{(1)})\cdot a)(h_{(3)}\cdot x)\\
		&=&\sum h_{(2)}\cdot((S^{-1}(h_{(1)})\cdot a)x)\\
		&=&0,
	\end{eqnarray*} 
	i.e., $h\cdot x\in r(A)$. Hence we can define a partial $H$-module algebra structure in $B$ given by $h\cdot [a]_r=[h\cdot a]_r$.
	
	For the algebra $C$, even if $H$ does not have bijective antipode, for every $x\in l(A)$, $a\in A$ and $h\in H$, we have that
	\begin{eqnarray*}
		(h\cdot x)a&=&\sum (h_{(1)}\cdot x)(h_{(2)}S(h_{(3)})\cdot a)\\
		&=&\sum (h_{(1)}\cdot (x(S(h_{(2)})\cdot a))\\
		&=&0,
	\end{eqnarray*}
	i.e., $h\cdot x\in l(A)$. Hence we can define a partial $H$-module algebra structure in $C$ given by $h\cdot [a]_l=[h\cdot a]_l$.
	
	Clearly these partial actions are symmetrical and the bimodules $_AA_B$ and $_BB_A$ (analogously $_AC_C$ and $_CA_A$) satisfy the hypotheses of Proposition 
	\ref{Proposition.Morita.equivalence.partial.bimodules}, i.e., the symmetrical partial action of $H$ on $A$ is Morita equivalent to the induced symmetrical partial action of $H$ on $C$ and to the induced (if $H$ has bijective antipode) partial action on $B$.
	
	We recall that if $A$ is idempotent, then $r(A/r(A))=0=l(A/l(A))$, hence $l(C)=0$. Therefore, by the previous discussion, we obtain the following result. 
	
	\begin{theorem} \label{thm.morita.equivalent.partial.action}
		Every symmetrical partial Hopf action on an idempotent algebra is Morita equivalent to a partial action on an idempotent algebra with trivial right (or left) annihilator. 
	\end{theorem}
	
	Considering the results presented in \cite{abadie}, the authors constructed a canonical Morita globalization for a regular partial action and proved that whenever two regular partial $G$-actions are Morita equivalent, the global actions of its canonical Morita globalizations are also Morita equivalent. Here, we will prove that a similar result holds for partial $H$-actions.
	
	\begin{theorem}
		\label{thm.morita.equivalence.globalization}
		Let $A$ and $B$ be two symmetrical partial $H$-module algebras with Morita equivalent partial actions. Then the global actions of its standard globalizations are also Morita equivalent.
	\end{theorem}
	\begin{proof}
		Consider the strict Morita context $(A,B,\,_AM_B,\,_BN_A,\tau,\sigma)$ from the Morita equivalence of the partial $H$-actions on $A$ and on $B$. We will denote the two partial actions on $A$ and $B$ by a dot  ``$\cdot$'', the elements of $A$ by $a,b,\ldots$, the elements of $B$ by $x,y,\ldots$ and the mappings of the Morita context by $\tau(m,n)=(mn)$ and $\sigma(n,m)=(nm)$, for every $m\in M$, $n\in N$. We will also denote the standard globalizations by $A'=H\triangleright \varphi_A(A)$ and $B'=H\triangleright \varphi_B(B)$ (note that here we will use the notation $\triangleright$ to represent the actions of both globalizations, because it will not give rise to any confusion in the calculations).
		
		The Morita equivalence of the partial actions yields two linear maps ${H\otimes M\to M}$ and ${H\otimes N\to N}$; then the Adjoint Isomorphism provides the linear maps $\varphi_M:M\to \HOM(H,M)$ and $\varphi_N: N\to \HOM(H,N)$ defined by $\varphi_M(m)(h)=hm$ and $\varphi_N(n)(h)=hn$, respectively, for all $m\in M$, $n\in N$, $h\in H$.
		
		Also, as for partial $H$-module algebras, the vector spaces $\HOM(H,M)$ and $\HOM(H,N)$ are $H$-module algebras with $H$-module structures given by $(k\triangleright f)(h)=f(hk)$ (again, the same notation for the action), particularly, $(k\triangleright \varphi_M(m))(h)=(hk)m$. 
		
		Now, we will consider the vector spaces $M'=H\triangleright \varphi_M(M)$ and $N'=H\triangleright \varphi_N(N)$ and prove that $M'$ is an $(A',B')$-bimodule and $N'$ is a $(B',A')$-bimodule. It is enough to prove those claims for $M'$, and we define the $A'$-module structure on $M'$ by 
		$$(h\triangleright \varphi_A(a)\blacktriangleright k\triangleright \varphi_M(m))(l)=\sum (h\triangleright\varphi_A(a)(l_{(1)}))(k\triangleright\varphi_M(m)(l_{(2)}))=\sum (l_{(1)}h\cdot a)((l_{(2)}k)m).$$
		Note that if $\sum h^i\triangleright \varphi_A(a^i)=0$, then $\sum lh^i\cdot a=0$ for every $l\in H$, hence $\sum h^i\triangleright \varphi_A(a^i) \blacktriangleright k\triangleright \varphi_M(m)=0$. Also, we have that 
		\begin{eqnarray*}
			(h\triangleright \varphi_A(a)\blacktriangleright k\triangleright \varphi_M(m))(l)&=&\sum (l_{(1)}h\cdot a)((l_{(2)}k)m)\\
			&=&\sum lh_{(1)}(a((S(h_{(2)})k)m))\\
			&=& (\sum h_{(1)}\triangleright\varphi_M(a((S(h_{(2)})k)m)))(l),
		\end{eqnarray*}
		i.e., $h\triangleright \varphi_A(a)\blacktriangleright k\triangleright \varphi_M(m) \in M'$, and we proved that $\blacktriangleright$ is well-defined. Moreover, since $\blacktriangleright$ works like the "convolution product" together with the $A$-module structure of $M$, it follows that $\blacktriangleright$ is actually an action.
		
		The right $B'$-module structure is given, analogously, by 
		$$( k\triangleright \varphi_M(m)\blacktriangleleft (h\triangleright \varphi_B(x)))(l)=\sum ((l_{(1)}k)m )(l_{(2)}h\cdot x).$$

		Now, instead of using the notations $\varphi_A$, $\varphi_B$, $\varphi_M$ and $\varphi_N$ we will only write $\varphi$.  We can also use the previous calculation to show that the mappings
		\begin{eqnarray*}
			\tau'(h\triangleright\varphi(m),k\triangleright\varphi(n))(l)&=&\sum \tau((l_{(1)}h)m,(l_{(2)}k)n),\\
			\sigma'(k\triangleright\varphi(n),h\triangleright\varphi(m))(l)&=&\sum \sigma((l_{(2)}k)n,(l_{(1)}h)m),
		\end{eqnarray*}
		define linear maps
		\begin{eqnarray*}
			\tau': M'\otimes N'&\to& A'\\
			\sigma': N'\otimes M' &\to& B',
		\end{eqnarray*}
		that are well-defined balanced bimodule maps. Additionally, they are surjective because $\tau$ and $\sigma$ are surjective and $M$ and $N$ are unital bimodules. In other words, $(A',B',M',N',\tau',\sigma')$ is a strict Morita context. Finally, with the considered $H$-actions on $M'$ and $N'$, we can easily see that the $H$-actions on $A'$ and on $B'$ are Morita equivalent.  
	\end{proof}
	
	\subsection{Morita Equivalence of Partial $\Bbbk G$-Actions and Morita Equivalence of partial $G$-actions}
	
	In \cite{abadie}, Abadie et al. proved that the global actions of the globalizations of two Morita equivalent partial group actions are also Morita equivalent. In the previous section we proved that the actions of the minimal globalizations of two Morita equivalent partial actions are also Morita equivalent; we also know from \cite{alves1} that every globalization of a partial $\Bbbk G$-action on a unital algebra is minimal. We will show that even if we consider nonunital algebras, under some assumptions, this result also holds.
	
	\begin{lemma}
		Let $A$ be an associative partial $\Bbbk G$-module algebra, where $G$ is a group, and let $(B,\theta)$ be a globalization. Then:
		\begin{enumerate}[\normalfont(1)]
			\item if $A$ is unital, then $B$ has local units;
			\item if $A$ has local units, then $B$ has local units;
			\item if $A$ is $s$-unital, then $B$ is $s$-unital;
			\item if $A$ is idempotent, then $B$ is idempotent;
			\item if $r(A)=0$, then $r(B)=0$ if and only if the globalization is minimal.
		\end{enumerate}
	\end{lemma}
	\begin{proof}
		We will only prove item (3), since items (1) and (2) are analogous and we already proved that items (4) and (5) hold for any Hopf algebra. 
		Assume that $A$ is $s$-unital. For every $h\triangleright \theta(a)\in B$, there exists $x\in A$ such that $xa=ax=a$, hence
		\[
		(h\triangleright \theta(x))(h\triangleright \theta(a))= h\triangleright \theta(a)
		=(h\triangleright \theta(a))(h\triangleright \theta(x))
		\]
		and therefore $B$ is $s$-unital. 
	\end{proof}
	
	\begin{corollary} Let $G$ be a group,
		$A$ be an associative partial $\Bbbk G$-module algebra and $(B,\theta)$ be a globalization. If $A$ is $s$-unital then the globalization is minimal.	
	\end{corollary}
	Now, we will relate
	the definitions of Morita equivalence of partial $\Bbbk G$-actions and of Morita equivalence of partial $G$-actions, as was done for partial actions.
	We begin by recalling the concept of a product partial action, where the intersections of domains are substituted by their products.  
	\begin{definition}[\cite{abadie}] A product partial action $\alpha$ of a group $G$ on an algebra $A$ consists of a family of two-sided ideals $D_g$ in $A$, $g\in G$, and algebra isomorphisms $\alpha_g: D_{g^{-1}} \to D_g$, such that:
		\begin{enumerate}[\normalfont(1)]
			\item $\alpha_1$ is the trivial isomorphism $A\to A$;
			\item $D_g^2 = D_g, D_g \cdot D_h = D_h \cdot D_g$;
			\item $\alpha_g(D_{g^{-1}} \cdot D_h)\subseteq D_g \cdot D_{gh}$;
			\item $\alpha_g(\alpha_h(x))=\alpha_{gh}(x)$, for any $x\in D_{h^{-1}}\cdot D_{(gh)^{-1}}$.
		\end{enumerate} 
	\end{definition}
	Every regular partial action is also a product partial action. 
	\begin{definition}[\cite{abadie}] Let 
		\begin{eqnarray*}
			\alpha=\{\alpha_g: D_{g^{-1}}\to D_g\}\,\,\,and\,\,\,\alpha'=\{\alpha'_g:D'_{g^{-1}}\to D'_g\}
		\end{eqnarray*}
		be regular partial actions of $G$ on algebras $A$ and $A'$, respectively. We say that $\alpha$ is Morita equivalent to $\alpha'$ if:
		\begin{enumerate}[\normalfont(1)]
			\item There exists a strict Morita context $(A,A',\,_AM_{A'},\,_{A'}M'_A,\tau,\sigma)$, where $M$ and $M'$ are unital bimodules such that $M'D_gM=D'_g$ for any $g\in G$;
			\item There exists a product partial action $\theta=\{\theta_g: E_{g^{-1}}\to E_g \}$ of $G$ on $C$, where $C$ is the context algebra $C=\left(\begin{array}{cc}
				A & M\\
				M' & A'
			\end{array}\right)$, such that $\theta$ restricted to $\left(\begin{array}{cc}
				A & 0\\
				0 & 0
			\end{array}\right)$ and $\left(\begin{array}{cc}
				0 & 0\\
				0 & A'
			\end{array}\right)$, is $\alpha$ and $\alpha'$, respectively.
		\end{enumerate}
	\end{definition}
	Recall from \cite{abadie} that "$\theta$ restrict to $\left(\begin{array}{cc}
		A & 0\\
		0 & 0
	\end{array}\right)$ is $\alpha$" means that, for every $g\in G$
	\begin{center}
		$E_g\cap$ $\left(\begin{array}{cc}
			A & 0\\
			0 & 0
		\end{array}\right)= \left(\begin{array}{cc}
			D_g & 0\\
			0 & 0
		\end{array}\right)$
	\end{center}
	and 
	\begin{center}
		$\theta_g \left(\begin{array}{cc}
			a & 0\\
			0 & 0
		\end{array}\right) = \left(\begin{array}{cc}
			\alpha_g(a) & 0\\
			0 & 0
		\end{array}\right),\,\,\, \forall a\in D_{g^{-1}}$.
	\end{center}
	
	A consequence of Proposition \ref{prop1} is that whenever $A$ is idempotent, $\alpha$ is a regular partial action of $G$ on $A$ and there exist $\alpha$-projections $p_g: A\to D_g$, then the linear map $\cdot: \Bbbk G\otimes A\to A$ defined by $g\cdot a=\alpha_g(p_{g^{-1}}(a))$ is a symmetrical partial action.
	
	\begin{lemma}\label{lem1}  Let 
		\begin{eqnarray*}
			\alpha=\{\alpha_g: D_{g^{-1}}\to D_g\}\,\,\,and\,\,\,\alpha'=\{\alpha'_g:D'_{g^{-1}}\to D'_g\}
		\end{eqnarray*}
		be Morita equivalent regular partial actions of $G$
		on idempotent algebras $A$ and $A'$,
		respectively, with product partial action on the context algebra $C$ given by $\theta=\{\theta_g:E_{g^{-1}}\to E_g\}$. Suppose that there exist algebra morphisms $p_g: A\to D_g$, $p'_g: A'\to D'_g$ and $P_g: C\to E_g$ that are $\alpha$-projections, $\alpha'$-projections and $\theta$-projections, respectively, and that every $P_g$ restricted to the copy of $A$ and $A'$ is $p_g$ and $p'_g$, respectively, i.e.
		\begin{eqnarray*}
			P_g\left(\begin{array}{cc}
				a & 0\\
				0 & 0
			\end{array}\right)&=&\left(\begin{array}{cc}
				p_g(a) & 0\\
				0 & 0
			\end{array}\right),\\
			P_g\left(\begin{array}{cc}
				0 & 0\\
				0 & a'
			\end{array}\right)&=&\left(\begin{array}{cc}
				0 & 0\\
				0 & p'_g(a')
			\end{array}\right).
		\end{eqnarray*}
		Then the induced partial actions of $\Bbbk G$ on $A$ and $A'$ are also Morita equivalent.
	\end{lemma}
	\begin{proof}	
		We only need to prove that the induced partial actions of $\Bbbk G$ on the context algebra restricted to $\left(\begin{array}{cc}
			A & 0\\
			0 & 0
		\end{array}\right)$ and $\left(\begin{array}{cc}
			0 & 0\\
			0 & A'
		\end{array}\right)$ are, respectively, the induced partial actions of $\Bbbk G$ on $A$ and $A'$. In fact, in the first case, we have that
		\begin{eqnarray*}
			g\cdot \left(\begin{array}{cc}
				a & 0\\
				0 & 0
			\end{array}\right)&=&\theta_gP_{g^{-1}}\left(\begin{array}{cc}
				a & 0\\
				0 & 0
			\end{array}\right)\\
			&=&\theta_g\left(\begin{array}{cc}
				p_{g^{-1}}(a) & 0\\
				0 & 0
			\end{array}\right)\\
			&=&\left(\begin{array}{cc}
				\alpha_g(p_{g^{-1}}(a)) & 0\\
				0 & 0
			\end{array}\right)\\
			&=&\left(\begin{array}{cc}
				g\cdot a & 0\\
				0 & 0
			\end{array}\right).
		\end{eqnarray*}
		Analogously, the induced partial action of $\Bbbk G$ on the context algebra restricted to $\left(\begin{array}{cc}
			0 & 0\\
			0 & A'
		\end{array}\right)$ is the induced partial action of $\Bbbk G$ on $A'$.
	\end{proof}
	
	\begin{theorem}\label{lem2} Let $A$ and $A'$ be idempotent partial $\Bbbk G$-module algebras with Morita equivalent symmetrical partial actions. Then the induced partial actions of $G$ on $A$ and $A'$ are Morita equivalent.
	\end{theorem}
	\begin{proof}
		Since $A$ is idempotent, we know that the induced partial action of $G$ on $C$ is regular, hence by \cite{abadie} it is a product partial action.
		
		Now, consider the Morita context $(A,A',M,M')$ given by the Morita equivalence of the symmetrical partial actions of $\Bbbk G$ on $A$ and on $A'$. In order to show that the restriction properties holds, note that $D_g=g\cdot g^{-1}\cdot A$ and $D'_g=g\cdot g^{-1}\cdot A'$, hence 
		$$g\cdot g^{-1}\cdot M=g\cdot g^{-1}\cdot AM=(g\cdot g^{-1}\cdot A)M=M(g\cdot g^{-1}\cdot A'),$$ $$g\cdot g^{-1}\cdot M'=g\cdot g^{-1}\cdot A'M'=(g\cdot g^{-1}\cdot A')M'=M'(g\cdot g^{-1}\cdot A).$$
		Then, since $MD'_gM'$ is generated by elements of the form $m(g\cdot g^{-1}\cdot a)m'$, we have that
		\begin{eqnarray*}
			m(g\cdot g^{-1}\cdot a)m'&=&(g\cdot(g^{-1}m)(g^{-1}\cdot a))m'\\
			&=&g\cdot(g^{-1}m)(g^{-1}\cdot a)(g^{-1} m')\\
			&=& g\cdot g^{-1}\cdot (\tau(m,am')) \in D_g.
		\end{eqnarray*}
		Conversely, for every $g\cdot g^{-1}\cdot a\in D_g$, there exist $m_1, \ldots, m_r\in M$ and $n_1, \ldots, n_r \in M'$ such that $a=\sum_i \tau(m_i,n_i)$; since $M$ is a unital right $A'$-module we may also write each $m_i$ as a linear combination $m_i=\sum_j m_{ij}a'_{ij}$, where $m_{i,j} \in M$, $a'_{ij} \in A$´, and therefore
		\begin{eqnarray*}
			g\cdot g^{-1} \cdot a&=& \sum_{i,j} g\cdot g^{-1}\cdot \tau(m_{ij}b_{ij},n_i)\\
			&=&\sum_{i,j} \tau(g(g^{-1}m_{ij})(g\cdot g^{-1}\cdot b_{ij}),g(g^{-1}n_i)) \in MD'_gM',
		\end{eqnarray*}
		proving that $MD'_gM'=D_g$. Analogously, we have that $M'D_gM= D'_g$. Hence
		\begin{eqnarray*}
			E_g&=& C(g\cdot g^{-1}\cdot C)\\
			&=& \left(\begin{array}{cc}
				A & M\\
				M' & A'
			\end{array}\right)\left(\begin{array}{cc}
				g\cdot g^{-1}\cdot A & g\cdot g^{-1}\cdot M\\
				g\cdot g^{-1}\cdot M' & g\cdot g^{-1}\cdot A'
			\end{array}\right)\\
			&=&\left(\begin{array}{cc}
				A & M\\
				M' & A'
			\end{array}\right)\left(\begin{array}{cc}
				D_g & D_gM\\
				D'_g M' & D'_g
			\end{array}\right)\\
			&=&\left(\begin{array}{cc}
				D_g & D_gM\\
				D'_g M' & D'_g
			\end{array}\right),
		\end{eqnarray*} 
		then
		\begin{eqnarray*}
			E_g\cap \left(\begin{array}{cc}
				A & 0\\
				0 & 0
			\end{array}\right)&=&\left(\begin{array}{cc}
				D_g & 0\\
				0 & 0
			\end{array}\right),\\
			E_g\cap \left(\begin{array}{cc}
				0 & 0\\
				0 & A'
			\end{array}\right)&=&\left(\begin{array}{cc}
				0 & 0\\
				0 & D'_g
			\end{array}\right),
		\end{eqnarray*}
		and 
		\begin{eqnarray*}
			\theta_g\left(\begin{array}{cc}
				g^{-1}\cdot g\cdot a & 0\\
				0 & 0
			\end{array}\right)&=&g\cdot g^{-1}\cdot g\cdot\left(\begin{array}{cc}
				g^{-1}\cdot g\cdot a & 0\\
				0 & 0
			\end{array}\right)\\
			&=&\left(\begin{array}{cc}
				g\cdot g^{-1}\cdot g\cdot g^{-1}\cdot g\cdot a & 0\\
				0 & 0
			\end{array}\right)\\
			&=&\left(\begin{array}{cc}
				g\cdot g^{-1}\cdot a & 0\\
				0 & 0
			\end{array}\right)\\
			&=&\left(\begin{array}{cc}
				\alpha_g(g^{-1}\cdot g\cdot a) & 0\\
				0 & 0
			\end{array}\right).
		\end{eqnarray*}
		Hence the restriction of the induced partial action of $G$ on $C$ to $\left(\begin{array}{cc}
			A & 0\\
			0 & 0
		\end{array}\right)$ is the induced partial action of $G$ on $A$. Analogously, the induced partial action of $G$ on $C$ restrict to $\left(\begin{array}{cc}
			0 & 0\\
			0 & A'
		\end{array}\right)$ is the induced partial action of $G$ on $A'$.
	\end{proof}

	\bibliographystyle{plain}
	\bibliography{Morita_equivalence_and_globalization_for_partial_Hopf_actions_on_nonunital_algebras}

\begin{thebibliography}{10}

\bibitem{abadie}
F.~Abadie, M.~Dokuchaev, R.~Exel, and J.~J. Simón.
\newblock {M}orita equivalence of partial group actions and globalization.
\newblock {\em Trans. Amer. Math. Soc.}, 368(7):4957--4992, 2016.

\bibitem{alves2}
E.~R. Alvares, M.~M.~S. Alves, and E.~Batista.
\newblock Partial {H}opf module categories.
\newblock {\em J. Pure Appl. Algebra}, 2012.

\bibitem{AAR17}
Edson~Ribeiro Alvares, Marcelo~Muniz Alves, and Maria~Julia Redondo.
\newblock Cohomology of partial smash products.
\newblock {\em J. Algebra}, 482:204 -- 223, 2017.

\bibitem{alves1}
M.~M.~S. Alves and E.~Batista.
\newblock Enveloping actions for partial {H}opf actions.
\newblock {\em Comm. Algebra}, 38(8):2872--2902, 2010.

\bibitem{alves3}
M.~M.~S. Alves and E.~Batista.
\newblock Globalization theorems for partial {H}opf (co)actions, and some of
  their applications.
\newblock {\em Contemp. Math.}, 537:13--30, 2011.

\bibitem{alves4}
M.~M.~S. Alves, E.~Batista, and J.~Vercruysse.
\newblock Partial representations of {H}opf algebras.
\newblock {\em J. Algebra}, 426:137--187, 2015.

\bibitem{dilations}
M.~M.~S. Alves, E.~Batista, and J.~Vercruysse.
\newblock Dilations of partial representations of {H}opf algebras.
\newblock {\em J. Lond. Math. Soc.}, 100(2):273--300, 2019.

\bibitem{partial.coreps}
M.M.S. Alves, E.~Batista, F.~Castro, and J.~Vercruysse G.~Quadros.
\newblock Partial corepresentations of {H}opf algebras.
\newblock {\em J. Algebra}, 577:74--135, 2021.

\bibitem{ABFFM21}
Danielle Azevedo, Eliezer Batista, Graziela Fonseca, Eneilson Fontes, and
  Grasiela Martini.
\newblock Partial (co)actions of multiplier {H}opf algebras: {M}orita and
  {G}alois theories.
\newblock {\em J. Algebra Appl.}, 20(08):215--142, 2021.

\bibitem{caenepeel}
S.~Caenepeel and K.~Janssen.
\newblock Partial (co)actions of {H}opf algebra and partial {H}opf-{G}alois
  theory.
\newblock {\em Comm. Algebra}, 36(8):2923--2946, 2008.

\bibitem{partial.weak.hopf}
Felipe Castro, Antonio Paques, Glauber Quadros, and Alveri Sant'Ana.
\newblock Partial actions of weak {H}opf algebras: Smash product, globalization
  and {M}orita theory.
\newblock {\em J.Pure Appl. Algebra}, 219(12):5511--5538, 2015.

\bibitem{rafael}
R.~Cavalheiro and A.~Sant'Ana.
\newblock On partial {H}-radicals of {J}acobson type.
\newblock {\em São Paulo J. Math. Sci.}, 10:140--163, 2016.

\bibitem{cibils-solotar-galois}
C.~Cibils and A.~Solotar.
\newblock {G}alois coverings, {M}orita equivalence and smash extensions of
  categories over a field.
\newblock {\em Doc. Math.}, 11:143–159, 2006.

\bibitem{ion}
S.~Dascalescu, B.~Ion, C.~Nastasescu, and J.~R. Montes.
\newblock Group gradings on full matrix rings.
\newblock {\em J. Algebra}, 200:709--728, 1999.

\bibitem{Dokuchaev.survey}
M.~Dokuchaev.
\newblock Recent developments around partial actions.
\newblock {\em São Paulo J. Math. Sci.}, 13:195--247, 2019.

\bibitem{global-s-unital}
M.~Dokuchaev, A.~Del~Río, and J.~J. Simón.
\newblock Globalizations of partial actions on nonunital rings.
\newblock {\em Proc. Amer. Math. Soc.}, 135:343--352, 2007.

\bibitem{dok}
M.~Dokuchaev and R.~Exel.
\newblock Associativity of crossed products by partial actions, enveloping
  actions and partial representations.
\newblock {\em Trans. Amer. Math. Soc.}, 357(5):1931--1952, 2005.

\bibitem{des}
M.~Dokuchaev, R.~Exel, and J.~J. Simón.
\newblock Crossed products by twisted partial actions and graded algebras.
\newblock {\em J. Algebra}, 320:3278--3310, 2008.

\bibitem{ferrero}
Michael Dokuchaev, Miguel Ferrero, and Antonio Paques.
\newblock Partial actions and {G}alois theory.
\newblock {\em J. Pure Appl. Algebra}, 208(1):77--87, 2007.

\bibitem{Circle.actions}
Ruy Exel.
\newblock Circle actions on {$C^*$}-algebras, partial automorphisms and a
  generalized {P}imsner-{V}oiculescu exact sequence.
\newblock {\em J. Funct. Anal.}, 122:361--401, 1994.

\bibitem{globalization-fonseca-fontes-martini}
G.~Fonseca, E.~Fontes, and G.~Martini.
\newblock Multiplier {H}opf algebras: Globalization for partial actions.
\newblock 30(3):539--565, 2020.

\bibitem{garciasimon}
J.~L. García and J.~J. Simón.
\newblock {M}orita equivalence for idempotent rings.
\newblock {\em J. Pure Appl. Algebra}, 76:39--56, 1991.

\bibitem{anca1}
A.~Stanescu and D.~Stefan.
\newblock Cleft comodule categories.
\newblock {\em Comm. Algebra}, 41(5):1697--1726, 2013.

\bibitem{tominaga}
H.~Tominaga.
\newblock On s-unital rings.
\newblock {\em Math. J. Okayama Univ.}, 18:117--134, 1976.

\bibitem{weibel}
C.~A. Weibel.
\newblock An introduction to homological algebra.
\newblock 1994.

\bibitem{ahnmarki}
P.~N. Ánh and L.~Márki.
\newblock Morita equivalence for rings without identity.
\newblock {\em Tsukuba J. Math}, 11:1--16, 1987.

\end{thebibliography}

\end{document}